\documentclass[11pt]{article} 

\usepackage[utf8]{inputenc} 
\usepackage[T1]{fontenc}
\usepackage{lmodern}

\usepackage{geometry} 
\usepackage{microtype}

\usepackage{graphicx} 

\usepackage{booktabs} 
\usepackage{array} 
\usepackage{paralist} 
\usepackage{verbatim} 
\usepackage{subfig} 
\usepackage{lastpage} 
\usepackage{extramarks} 
\usepackage{graphicx} 
\usepackage{lipsum} 
\usepackage{xcolor}
\usepackage{empheq}
\usepackage{mathtools}
\usepackage{amssymb}
\usepackage{hyperref}
\usepackage{amsthm}
\usepackage{cleveref}
\usepackage{stmaryrd}
\usepackage{calc}
\usepackage{mdframed}
\usepackage{environ}
\usepackage{xfrac}

\usepackage[backend=biber]{biblatex}
\newtheorem{theorem}{Theorem}

\newtheorem{remark}{Remark}
	\newtheorem{prop}{Proposition}

\DeclarePairedDelimiter{\norm}{\lVert}{\rVert}
\DeclarePairedDelimiter{\abs}{\lvert}{\rvert}
\DeclarePairedDelimiter{\dbracket}{\llbracket}{\rrbracket}

\newcommand{\NN}{\mathbb{N}}
\newcommand{\RR}{\mathbb{R}}
\newcommand{\EE}{\mathbb{E}}

\newcommand{\PP}{\mathbb{P}}

\newcommand{\calX}{\mathcal{X}}
\newcommand{\calY}{\mathcal{Y}}

\newcommand{\calL}{\mathcal{L}}
\newcommand{\calS}{\mathcal{S}}
\newcommand{\calN}{\mathcal{N}}
\newcommand{\calR}{\mathcal{R}}
\newcommand{\calF}{\mathcal{F}}
\newcommand{\calV}{\mathcal{V}}
\newcommand{\calU}{\mathcal{U}}
\newcommand{\calM}{\mathcal{M}}
\newcommand{\calE}{\mathcal{E}}

\newcommand{\eqdef}{\coloneqq}

\DeclareMathOperator{\Var}{Var}
\DeclareMathOperator{\Cov}{Cov}
\DeclareMathOperator*{\argmin}{arg\,min}

\newcommand\numberthis{\addtocounter{equation}{1}\tag{\theequation}}
\newmdenv[topline=false,rightline=false,bottomline=false,nobreak=false]{proofaside}
\newcommand\oversetclap[2]{\overset{\mathclap{#1}}{#2}}

\NewEnviron{meqn}{%
    \begin{align*}
    \BODY
    \numberthis
    \end{align*}
}

\usepackage{fancyhdr} 
\usepackage{sectsty}
\allsectionsfont{\sffamily\mdseries\upshape} 

\usepackage[nottoc,notlof,notlot]{tocbibind} 
\usepackage[titles,subfigure]{tocloft} 

\title{Asymptotics of Cross-Validation}
\author{Morgane Austern\\{\small Microsoft Research} \and Wenda Zhou\\ {\small Columbia University}}
\date{} 

\begin{document}
\maketitle

\begin{abstract}
    \noindent Cross validation is a central tool in evaluating the performance of machine learning and statistical models.
    However, despite its ubiquitous role, its theoretical properties are still not well understood.
    We study the asymptotic properties of the cross validated-risk for a large class of models.
    Under stability conditions, we establish a central limit theorem and Berry-Esseen bounds, which enable us to compute asymptotically accurate confidence intervals.
    Using our results, we paint a big picture for the statistical speed-up of  cross validation compared to a train-test split procedure.
    A corollary of our results is that parametric M-estimators (or empirical risk minimizers) benefit from the ``full'' speed-up when performing cross-validation under the training loss.
    In other common cases, such as when the training is performed using a surrogate loss
    or a regularizer, we show
    that the behavior of the cross-validated risk is complex with a variance reduction which may be smaller or larger than the ``full'' speed-up, depending on the model and the underlying distribution.
    We allow the number of folds $K_n$ to grow with the number of observations at any rate. 
\end{abstract}

\newpage
\tableofcontents
\newpage

\section{Introduction}
Let $(X_i)$ be a sequence of independent and identically distributed (i.i.d) observations.
Consider a statistical model that given $n$ data points computes an estimator $f_n(X_1,\dots,X_n)$.
For a loss function $\calL$, our goal is to estimate $\EE \Big[\calL(\tilde{X}, f_n(X_1,\dots,X_n))\Big]$, the expected risk of the estimator on a new observation. 
It is well-known that if we approximate it by the empirical loss $\frac{1}{n}\sum_{i\le n}\mathcal{L}\big(f_n(X_1,\dots,X_n), X_i\big)$, we run the risk of systematically underestimating it \autocite{book_eosl}.
To avoid such issues, the most common strategy is that of data splitting.

In a data splitting procedure, the dataset is separated into a training set, and a testing set.
The model will be trained on the first one and evaluated on the second one.
Formally, if we let $m_n$ denote the size of the training set, we estimate the expected risk by:
\begin{equation}
    \hat{R}_{\text{split}} \eqdef \frac{1}{n - m_n}  \sum_{i = m_n + 1}^n \calL(X_i, f_{m_n}(X_1, \dotsc, X_{m_n})).
\end{equation}
The performance of this estimator depends on two distinct sources of error.
The first one is a bias induced by the smaller size of the training set which implies that $\hat{R}_{\text{split}}$  estimates the expected risk of the estimator $f_{m_n} (X_1,\dots,X_{m_n})$ instead of the original one $f_n$.
The magnitude of this bias depends on the model itself and the size of the training set $m_n$.
The second source of error is statistical and is caused by the randomness of the empirical average $\hat{R}_{\text{split}}$.
The magnitude of this effect decreases as the size of the testing set increases.
Thus, we observe a trade-off: as we increase the size of the training set, we reduce the bias but increase the randomness (and hence the variance) of $\hat{R}_s$.
This trade-off is even more present for high-dimensional models for which the bias may be asymptotically large.

To compensate for this issue, a second method has been proposed: $K$-fold cross validation.
$K$-fold cross-validation proceeds to split the dataset into $K$ blocks $X^n_{B^n_1}, \dotsc, X^n_{B_K}$
of the same size.
The procedure successively omits the $i$th block, training the model on the remaining blocks, and evaluating the risk on the left-out block.
This gives us $K$ different models with $K$ different estimated risks.
The cross validated risk $\hat{R}_{\rm{cv}}$ is the average of those empirical risks.
Once again, there are two sources of error.
The first one is a bias due to the smaller size the training sets, which decreases with the number of folds.
The second one is due to the randomness of the estimator $\hat{R}_{\rm{cv}}$.
How large is the second effect and how much it depends on the model itself and the number of folds is largely an open question. 
This question is made more complicated by the dependence between the folds. 

In this paper we study the asymptotic properties of the cross-validated risk.
More specifically, we study its speed of convergence to the expected risk and compare it to the data splitting estimator $\hat{R}_{\text{split}}$. 
This question is only made more important by the popularity of the cross validation method.
In \cref{sec:main-results}, we prove, under stability conditions, central limit theorems and Berry-Esseen bounds for the cross-validated risk as well as the split risk.
We allow the number of folds $K_n$ to grow with the number of observations $n$ at any rate.
~By obtaining the limiting distribution of both the cross-validated risk and the split risk we can exactly quantify the reduction in variance obtained by the cross-validated procedure.
In some cases -- such as parametric empirical risk minimizers trained and evaluated with the same loss -- we obtain the full speed-up, meaning that the cross-validated risk enjoys a $K_n$-fold reduction in variance compared to its split counterpart.
However, in other cases -- such as when using surrogate losses or regularizers -- the effect of cross-validation varies significantly depending on the class of models and the distribution of the data on which it is trained. We present examples for which the variance reduction is slower than ``full speed-up'', as well as cases for which this is faster. 
Such phenomenons are illustrated in the parametric case for which we develop a full theory in \cref{sec:parametric}.
In addition, in \cref{sec:estimation-variance} we give estimators for the asymptotic variances and quantify their speed of convergence.
This allows one to draw asymptotically accurate confidence intervals.
The proofs are based on an adaptation of the Stein method for central limit theorems.

\subsection{Related work}

Cross-validation is probably the most widely used method for risk estimation today in machine learning and
statistics \autocite{book_eosl,book_adafaepov}. However, the analysis of the statistical improvements offered by the
cross-validated estimator compared to the data splitting estimator has proved difficult.
Furthermore, we note that there are at least two different quantities which may be seen as targets for the cross-validated risk estimator: 1) the average risk of the hypotheses learned on each fold and, 2) the risk of the estimation procedure over replications of the dataset.
Previously, work has mostly focused on understanding the performance in the case (1), whereas we consider both problems, with a heavier emphasis on the more delicate problem (2).
In the first analysis of the statistical performance of the cross-validated estimator, \textcite{blum1999holdout}
show that, under mild conditions, the cross-validated estimator performs no worse than
the corresponding data splitting estimator.
More recently, \textcite{kale2011cross} establish that, under some stability conditions, the $k$-fold cross-validated estimator achieves at least an asymptotic $k$-times reduction in variance. 
\textcite{kumar2013optimal} further improve on this result by relaxing the stability condition.

\section{Main results}
\label{sec:main-results}

\subsection{Notations and preliminaries}
As the complexity of the fitted models might increase with the number of observations, the results are presented for triangular arrays of estimators.
Let $(\calX_n)$ and $(\calY_n)$ be two sequences of Borel spaces.
Observations will take values in $\calX_n$ and estimators in $\calY_n$.
Classical examples may have $\calX_n = \mathbb{R}^d$ and $\calY_n = \mathbb{R}^p$ be euclidean spaces.
We denote $X^n \eqdef (X^n_{i})_{i\in \NN}$  a sequence of processes of i.i.d random variables taking values in a Borel
space $\calX_n$.
One can think of estimators as (measurable) functions that map observations to an element of $\calY_n$.
For example if $\calY_n=\calX_n=\RR$ the following function $f:(x_1,\dots,x_n)\rightarrow \frac{1}{n}\sum_{i\le n} x_i$ defines the empirical average.
For all $n\in \mathbb{N}$, we consider a sequence of estimators $(f_{l,n}:\Pi_{i=1}^l\calX_n\rightarrow \calY_n)_{l\in \NN}$.
The associated loss on an observation $x\in \calX_n$ is measured by a sequence of loss functions
$(\calL_n : \calX_n\times \calY_n \rightarrow \RR)$. 
Our goal is to approximate $\EE\Big[\calL_n\big(X^n_0, f_{l_n,n}(X^n_{1},\dots,X^n_n))\Big]$ where $l_n$ is the size of the training set.

We define $(K_n)$ to be a non decreasing sequence of integers; and for all
integers $n\in \mathbb{N}$ we write $(B^n_i)_{i\le K_n}$ a partition of
$\dbracket{n} = \{1, \dotsc, n\}$ that satisfies the following property:
\begin{equation}
\abs{B_i^n-B_j^n}\le 1,~\forall i,j\le K_n.
\end{equation}
This partition is used to define the folds: two observations $X^n_i$ and $X^n_j$ are in the same fold if $i,j$ belong to the same $B_k^n$ for some $k$. We abbreviate $(X^n_1, \dotsc, X^n_n)$ by $X^n_{1:n}$ and for all
subsets $B\subset \NN$ we write $X^n_{B} \eqdef (X^n_l)_{l\in B}$.
Therefore $X^n_{B_i^n}$ denotes all the observations in the $i$th fold and $X^n_{\dbracket{n} \setminus B_i^n}$ all the observations in $(X^n_1,\dots, X^n_n)$ that are not in the $i$th fold.

The
$K_n$-fold cross validated risk will involve $K_n$ different estimators, to
simplify matters we write
$\hat{f}_j(X^n) \eqdef f_{n-\abs{B_j^n}, n}(X^n_{\dbracket{n}\setminus B_j^n})$ the
estimators trained on $X^n_{\dbracket{n} \setminus B_j^n}$.
For a given hypothesis $f \in \calY_n$, we define its risk:
\begin{equation}
  R_n(f) \eqdef \EE(\mathcal{L}_n(\tilde{X}^n_1, f)),
\end{equation}
where $\tilde{X}^n_1$ denotes an independent copy of $X^n_1$.
Similarly, we define the average loss at a given observation, and the risk of the estimator:
\begin{align}
  \bar{\calL}_{l,n}(x) &\eqdef \EE\bigl[\calL_n(x, f_{l,n}(X^n_{1:l})\bigr], \\
  R_{l, n} &\eqdef \EE\bigl[\calL_n(\tilde X_0^n, f_{l,n}(X^n_{1:l},)\bigr].
\end{align}
Note that $  R_{l, n}$ is the quantity we want to approximate.
Finally, let $b_n(i)$ be the index of the  partition element $i$ belongs to, meaning $i\in B_{b_n(i)}^n$.

We wish to compare the performance of the two following estimators of the
predictive risk of an estimator:
\begin{align}
\hat{R}_{\mathrm{cv}} &\eqdef \frac{1}{n}\sum_{i \leq n}\mathcal{L}_n(X^n_i,\hat{f}_{b_n(i)}(X^n_{1:n})), \\
\hat{R}_{\mathrm{Split}} &\eqdef \frac{K_n}{n}\sum_{i\in B_1^n}\mathcal{L}_n(X^n_i,\hat{f}_{1}(X^n_{1:n})).
\end{align}
To do so we need to make some assumption on the stability
of the estimating procedure. In this goal we write $X^{'}$ to be an independent
copy of the process $X^n$ and for all integers $i,j\in \mathbb{N}$ the following
processes $X^{n,i}$ and $X^{n,i,j}$ are such that
\begin{equation}
  X^{n,i}_k= \begin{cases}
    X^n_k &\text{if } k \neq i,\\
    X^{'}_i &\text{otherwise.}
  \end{cases},
  \text{ and }
  X^{n,i,j}_k = \begin{cases}
    X^n_k &\text{if } k\neq i,j, \\
    X^{'}_k & \text{otherwise.}
  \end{cases}.
\end{equation}
We define the following two functionals: for $i, j \leq n$ and $i \neq j$,
$\nabla_i: (\calX^n_n \rightarrow \RR) \times \calX^n_n \rightarrow \RR$
and $\nabla_{i,j} : (\calX^n_n \rightarrow \RR) \times \calX^n_n \rightarrow \RR$ to be such that,
\begin{align}
  \nabla_i(f, X^n) &\eqdef f(X^n)-f(X^{n,i}), \\
  \nabla_{i,j}(f, X^n) &\eqdef \nabla_i(f,X^n)-\nabla_i(f,X^{n,j}).
\end{align}
Finally, we let $d_W(P, Q)$ denote the Wasserstein-1 distance between two distributions $P$ and $Q$.

Following common conventions in statistics, we may omit the explicit dependency on $X^n$ where appropriate: for example, we may write $\hat{f}_j = \hat{f}_j(X^n)$ for our estimator.
Additionally, although all of our results are presented in the triangular setup, we may often omit the index $n$ in our proofs when it is obvious from the context.

\subsection{Main Results}
\label{sec:main-results-statement}

We recall that $\hat{R}\big(\hat{f}_j(X^n)\big)$ is the conditional expected loss of the (random) estimator $\hat{f}_j(X^n)$ on a new data point. 
Consider the risk of the ensemble estimator:
\begin{equation}
\hat{R}^{\rm{average}}_{n,K_n} \eqdef \frac{1}{K_n}\sum_{j\le K_n}R_n \big(\hat{f}_j(X^n)\big),
\end{equation}
it may be viewed as the average risk of the hypotheses obtained from each fold, or alternatively viewed as the risk of an ensemble hypothesis which randomly selects one of the learnt hypotheses on a fold.
Note that it is still random as it depends on the $K_n$ different estimates $(\hat{f}_1(X^n),\dots, \hat{f}_{K_n}(X^n))$.
Its expected value is:
\begin{equation}
\overline{R}_{n,K_n} \eqdef \frac{1}{K_n}\sum_{j\le K_n}R_{n-|B_i^n|,n},
\end{equation}
which can be simplified to $R_{n - \abs{B_1^n}, n}$ when all blocks have the same size.
When studying the asymptotics of the cross-validated risk estimator $\hat{R}_{\text{cv}}$, one can study its convergence to either ${\hat{R}^{\rm{average}}_{n,K_n}}$ or alternatively to  $\overline{R}_{n,K_n}$.
The first one requires weaker conditions, and is of interest to characterise the performance of $\hat{R}_{\text{cv}}$ as an estimator of the risk of the ensemble \emph{hypothesis} (which is itself random).
On the other hand, convergence to $\overline{R}_{n, K_n}$ requires somewhat stronger conditions, and is of much broader interest.
Indeed, this is the regime to consider to understand the performance of $\hat{R}_{\text{cv}}$ in estimating the average performance of the \emph{estimator} of interest.
We present results for both cases.

We start by considering the convergence of $\hat{R}_{\text{cv}}$ to $\hat{R}_{n, K_n}^{\mathrm{average}}$.
In general, characterizing such convergence may require stability conditions on the estimators under consideration.
To express these stability conditions,
we define $\beta^{l,n}_1(X^n)$ and $\beta^{l,n}_2(X^n)$ to be the following random variables: 
\begin{align}
  \beta^{l,n}_1(X^n) &\eqdef \nabla_1\bigl(\calL_n(\tilde{X}^n_0, f_{l,n}(\cdot_{1:l})),X^n\bigr), \\
  \beta^{l,n}_2(X^n) &\eqdef \nabla_{1,2}\bigl(\calL_n(\tilde{X}^n_0, f_{l,n}(\cdot_{1:l})),X^n\bigr).
\end{align}
$\beta^{l,n}_1(X^n)$ represents the stability of $\calL_n(\tilde{X}^n_0, f_{l,n}(X_{1:l}))$ if the observation $X_1$ were changed for an independent copy $X'_1$.
This quantifies the first order stability of our model.
As for $\beta^{l,n}_2(X^n)$, it quantifies the second order stability by measuring the change to $\beta^{l,n}_1(X^n)$ if the observation $X_2$ was replaced by an independent copy.
Finally, the asymptotic variance will depend on:
\begin{equation}
    {\sigma}_{l,n}^2 \eqdef \EE\Bigl[\Var \Bigl(\calL_n(\tilde{X},f_{l,n}(X^n)) \mid f_{l,n}(X^n)\Bigr)\Bigr].
\end{equation}

\begin{theorem}\label{clt1}
  Let $(X^n_i)$ be a triangular array for i.i.d observations, and
  $(f_{l,n}:\mathcal{X}^l_n\rightarrow \mathcal{Y}_n)$ be a sequence of
  predicting functions. Write $(K_n)$ to be an increasing sequence, and
  let $N_n$ denote the set of at most two elements $N_n = [n - n / K_n, n + n / K_n + 1] \cap \NN$.
  Suppose that the following holds.
\begin{itemize}
\item[$H_0$.] The estimators are symmetric, i.e $f_{l,n} = f_{l,n}\circ \pi$ for all permutations $\pi\in \mathbb{S}_l$.
\item[$H_1$.] $\displaystyle \bigcup_{n \in \NN} \Bigl\{\calL_n(\tilde{X}^n_1,f_{l,n}(X^n_{1:l}))^2 : l \in N_n\Bigr\}$ is U.I.
\item[$H_2$.] $\displaystyle \max_{l \in N_n} \sqrt{n}\norm[\big]{ \beta_1^{l,n}(X^n)}_{L_2}=o(1)$; and  $\displaystyle \max_{l \in N_n} n\norm[\big]{\beta_2^{l,n}(X^n)}_{L_2}=o(1)$.
\item[$H_3$.] There exists $\sigma^2_1> 0$, s.t. $\displaystyle \max_{l \in N_n}\abs[\big]{\sigma_{l,n}^2-\sigma^2_1}\rightarrow 0$ as $n \rightarrow \infty$.
\end{itemize}
Then we have that:
\begin{equation}
  d_W\Bigl(\sqrt{n}
  \Big[\hat{R}_{\rm{cv}}-{\hat{R}^{\rm{average}}_{n,K_n}}\Bigr], \,
  \calN(0,\sigma^2_1)\Bigr)\rightarrow 0,
\end{equation}
and additionally if $K_n=o(n)$ then the following holds:
\begin{equation}
  d_W\Bigl(\sqrt{\frac{n}{K_n}}\Bigl[\hat{R}_{\rm{split}}-\hat{R}\big(\hat{f}_1(X^n)\big)\Bigr], \,
  \calN(0,\sigma^2_1)\Big) \rightarrow 0.
\end{equation}
\end{theorem}
\begin{remark}
  If the folds were all of the same size, then the terms $\max_{l\in N_n}$ in the statement of theorem could be deleted. However, if the folds are uneven then some estimators are trained on datasets with exactly one more observation than the others, hence the term $\max_{l\in N_n}$.
  We note that at each time the maximum is taken on at most two elements.
\end{remark}
\begin{remark}
\Cref{clt1} implies that the cross validated risks converges $\sqrt{K_n}$ times faster to $\hat{R}_{n,K_n}^{\mathrm{average}}$ than the split risk converges to $\hat{R}\big(\hat{f}_1(X^n)\big)$ when $K_n = o(n)$.
\end{remark}

Under stronger moment conditions one can prove Berry-Esseen bounds. Let $N_n$ denote the same interval
as in \cref{clt1}, and define the following quantities:
\begin{align*}
  \calS &\eqdef \max_n \max_{l \in N_n}\norm*{\calL_n(\tilde{X}^n_1,f_{l,n}(X^n_{1:l_n}))}_{L_4}, \\
  \epsilon_n(\Delta) &\eqdef \sqrt{n}\max_{l \in N_n}\norm*{\beta^{l,n}_1(X^n)}_{L_4}, \\
  \epsilon_n(\nabla) &\eqdef n\max_{l \in N_n}\norm*{\beta^{l,n}_2(X^n)}_{L_4}.
\end{align*}

\begin{theorem}\label{clt2}
  Let $(X^n_i)$ be a triangular array for i.i.d observations, and
  $(f_{l,n}:\mathcal{X}^l_n\rightarrow \mathcal{Y}_n)$ be a sequence of predicting
  functions. Write $(K_n)$ to be an increasing sequence. Suppose that the
  following holds:
\begin{itemize}
\item[$H_0$.] The estimators are symmetric i.e.  $f_{l,n} = f_{l,n}\circ \pi$ for all permutations $\pi\in \mathbb{S}_l$,
\item[$H_1$.] $\displaystyle \sup_{n}\max_{l \in N_n}\norm*{\calL_n(\tilde{X}^n_1,f_{l,n}(X^n_{1:l}))}_{L_4}<\infty$,
\item[$H_2$.] $\displaystyle \sqrt{n}\max_{l \in N_n}\norm*{\beta_1^{l,n}(X^n)}_{L_4}=o(1)$; and  $\displaystyle n\max_{l \in N_n}\norm*{\beta_2^{l,n}(X^n)}_{L_4}=o(1)$,
\item[$H_3$.] There exists $\sigma^2_1>0$ s.t. $\displaystyle \max_{l \in N_n}\abs*{\sigma_{l,n}^2-\sigma^2_1} \rightarrow 0$,
\end{itemize}
then there is a constant $C$ that does not depend on $n\in \mathbb{N}$,
$(f_{n,l})$ and $(X^n)$ such that: 
\begin{equation}
  \begin{split}
  &d_W\Big(\sqrt{n}\bigl(\hat{R}_{CV}- {\hat{R}^{\rm{average}}_{n,K_n}}\bigr), \, \calN(0,\sigma^2_1)\Big)\\
  &\quad\leq C \Bigl\{\calS\bigl(\epsilon_n(\Delta)+\epsilon_n(\nabla)\bigr)
    +\frac{1}{\sqrt{n}} \bigl(\calS^3+\epsilon_n(\Delta)^2\bigl(\epsilon_n(\nabla)+\calS\bigr)+\calS\epsilon_n(\nabla)^2 \bigr)\Bigr\}
 \\ &\qquad +\max_{l \in N_n}\abs[\big]{\sigma_{l,n}^2-\sigma^2_1}.
\end{split}
\end{equation}

\end{theorem}

The previous two theorems characterize the speed of convergence of the cross-validated risk
to the average risk of the $K_n$ different models $\hat{R}^{\mathrm{average}}_{n, K_n}$.
This is a random quantity and might not be the key quantity we are interested in.
The next theorem studies the speed of convergence of the cross validated
risk to the expected risk (its expectation). 
To state it we require some additional notation, let us write:
\begin{align}
  \calS(R) &\eqdef \sup_{n} n \max_{l \in N_n}\norm*{\EE\Bigl[\beta^{l,n}_1(X^n) \mid X^n]}_{L_4}, \\
  \epsilon_n(R) &\eqdef n^{\frac{3}{2}}\max_{l \in N_n}\norm*{\EE\Bigl[\beta^{l,n}_2(X^n) \mid X^n]}_{L_4}.
\end{align}
Note that in general, we expect the estimators to be close to the minimizer (among a certain class) of $R_n(f)$.
Therefore we may expect $\norm{\EE(\beta^{l,n}_1(X^n) \mid X^n)}$ to be smaller than $\norm{\beta^{l,n}_1(X^n)}$.
This will notably be true for parametric models.
A similar intuition is valid for $\beta_2^{l,n}(X^n)$. 

Although the problem appears similar at first, convergence to $\overline{R}_{n, K_n}$ behaves differently from convergence to $\hat{R}_{n, K_n}^{\mathrm{average}}$, and the asymptotic variances will not be the same than in \cref{clt1}.
To state the theorem more clearly, let $(\epsilon_n(\sigma))_{n \in \NN}$ and $(\epsilon_n(d))_{n\in \NN}$ be sequences and consider $\sigma_1$, $\sigma_2, \rho \in \mathbb{R}$ such that:
\begin{equation}\begin{split}
    \max_{l \in N_n} \abs*{\sigma_{l,n}^2-\sigma_1^2} \leq \epsilon_n(\sigma), \\
    \max_{l_1,l_2 \in N_n} \abs*{l_1\Cov\Bigl(\bar\calL_{l_2,n}( X_1), \EE\bigl(\beta^{l_1,n}_1(X^n) \mid X_1\bigr)\Bigr) - \rho} \leq \epsilon_n(\sigma), \\
    \max_{l \in N_n} \abs*{l\Var\Bigl[R\bigl(f_{l,n}(X_{1:l})\bigr)\Bigr] - \sigma_2^2} \leq \epsilon_n(\sigma), \\
    \max_{l_1,l_2 \in N_n} n\norm[\Big]{\EE\bigl[\beta_1^{n,l_1}(X^n)-\beta_1^{n,l_2}(X^n) \mid X_1\bigr]}_{4} \leq \epsilon_n(d).
\end{split}\end{equation}
As before, we note that $N_n$ is a set of size 2, and the maximums are not intended to denote any form of uniform convergence, and simply account for the potentially uneven splits.

\begin{theorem}\label{clt3}
  Let $(X^n_i)$ be a triangular array for i.i.d observations, and
  $(f_{l,n}:\calX^l_n\rightarrow \calY_n)$ be a sequence of predicting
  functions. Let $(K_n)$ denote any increasing sequence.
  Suppose that the following holds:
\begin{itemize}
\item[$H_0$.] The estimators are symmetric, i.e.  for all integers $l, n$, and for all permutations $\pi \in \mathbb{S}_n$ , $f_{l,n} = f_{l,n}\circ \pi$.
\item[$H_1$.] $\displaystyle \sup_{n}\max_{l \in N_n}\norm[\big]{\calL_n(\tilde{X}^n_1,f_{l,n}(X^n_{1:l}))}_{L_4} < \infty$,
\item[$H_2$.] $\displaystyle \sqrt{n}\max_{l \in N_n}\norm[\big]{\beta^{l,n}_1(X^n)}_{L_4}=o(1)$, and $\displaystyle n\max_{l \in N_n}\norm*{\beta^{l,n}_2(X^n))}_{L_4}=o(1)$,
\item[$H_3.$] $\displaystyle n\max_{l \in N_n}\norm[\big]{\EE\Bigl[\beta_1^{l,n}(X^n) \mid X^n]}_{L_4}=O(1)$; and $\displaystyle n^{\frac{3}{2}}\max_{l \in N_n}\norm[\big]{\EE\Bigl[\beta_2^{l,n}(X^n) \mid X^n]}_{L_4}=o(1)$.
\end{itemize}
Then there is a constant $C_1$ that does not depend on $n,
(f_{n,l}), (X^n)$ and $(\calL_n)$ such that the following holds:
\begin{equation}
  \begin{split}
    &d_W\Bigl(\sqrt{n}\Bigl[\hat{R}_{\rm{cv}}-\overline{R}_{n,K_n}\Bigr], \, \calN(0,\sigma_{\mathrm{cv}}^2)\Bigr)\\
    &\leq C_1 \Bigl\{\Bigl(\epsilon_n(\nabla)+ \epsilon_n(\Delta)+{\epsilon_n(\hat{R}})\Bigr) \Bigl(\epsilon_n(\nabla)
        + \epsilon_n(\Delta)  +\epsilon_{\frac{n}{2}}(\hat{R}) + \calS+\calS(R)\Bigr)
\\&\qquad +\frac{1}{\sqrt{n}} \bigl(\calS(R)+\calS\bigr)^3
+\calS(R)\epsilon_n(d)+ \epsilon_n(\sigma)\Bigr\}
 \end{split}
\end{equation}
where $\sigma_{\rm{cv}}^2 \eqdef \sigma_1^2+\sigma_2^2+2\rho$.
Additionally, if $K_n=o(n)$, then there exists $C_2$ which does not depend on $n$, $(f_{n,l})$, $(X^n)$ and $(\calL_n)$ such that:
\begin{equation}
  \begin{split}
    &d_W\Bigl(\sqrt{\frac{n}{K_n}}\Bigl(\hat{R}_{\rm{split}}- R_{n-|B_1^n|,n}\Bigr), \, \calN(0,\sigma_{\mathrm{split}}^2)\Bigr)\\
    &\quad\leq C_2  \Bigl\{n^{-1/2}\sqrt{K_n}\calS\bigl(\calS+\calS(R)\bigr)^2\\
    &\qquad+\Bigl(\epsilon_n(\Delta)+\frac{\epsilon_n(\nabla)+\epsilon_n(R)}{\sqrt{K_n}}\Bigr)\frac{\calS(R)}{\sqrt{K_n}}+\epsilon_n(\Delta)^2
+ \epsilon_n(\sigma)\Bigr\},
 \end{split}
\end{equation}
where $\sigma_{\rm{split}}^2 \eqdef \sigma_1^2+\sigma_2^2$
\end{theorem}

\begin{remark}
In \cref{clt1} the asymptotic variance of the cross-validated risk was the
same as the one of the split risk. This allowed us to conclude that the
cross-validated risk converged $\sqrt{K_n}$ times faster than the simple
split between train and test. However in \cref{clt3} the variances are
different. The cross-validated risk converges faster by a factor of:
\begin{equation}
    \sqrt{K_n\Bigl(1 + \frac{2\rho}{\sigma_1^2+\sigma_2^2}\Bigr)^{-1}}.
\end{equation}
In particular, the value of $\rho$ determines whether a ``full'' speed-up analogous to \cref{clt1} takes place.
For $\rho < 0$, we observe a reduction in variance by a factor larger $K_n$. For $\rho = 0$, we observe an exactly
$K_n$ times reduction in variance. For $\rho > 0$, we observe a reduction in variance by a factor less than $K_n$.
\end{remark}
\begin{remark}
  \Cref{clt3} only presents a Berry-Esseen bounds. Similarly to \cref{clt1}, a simple central limit theorems may be obtained under weaker moment conditions.
  Moreover if the loss functions $\mathcal{L}_n$ are uniformly bounded, i.e $\sup_n\| \calL_n \|_{\infty}<\infty$, then it is sufficient to replace $H_1-H_3$ with first moment conditions.
\end{remark}

\section{Estimation of the asymptotic variance and confidence intervals}
\label{sec:estimation-variance}

In the previous section we proved  central limit theorems and Berry-Esseen bounds for the cross-validated risk.
Using this, we wish to draw confidence intervals.
To do so we need to be able to estimate the asymptotic variance in a consistent manner.
In this section we propose estimators for $\sigma^2_1$ and $\sigma_{\rm{cv}}^2$; and characterize their speed of convergence to the desired quantity.

\begin{prop}
  \label{nice}
  Let $(X^n_i)$ be a triangular array for i.i.d observations, and
  $(f_{l,n}:\mathcal{X}^l_n \rightarrow \mathcal{Y}_n)$ be a sequence of predicting functions.
  Write $(K_n)$ to be an increasing sequence, and suppose that the conditions of \cref{clt2} are respected.
  Let $(\hat{\Sigma}_j)$ denote the empirical variance of each block:
  \begin{equation}
    \hat{\Sigma}^2_j(X^n) \eqdef
    \frac{1}{\abs{B_j^n}-1}\sum_{l\in B_j^n}\Bigl[
      \calL_n(X^n_l,\hat{f}_j(X^n))
      - \frac{1}{\abs{B_j^n}}\sum_{l\in B_j^n}\calL_n(X^n_l,\hat{f}_j(X^n))\Bigr]^2.
  \end{equation}
  Write $\Sigma_{\rm{cross}}^2 \eqdef \frac{1}{K_n}\sum_{j\le K_n}\hat{\Sigma}^2_j(X_{1:n})$, it is a consistent estimator of $\sigma_1^2$ and the following holds:
  \begin{equation}
  \begin{split}
    &\norm[\Big]{\Sigma_{\rm{cross}}^2-\sigma^2_1}_{L_2}  \leq \frac{8\calS}{\sqrt{n}}\Big[\calS+ \epsilon_n(\Delta)\Big]+\epsilon_n(\sigma)+o(\frac{1}{\sqrt{n}}).
  \end{split}
\end{equation}
\end{prop}
\noindent We delay the proof to \cref{sec:proof_estimator_sigma1}.

We now propose an estimator of $\sigma^2_{\rm{cv}}$.
However, estimating $\rho$ and $\sigma_2^2$ is not straightforward.
Indeed, $\sigma^2_2$ being the variance of $R_n(\hat f_j(X^n))$, for the classical empirical estimator for $\sigma^2_2$  to be consistent requires the number of folds $K_n$ to grow to infinity.
Furthermore, $\beta^{l_n,n}_1(X^n)$ is not observable which makes estimating $\rho$ harder. The key idea we exploit is to notice that $\EE\Big[(\hat{R}_{\rm{cv}}(X)-\hat{R}_{\rm{cv}}(X^i))^2\Big]$ is close to $\sigma_{\rm{cv}}^2$. 
We denote $\hat{R}^{n/2}_{\rm{cv}}(x)$ the $K_{\lfloor\frac{n}{2}\rfloor}$ cross validated on $\lfloor\frac{n}{2}\rfloor$ observations $x_{1:\lfloor\frac{n}{2}\rfloor}$.
We write  $\check{X}^i \eqdef (X_{1},\dots, X_{i-1},X_{i+\lfloor\frac{n}{2}\rfloor}, X_{i+1}, \dots, X_{\lfloor\frac{n}{2}\rfloor} )$ the observations $X_{1:\lfloor\frac{n}{2}\rfloor}$ where the i-th observations has been replaced by $X_{i+\lfloor\frac{n}{2}\rfloor}$.
\begin{prop}\label{cannes}
  Let $(X^n_i)$ be a triangular array for i.i.d observations, and
  $(f_{l,n}:\mathcal{X}^n_n\rightarrow \mathcal{Y}_n)$ be a sequence of predicting functions.
  Write $(K_n)$ to be a non-decreasing sequence.
  Suppose that the conditions of \cref{clt3}
  are respected; and let denote $\hat{R}^{n/2}_{\rm{cv}}(x)$ the $K_{\lfloor\frac{n}{2}\rfloor}$ cross-validated risk on $\lfloor\frac{n}{2}\rfloor$ observations.
  We consider the following estimator:
  \begin{equation}
  \hat{S}^2_{\rm{cv}}
  \eqdef \frac{n}{2}\sum_{i\le \frac{n}{2}}\Bigl[ \hat{R}^{n/2}_{\rm{cv}}(X)-\hat{R}^{n/2}_{\rm{cv}}(\check X^i)\Big]^2.
  \end{equation}
  It is a consistent estimator of $\sigma_{\text{cv}}^2$ and there is an absolute constant $C$ such that:
    \begin{equation}\begin{split}
    &\norm[\Big]{
      \hat{S}^2_{\text{cv}}-\sigma_{\text{cv}}^2}_{L_1}\\
      &\leq C\Bigl\{ \frac{1}{\sqrt{n}} \bigl(\calS+\calS(R)\bigr)^2 + \bigl(\calS+\calS(R)\bigr) \bigl(\epsilon_{\lfloor\frac{n}{2}\rfloor}(\Delta)+\epsilon_{\lfloor\frac{n}{2}\rfloor}(\hat{R})+\epsilon_{\lfloor\frac{n}{2}\rfloor}(d)\bigr) + \epsilon_{\lfloor\frac{n}{2}\rfloor}(\sigma) \Bigr\}
    \end{split}\end{equation}
\end{prop}
\begin{remark}
Note that in many parametric cases the upper bound is of the following order $\norm[\Big]{
      \hat{S}^2_{\text{cv}}-\sigma_{\text{cross}}^2}_{L_1}=O(\frac{1}{\sqrt{n}})$.
\end{remark}
\Cref{cannes} implies that the following interval:
\begin{equation}
\Bigl[\hat{R}_{CV} - \frac{\hat{S}_{\rm{cv}}}{\sqrt{n}}\Phi_{\frac{\alpha}{2}},
    \hat{R}_{CV} + \frac{\hat{S}_{\rm{cv}}}{\sqrt{n}}\Phi_{\frac{\alpha}{2}}\Bigr]
\end{equation}
is an asymptotically consistent $1-\alpha$ confidence interval for $R_{n-\frac{n}{K_n}}$.
However, we note that the estimator proposed in \cref{cannes} is computationally intractable
for large sample sizes and general estimators, due to the requirement of computing leave-one-out
type estimates.
However, such computation of fast approximate variants of such quantities has recently garnered much interest, especially in the context of leave-one-out cross-validation \cite{wang2018alocv,giordano2019ij}.
In the current work, we present some preliminary results for the estimation of $\sigma^2_{\text{cv}}$ for a ridge estimator, in which case a closed-form solution for $\hat{S}^2_{\text{cv}}$ is possible (see derivation in \cref{sec:proof_confidence_ridge}).
For illustration, we also present some simple simulation results in \cref{table:ridge-confidence}.
Details of the simulation may be found in \cref{sec:confidence-ridge-simulation}.
We leave further investigation to future work.
\begin{table}
    \centering
    \begin{tabular}{rrrr}
        \toprule
        & \multicolumn{3}{c}{level} \\
        $n$ & 80\% & 90\% & 95\% \\
        \midrule
        20  &   0.8300 & 0.8920 & 0.9288 \\
        40  &   0.8316 & 0.9078 & 0.9464 \\
        100 &   0.8300 & 0.9166 & 0.9520 \\
        200 &   0.8238 & 0.9108 & 0.9516 \\
        400 &   0.8068 & 0.9058 & 0.9494 \\
        800 &   0.8176 & 0.9120 & 0.9550 \\
        \bottomrule
    \end{tabular}
    \caption{Simulated coverage probability for a ridge regression example. \label{table:ridge-confidence}}
\end{table}

\section{Parametric estimation}
\label{sec:parametric}
In this section, we present examples in the class of parametric models.
This is not an exhaustive list of examples; but they  illustrate the richness of behavior of the cross validated risk and demonstrate the main points of this paper.
Additionally, the universality of parametric theory enables us to present more explicit formulas for the quantities which characterize the behaviour of the cross-validated risk.
The conditions presented are sometimes stronger than necessary; but make for simpler proofs which demonstrate better  how to apply our theorems. 

Theorem \ref{clt3} guarantees, under stability conditions, that the cross-validated risk converges $\sqrt{K_n}\sqrt{\frac{1}{1+2\frac{\rho}{\sigma_1^2+\sigma_2^2}}}$ faster than the simple split risk.
We use the term ``full speed-up'' to refer to the case where the cross validated risk converges $K_n$-times faster compared to the corresponding train-test split estimator, as if the folds were independent.
This corresponds to  $\rho=0$.  If $\rho>0$ is positive then the cross validated risk converges slower compared to the ``full speed-up'' case, whereas if  $\rho<0$ is negative the convergence is even faster. 
It is therefore apparent that the covariance term $\rho$ is key.
Intuitively one might expect than the cross-validated risk could not, outside of pathological cases, converge more than $K_n$ times faster than the split risk.
We demonstrate here that this is not the case.

In \cref{nelson} under mild conditions we compute $\rho$ for parametric models; and prove that the conditions of \cref{clt3} hold.
In \cref{lunajtm}  we give an illustrative example where $\rho>0$ is positive.
In \cref{montreal} we study a case where $\rho$ may take both positive and negative values for the same estimator, depending on the data-generating distribution.

\subsection{General properties of the cross-validated risk for parametric models.}\label{nelson}

Let $(X_i)_{i \in \NN}$ denote an i.i.d process, with $X_i \in \calX$.
Write $\Psi : \calX \times \RR^d \rightarrow \RR$ to be a function, strictly-convex twice-differentiable in its second argument.
Define $(\hat{\theta}_{l,n})$ to be the following sequence of estimators:
\begin{equation}
    \hat{\theta}_{m,n}(X_{1:n}) \eqdef \argmin_{\theta\in \RR^d}\frac{1}{m}\sum_{i\le m}\Psi(X_i, \theta).
\end{equation}
Such an estimator is often called a M-estimator.
We evaluate it under a loss $\calL : \calX \times \RR^d \rightarrow \RR$ (that may be different from $\Psi$) which verifies the following conditions:
\begin{itemize}
    \item For  all $\theta \in \RR^d$, $\norm{\mathcal{L}(X_1, \theta)}_{L_2} < \infty$;
    \item $R(\theta) \eqdef \EE(\calL(X_1, \theta))$ is continuously differentiable.
\end{itemize}
We write $\hat{R}_{\rm{cv}}$ be the cross-validated risk evaluated on the loss  function $\calL$.

\begin{prop}[Computing $\rho$ in parametric models]\label{ien}
Let $d, k\in \NN$ be integers, and let $(X_i)$ denote an i.i.d process. 
Write $\Psi: \calX \times \RR^d \rightarrow \RR$ be a function strictly-convex and twice-differentiable in its second argument.
Define $(\hat{\theta}_{l,n})$ to be the following sequence of estimators:
\begin{equation}
    \hat{\theta}_{m,n}(X_{1:n}) \eqdef \argmin_{\theta\in \mathbb{R}^d}\frac{1}{m}\sum_{i\le m}\Psi(X_i,\theta).
\end{equation}
Let $\theta^* \eqdef \argmin_{\theta\in \RR^d} \EE\bigl[\Psi(X_1,\theta)\bigr]$,
and consider a loss a function $\calL : \calX \times \RR^d \rightarrow \RR$ verifying the conditions stated previously.
Additionally, suppose that the conditions $(H_0)-(H_3)$ of theorem \ref{clt3} hold, then we have:
\begin{equation}
    \rho = -\Cov\Bigl(\partial_\theta R(\theta^*)^\top\bigl(\partial^2_\theta \EE(\Psi(X_1,\theta^*))\bigr)^{-1} \partial_\theta \Psi(X_1,\theta^*), \, \calL(X_1, \theta^*)\Bigr).
\end{equation}
\end{prop}
\noindent We delay the proof to \cref{sec:proof_parametric_rho}.

\begin{remark}
Note that if the model is evaluated on the same loss it has been trained on then $\Psi=\calL$.
Therefore we have $\partial_\theta R(\theta^*) = 0$ which implies that $\rho=0$.
\end{remark}

Conditions $(H_0)-(H_3)$ of \cref{clt3} are respected for a large range of models.
This is notably the case under regularity conditions on the losses $\Psi$ and $\calL$.
Those conditions are stronger than necessary and can be relaxed at the expense of a more technical proof.
\begin{prop}
 \label{ien_2}
Let $d,k \in \NN$ be integers, and let $(X_i)$ denote an i.i.d process. 
Let $\Psi : \calX \times \RR^d \rightarrow \RR$ and $\calL : \calX \times \RR^d \rightarrow \RR$ denote functions twice-differentiable in their second argument.
Suppose in addition that $\Psi$ is strictly convex in its second argument.
Define $(\hat{\theta}_{l,n})$ to be the following sequence of estimators:
\begin{equation}
    \hat{\theta}_{l,n}(X_{1:n}) \eqdef \argmin_{\theta \in \RR^d} \frac{1}{l}\sum_{i\le l}\Psi(X_i,f).
\end{equation}
Let $\hat{R}_{\rm{cv}}$ and $\hat{R}_{\rm{split}}$ denote respectively the cross-validated risk and the split risk evaluated with the loss $\calL$.
Let $\theta^* \eqdef \argmin_{\theta \in \RR^d}\EE(\Psi(X_1, \theta))$.
 Suppose that 
 \begin{equation}
 \sup_{\theta \in \RR^d} \EE\Bigl[\calL(X_1,\theta)^4\Bigr] < \infty,
 \end{equation}
 and that there is an open neighborhood $\calV \ni \theta^*$ and some $\delta > 0$ such that following holds:
 \begin{gather*}
  \inf_{\theta \in \calV} \lambda_{\text{min}} \Bigl\{ \partial^2_\theta \Psi(X_1, \theta) \Bigr\} \overset{a.s}{\ge} \delta, \\
  \norm[\Big]{\sup_{\theta \in \calV} \norm[\big]{\partial_\theta \calL(X_1, \theta)}_{L_2(v)} }_{L_{16}} < \infty,
  \qquad \norm[\Big]{\sup_{\theta \in \calV} \norm[\big]{\partial_\theta \Psi(X_1, \theta)}_{L_2(v)} }_{L_{16}} < \infty, \\
  \norm[\Big]{\sup_{\theta \in \calV} \lambda_{\text{max}} \Bigl\{ \partial^2_\theta \calL(X_1, \theta) \Bigr\} }_{L_{16}} < \infty,
  \qquad  \norm[\Big]{\sup_{\theta \in \calV} \lambda_{\text{max}} \Bigl\{ \partial^2_\theta \Psi (X_1, \theta) \Bigr\} }_{L_{16}} < \infty.
 \end{gather*}
Then the conditions of $(H_0)-(H_3)$ of \cref{clt3} are respected, and the following holds:
\begin{gather}
  d_W\Bigl(
    \sqrt{\frac{n}{K_n}}\bigl(
      \hat{R}_{\rm{split}}
      -R_{ \abs{B_1^n},n} \bigr) ,
  \, \calN(0,\sigma_1^2+\sigma_2^2)\Bigr)\rightarrow 0, \\
  d_W\Bigl(
    \sqrt{n}
     \bigl(\hat{R}_{\rm{cv}}- \overline{R}_{n,K_n}\bigr),
      \, \calN(0,\sigma_1^2+\sigma_2^2+2\rho)\Big)\rightarrow 0,
\end{gather}
where :
\begin{align*}
G_R &\eqdef \partial_{\theta} R(\theta^*), \quad
G_{\Psi}(X_1) \eqdef \partial_{\theta} \Psi(X_1, \theta^*),\quad 
H \eqdef \EE\bigl[\partial^2_\theta \Psi(X_1, \theta^*)\bigr], \\
\Sigma &\eqdef \Cov(G_\Psi(X_1)),\quad
\sigma_1^2 \eqdef \Var(\calL(X_1, \theta^*)), \\
\sigma_2^2 &\eqdef G_R ^\top H^{-1} \Sigma H^{-1} G_R,\quad
\rho \eqdef- \Cov\Bigl(G_R ^\top H^{-1} G_\Psi(X_1), \calL(X_1, \theta^*)\Bigr). \\
\end{align*}

\end{prop}
\noindent We delay the proof to \cref{sec:proof_conditions_parametric}.
\subsection[Example: full speed-up where training and testing loss coincide]{Example: full speed-up where $\Psi=\calL$}
One particular case where the cross-validated estimator is particularly well-behaved is when the estimator is trained on the same loss as is used for evaluation.
In that case, we have $\rho = 0$, and the reduction in variance does not depend on the data-generating distribution.
\begin{prop}
\label{prop:param}
Let $d,k\in \mathbb{N}$ be integers; and $(X_i)$ be an i.i.d process. 
Write  $\calL : \calX \times \mathbb{R}^d\rightarrow \mathbb{R}$ a function twice-differentiable and strictly convex in its second argument.
Define $(\hat{\theta}_{l,n})$ to be the following sequence of estimators:
\begin{equation}
\hat{\theta}_{m,n}(X_{1:n}) \eqdef \argmin_{\theta\in \mathbb{R}^d}\frac{1}{m}\sum_{i\le m}\calL(X_i,\theta).
\end{equation}
Let $\theta^* \eqdef \argmin_{\theta\in \mathbb{R}^d}\EE(\calL(X_1,\theta))$.
  Suppose that 
 \begin{equation}
 \sup_{\theta\in \RR^d} \EE\bigl[\calL(X_1,\theta)^4\bigr] < \infty
 \end{equation} and that there is an open neighborhood $\calV \ni \theta^*$ and some $\delta>0$ such that following holds:
 \begin{gather*}
  \inf_{\theta\in \calV} \lambda_{\text{min}} \Bigl\{ \partial^2_{\theta} \calL(X_1,\theta) \Bigr\} \overset{a.s}{\ge} \delta, \\
  \norm[\Big]{\sup_{\theta\in \calV} \norm[\big]{\partial_\theta \calL(X_1, \theta)}_{L_2(v)} }_{L_{16}} < \infty, \qquad\norm[\Big]{\sup_{\theta\in \calV} \lambda_{\text{max}} \Bigl\{ \partial^2_{\theta}\calL(X_1, f) \Bigr\} }_{L_{16}} < \infty 
 \end{gather*}
The conditions of $(H_0)-(H_3)$ of \cref{clt3} are respected. and the following holds:
\begin{gather}
  d_W\Bigl(
    \sqrt{\frac{n}{K_n}}\big[
      \hat{R}_{\rm{split}}
      -R_{ \abs{B_1^n},n} \bigr] ,
  \, \calN(0,\sigma^2)\Bigr)\rightarrow 0, \\
  d_W\Big(
    {\sqrt{n}}
     \big[\hat{R}_{\rm{cv}}- \overline{R}_{n,K_n}\big],
      \, \calN(0,\sigma^2)\Big)\rightarrow 0.
\end{gather}
where  $\sigma^2 \eqdef \Var(\calL(X_1,\theta^*))$.
\end{prop}
We note that although \cref{prop:param} may at first appear to include a large range of models used in practice (given the popularity of empirical risk minimization methods), it requires the \emph{exact} same loss to be used for training and testing.
In particular, many common techniques, such as regularizers, or surrogate losses for classification, violate the assumptions of \cref{prop:param}.
In those cases, the folds may not behave as if independent, and we present some examples below.

\subsection{Example: ridge regression}\label{lunajtm}
One of the most common cases in which the training and evaluation loss may differ is that we may wish to use a penalizer in the training process.
In this section, we present an example which illustrates how the presence of such a penalizer in the training loss (but not in the test loss) may affect the convergence of the cross-validated risk.

Let $(Z_i)$ and $(Y_i)$ be a sequence of i.i.d observations in  respectively $\mathbb{R}^d$ and $\mathbb{R}$; and $\lambda \in \mathbb{R}^+$ be a constant.
Define $(\hat{\theta}_{l,n})$ be the following sequence of estimators:
\begin{equation}
\hat\theta_{m,n}(X_{1:m}) \eqdef \argmin_{\theta\in \RR^d}\frac{1}{m}\sum_{i\le m} (Y_i-X_i^\top\theta)^2 + \lambda \|\theta\|_{L_2}^2.
\end{equation}
This estimator minimizes a penalized loss (often called \emph{ridge} loss).
However, we are more often interested in evaluating its performance under the mean-squared loss.
Let us write:
\begin{equation}
\theta^{*} \eqdef \argmin_{\theta\in \mathbb{R}^d}\EE\bigl[(Y_1-X_1^\top \theta)^2\bigr]+ \lambda \|\theta\|^2_{L_2},
\qquad \theta^{\rm{opt}} \eqdef \argmin_{\theta\in \mathbb{R}^d}\EE\bigl[(Y_1-X_1^\top \theta)^2\bigr].
\end{equation}
We define $\Delta \theta \eqdef \theta^{*}-\theta^{\rm{opt}}$.
For a random vector $Z$ we let $S_Z$ denote its corresponding variance-covariance matrix.
We can use \cref{ien} to analyze the behavior of the cross-validated risk.
\begin{prop}\label{prop:ridge-exact}
Let $(Z_i)$ be a sequence of i.i.d observations taking value in $\mathbb{R}^d$, with variance-covariance matrix $\Sigma^2$; and admitting a fourth moment.
Define $Y_i \mid Z_i\sim \calN( Z_i^\top \theta^{\mathrm{opt}}, \sigma^2)$; and write $X_i=(Z_i,Y_i)$.
Let $\lambda\in \mathbb{R}^+$ be a constant.
Define $(\hat\theta_{m,n})$ to be the following sequence of estimators:
\begin{equation}
    \hat\theta_{m,n}(X_{1:m}) \eqdef \argmin_{\theta\in \mathbb{R}^d}\frac{1}{l}\sum_{i\le m} (Y_i-Z_i^\top\theta)^2+ \lambda \norm{\theta}_{L_2}^2.
\end{equation}
The following holds:
\begin{align*}
    \rho &= -2\Delta_\theta^\top S_X \Bigl(S_X + \lambda I\Bigr)^{-1} \Bigl( S_{XX^\top\Delta_\theta} +2\sigma^2 S_X\Bigr)\Delta_\theta.
\end{align*}
In the special case where $X \sim \calN(0, S_X)$, we have in addition that:
\begin{align*}
    \sigma_1^2 &= 2 (\Delta_\theta^\top S_X \Delta_\theta + \sigma^2)^2, \\
    \Sigma &= (\Delta_\theta ^\top S_X \Delta_\theta + \sigma^2) S_X + (S_X \Delta_\theta)(S_X \Delta_\theta)^\top, \\
    \rho &= -4(\Delta_\theta^\top S_X \Delta_\theta + \sigma^2) \Delta_\theta^\top S_X (S_X + \lambda I)^{-1}S_X \Delta_\theta.
\end{align*}
\end{prop}

We note that here $\rho < 0$ in general, and its impact varies depending on the various
parameters of the underlying distribution and of the estimator. For illustration, we also
present some empirical results in \cref{table:ridge} for 2-fold cross-validation.

\begin{table}
    \centering
    \begin{tabular}{rrrr}
        \toprule
        $n$ & $n \Var \hat{R}_{\text{split}}$ & $n \Var \hat{R}_{\text{cv}}$ & Speedup \\
        \midrule
        50 & 8.08 (0.06) & 2.78 (0.02) & 2.90 (0.03)\\
        100 & 7.65 (0.05) & 2.42 (0.02) & 3.16 (0.03)\\
        200 & 7.45 (0.05) & 2.30 (0.01) & 3.24 (0.03)\\
        500 & 7.15 (0.05) & 2.19 (0.01) & 3.27 (0.03)\\
        1000 & 7.23 (0.05) & 2.14 (0.01) & 3.38 (0.03)\\
        $\infty$ & \makebox[\widthof{0.00 (0.00)}][l]{7.140} & \makebox[\widthof{0.00 (0.00)}][l]{2.124} & \makebox[\widthof{0.00 (0.00)}][l]{3.362} \\
        \bottomrule
    \end{tabular}
    \caption{Observed performance of 2-fold cross-validation for a ridge estimator. Parentheses denote standard error. $n=\infty$ denotes the value computed according to \cref{prop:ridge-exact}. \label{table:ridge}}
\end{table}

\subsection{Example: linear discriminant analysis}
\label{montreal}
In the context of classification, we are often interested in metrics such as the accuracy (i.e. $0-1$ loss) or other normalized variants.
Due to the non-smooth nature of the $0-1$ loss, models are most often trained on some smooth surrogate loss.
As \cref{ien} predicts, this may lead to different behaviours in the speed of convergence of the cross-validated risk.
The example we present in this section is particularly interesting, as the value of $\rho$ may vary substantially depending on the true data generating distribution.

Let $(Y_i)$ be an i.i.d Bernoulli process of parameter $1/2$.
Let $F_1,F_2$ be two different c.d.f admitting continuous p.d.f. $g_1$ and $g_2$ respectively.
We write $\EE_{g_1}(\cdot)$ (resp. $\EE_{g_2}(\cdot)$) the expectation taken with respect of the distributed generated by $g_1$ (resp. by $g_2$).
Let $\mu_1=\EE_{g_1}(Z)$ and $\mu_2=\EE_{g_2}(Z)$, and assume w.l.o.g. that $\mu_1>\mu_2$.
Define $(X_i)$ and $(Z_i)$ to be the processes such that
\begin{equation}
 Z_i \mid Y_i \eqdef Y_i g_1(\cdot)+ (1-Y_i) g_2(\cdot), \qquad X_i \eqdef (Z_i,Y_i).
\end{equation}
Our goal is to build a classifier that predicts the latent value of $\tilde Y$ given $\tilde Z$.
In this example, we consider the classical linear discriminant analysis method.
Define $\hat{\mu}_1 (X)\eqdef \frac{2}{n}\sum_{i\le n} Z_i ~\mathbb{I}(Y_i=1)$ and
$\hat{\mu}_2 (X)\eqdef \frac{2}{n}\sum_{i\le n} Z_i ~\mathbb{I}(Y_i=0)$.
Then given a new observation $\tilde Z$ we predict $\tilde Y$ as:
\begin{equation}
  C_n(\tilde Z) \eqdef \mathbb{I}\bigl\{(\tilde Z-\hat{\mu}_1)^2-(\tilde Z-\hat{\mu}_2)^2 < 0\bigr\}.
\end{equation}
We denote our estimator $\hat{f}_n(X) \eqdef (\hat{\mu}_1(X), \hat{\mu}_2(X))$.
In this classification framework, we wish to evaluate our estimators under
the $0-1$ loss:
\begin{equation}
  \calL(\tilde X) \eqdef \mathbb{I}\Bigl(\tilde Y \neq C_n(\tilde Z)\Bigr)
\end{equation}

\begin{prop}\label{cabris}
Consider the classification setup described above. Let the number of folds be $K=2$. We have that:
\begin{gather}
  d_W\Bigl(
    \sqrt{\frac{n}{2}}
      \Big[\hat{R}_{\rm{split}}-\overline{R}_{\frac{n}{2},n}\Big],
  \, \calN(0,\sigma^2)\Bigr)\rightarrow 0, \\
  d_W\Bigl(
    \sqrt{n}
      \Big[\hat{R}_{\rm{cv}}-\overline{R}_{\frac{n}{2},n}\Big],
      \, \calN(0,\sigma^2+2\rho)\Bigr)\rightarrow 0.
\end{gather}
where:
\begin{align*}
  \mu &\eqdef \frac{\mu_1 + \mu_2}{2}, \quad \Delta \eqdef g_1(\mu)-g_2(\mu) ,\quad q \eqdef F_1(\mu)-F_2(\mu) + 1, \\
  \sigma^2 &\eqdef \frac{\Delta^2}{8}\Big[\Var_{g_1}(Z)+\Var_{g_2}(Z) + \frac{1}{2}(\mu_1-\mu_2)^2\Big]+ q(1-q), \\
  \rho &\eqdef \frac{\Delta}{4}\Bigl[\EE_{g_2}\Big(X\mathbb{I}(X>\mu)\Big) + \EE_{g_1}\Big( X\mathbb{I}(X\le\mu)\Big) -2q\mu\Bigr].
\end{align*}
\end{prop}
\begin{proof}\renewcommand{\qedsymbol}{}
We report the proof to \cref{sec:proof-of-cabris}.
\end{proof}

In the case of LDA, we see that the potential speed-up may have a variety of behaviors, depending on the value
of $\rho$. For example, in the case where $g_1, g_2$ are two instances of a gaussian location family, it is not difficult
to check that $\Delta = 0$ and hence $\rho = 0$.
However, other distributions may exhibit different behaviors,
and in particular, we wish to emphasize that the observed speed-up is not a property solely of the estimator,
but rather jointly of the estimator and the distribution of the data.
For example, we present two distributional
settings in \cref{table:lda}, which yield two different regimes.
In the ``slow'' regime, we consider a setup where
$F_1 \sim \Gamma(10, 0.15)$ and $F_2 \sim \Gamma(1, 1)$, whereas in the ``fast'' regime, we consider a setup where
$F_1 \sim \Gamma(1, 10)$ and $F_2 \sim \Gamma(1, 1)$.
We observe that the speedup for two-fold cross-validation is slower than two in the slow regime, whereas it is
faster than two in the fast regime.

\begin{table}
    \centering
    \begin{tabular}{rrrrrrr}
    \toprule
    & \multicolumn{3}{c}{Slow} & \multicolumn{3}{c}{Fast} \\
    \cmidrule(r){2-4}\cmidrule(r){5-7}
    $n$ & $n \Var \hat{R}_{\text{split}}$ & $n \Var \hat{R}_{\text{CV}}$ & Speedup & $n \Var \hat{R}_{\text{split}}$ & $n \Var \hat{R}_{\text{CV}}$ & Speedup \\
    \midrule
    40 & 1.44 (0.01) & 0.83 (0.01) & 1.72 (0.02) & 0.43 (0.01) & 0.19 (0.00) & 2.31 (0.04)\\
    80 & 1.87 (0.02) & 1.10 (0.01) & 1.71 (0.03) & 0.43 (0.00) & 0.18 (0.00) & 2.34 (0.03)\\
    160 & 1.93 (0.04) & 1.13 (0.02) & 1.71 (0.04) & 0.42 (0.00) & 0.18 (0.00) & 2.33 (0.03)\\
    320 & 1.32 (0.04) & 0.78 (0.02) & 1.69 (0.07) & 0.43 (0.00) & 0.18 (0.00) & 2.39 (0.04)\\
    640 & 0.66 (0.02) & 0.40 (0.01) & 1.63 (0.08) & 0.43 (0.00) & 0.18 (0.00) & 2.34 (0.03)\\
    1280 & 0.53 (0.01) & 0.33 (0.00) & 1.60 (0.03) & 0.44 (0.00) & 0.18 (0.00) & 2.41 (0.03)\\
    2560 & 0.53 (0.01) & 0.33 (0.00) & 1.62 (0.02) & 0.44 (0.00) & 0.18 (0.00) & 2.37 (0.03)\\
    5120 & 0.53 (0.01) & 0.33 (0.00) & 1.62 (0.02) & 0.43 (0.00) & 0.18 (0.00) & 2.36 (0.03)\\
    $\infty$ & \makebox[\widthof{0.00 (0.00)}][l]{0.534}&\makebox[\widthof{0.00 (0.00)}][l]{0.326}&\makebox[\widthof{0.00 (0.00)}][l]{1.638} & \makebox[\widthof{0.00 (0.00)}][l]{0.438}&\makebox[\widthof{0.00 (0.00)}][l]{0.185}&\makebox[\widthof{0.00 (0.00)}][l]{2.367} \\
    \bottomrule
    \end{tabular}
    \caption{
        Variance of train-test split and 2-fold cross-validated accuracy for LDA. Parentheses denote standard error. $n = \infty$ denotes the value computed according to \cref{cabris}.
        \label{table:lda}
    }
\end{table}

\subsection{Counter-examples and remarks on conditions}
In this paper we focused on studying the case of asymptotically normal cross-validated risk. To do so we imposed the conditions $H_0$ -- $H_3$. In this subsection, we give examples of cases where those conditions are violated and the asymptotic distribution is not normal. We observe in those cases that  the cross-validated risk behaves quite differently than when it is asymptotically normal.

\subsubsection{Random speed-up: nearest neighbours}
We propose here, as a counter-example, an estimation framework that does not satisfy $(H_2)$ of \cref{clt1} and whose cross-validated risk does not satisfy the conclusion of \cref{clt1}.
Consider an i.i.d process $(Z_i)_{i\in\NN}$ uniformly distributed $Z_i \sim \calU[0,1]$, and define
for all $i \in \NN$:
\begin{equation}
  Y_i \eqdef \mathbb{I}(Z_i\le \frac{1}{2}),~X^n_i=\Big(Z_i,Y_i\Big).
\end{equation}

Suppose that we wish to build a classifier for $Y=\mathbb{I}(Z\le \frac{1}{2})$ given an observation $Z$.
An idea could be to use a nearest neighbour classifier, i.e a classifier
$f_n:[0,1]\times [0,1]^{n}\times \{0,1\}^{n} \rightarrow \{0,1\}$ which is defined in the following way:
\begin{equation}
  f_n(z,Z_{1:n}) = Y_{c(z,Z_{1:n})}, \quad \forall z\in [0,1]
\end{equation}
where $c(z,Z_{1:n}) \eqdef \argmin_{i \leq n} \abs{Z_i - z}$ 
denotes the index of the observation $Z^n_i$ closest to $z$.
The misclassification risk of this estimator will decrease to 0 at a speed of order $O(\frac{1}{n})$.
The non-vacuous loss function to consider is:
\begin{equation}
  \calL_n(y,f_n(z,X^n_{1:n})) \eqdef {\sqrt{n}}\times\mathbb{I}(y\ne f_n(z,X^n_{1:n}))
\end{equation}
We can estimate the expected loss, on a new observation, by a 2-fold cross-validation procedure.
However the error will not converge to a Gaussian distribution, and the ``speed-up'' is random. 
We also remark that $\epsilon_n(\nabla)$ does not converge to zero; and hence the conditions of \cref{clt1} do not hold.

\begin{prop}\label{luna}
We have the following:
the sequence $\bigl(\calL_n(\tilde Y,f_n(\tilde Z,X^n_{1:n})) ^2\bigr)_{l\le n}$ is U.I;
but  $\max_{l\in N_n}\sqrt{n}\norm[\big]{\beta_1^{l,n}(X)}_{L_2}=O(1)$.

Additionally, let $N_1,N_2\overset{i.i.d}{\sim}\rm{exp}(0.5)$, $U_1,U_2\overset{i.i.d}{\sim}\rm{exp}(1)$ and $s_1,s_2\overset{i.i.d}{\sim}\rm{Unif}(\{-1,1\})$, then we have that:
\begin{equation*}\begin{split}&
    \sqrt{n}\hat{R}_{cv}\xrightarrow{d} \mathbb{I}(s_1s_2=-1)\Bigl[\mathbb{I}(r^d_1\ge 0)\bigl[1+\mathrm{Poisson}(r^d_1)\bigr]+\mathbb{I}(r^d_2\ge 0)\bigl[1+\mathrm{Poisson}(r^d_2)\bigr] \Bigr]
    \\&\qquad \qquad+\mathbb{I}(s_1s_2=1)\Bigl[\mathbb{I}(r^s_1\ge 0)\bigl[1+\mathrm{Poisson}(r^s_1)\bigr]+\mathbb{I}(r^s_2\ge 0)\bigl[1+\mathrm{Poisson}(r^s_2)\bigr] \Bigr]
    \\&   \sqrt{n}\hat{R}_{\rm{split}}\xrightarrow{d} \mathrm{Poisson}(N_1)
\end{split}\end{equation*}
 where we have written:
 \begin{gather*}
    r^d_1 \eqdef s_1 N_1 + s_2 U_2, \quad
    r^d_2 \eqdef s_2 N_2 + s_1 U_1, \\
    r^s_1 \eqdef r^d_1 - 2 s_2 N_2, \quad
    r^s_2 \eqdef r^d_2 - 2 s_1 N_1.
 \end{gather*}

\end{prop}

\begin{remark}
We note that the magnitude of $\sqrt{n}\hat{R}_{\rm{split}}$ compared to the one of $\sqrt{n}\hat{R}_{\rm{cv}}$ is random as it depends on the relative size of $U_1,U_2,N_1$ and $N_2$.
\end{remark}
\noindent We delay the proof to \cref{sec:proof_nearest_neighbour}.

\subsubsection{Noiseless models}
We present an example where hypothesis $H_3$ does not hold, often called ``noise-free'' or ``realizable'' setting, although the terminology is somewhat misleading as $\calL(X, \hat{\theta})$ does not vanish, but rather its variance vanishes.
We consider the problem of estimating a mean in a symmetric Bernoulli model.
More specifically, suppose that $(X_i)$ is uniformly distributed on $\{ \pm 1 \}$, and our estimator is given by:
\begin{equation}
    \hat{\theta} = \argmin_{\theta \in \RR} \sum_{i = 1}^n (y_i - \theta)^2 = \frac{1}{n} \sum_{i = 1}^n y_i.
\end{equation}
We note that in our case, with the natural loss $\calL(x, \theta) =  (x - \theta)^2$, we have that:
\begin{align*}
    \sigma^2_{m,n} = \EE[\Var( (y - \hat{\theta})^2 \mid \hat{\theta})] = 4 \EE \hat{\theta}^2 = \frac{1}{m}.
\end{align*}
In particular, we have that $\sigma^2_{m,n} \rightarrow 0$, which would lead to a degenerate limit
in the regime of theorem 3.
Therefore we may choose to work with the rescaled loss $\calL_n(x,\mu) \eqdef \sqrt{n}(x-\mu)^2$, but we note that such a loss does not satisify the stability conditions of \cref{clt3}.
\begin{prop}
Let $(X_i)$ be an i.i.d process with marginal $X_i\sim \mathrm{unif}\{-1,+1\}$. Let $\calL_n$ denote the rescaled square loss loss function $\calL_n: (x,\theta) \mapsto \sqrt{n}(x-\mu)^2$. The estimators $(\hat\theta_{m,n})$ are such that they respect: 
\begin{equation}
    \hat{\theta}_{m,n} \eqdef \argmin_{\theta\in \mathbb{R}}\frac{1}{m}\sum_{i\le n}\calL_n(X_i,\theta).
\end{equation}
Define $(Y_i)\overset{i.i.d}{\sim}N(0,1)$ to be i.i.d normal.
The $K$ fold cross-validated risk satisfies:
$$n\Big[\hat{R}_{CV}-1\Big]\xrightarrow{d}\frac{1}{4K}\sum_{i\le K}\Bigl[Y_i-\frac{1}{K-1}\sum_{j\ne i}^K Y_j\Big]^2.$$
\end{prop}
We note in particular that the limiting distribution is not normal.
The proof is a direct computation using the properties of quadratic forms of i.i.d. normal variables, and is omitted.

\newpage
\section{Proofs}
\label{sec:proofs}

\subsection{Preliminary propositions}
We first present a few known results which will be of use in our proofs.
The martingale central limit theorem is an important tool that we use to establish \cref{clt1}.
An introduction may be found in \cite{hall2014martingale}, we reproduce the statement here for completeness.
\begin{prop}[Martingale Central Limit Theorem]\label{proof:martingale}
Let $(S_{i,n}, \mathbb{F}_{i})_{i,n\in \NN}$ be a triangular array of martingales with martingale differences $Y_{i,n} \eqdef S_{i, n} - S_{i - 1, n}$.
Suppose that:
\begin{itemize}
\item For all positive $\epsilon>0$, we have that: $\sum_{i\le n}\EE\bigl[Y_{i,n}^2\mathbb{I}(|Y_{i,n}|\ge \epsilon) \mid \mathbb{F}_{i-1}\bigr]\rightarrow 0$,
\item  There exists a r.v. $\sigma$ which is $\mathbb{F}_{\infty}$-measurable such that $\sum_{i\le n}\EE\bigl[Y^2_{i,n} \mid \mathbb{F}_{i-1}^n\bigr]\xrightarrow{P} \sigma^2$.
\end{itemize}
Then we have that:
$S_{i,n}\xrightarrow{d}\sigma Z,$ where $Z\sim N(0,1).$
\end{prop}
We will repeatedly make use of the Efron-Stein inequality, which helps us prove that the empirical variance concentrates.
It is used in the proof of \cref{clt1,clt2,clt3}, \cref{nice}, and \cref{cannes}.
It is a standard tool for concentration, and a reference may be found in \cite{boucheron2013concentration}.
We reproduce the statement here for completeness.
\begin{prop}[Efron-Stein inequality]
Let $(X_i)$ be an i.i.d process taking value in $\calX$ and $h:\calX^n\rightarrow \mathbb{R}$ be a measurable function. If $h(X_{1:n})$ admits a second moment then the following holds:
\begin{equation}
\Var\Bigl[h(X_{1:n})\Bigr] \le \sum_{i=1}^n\EE\Bigl[\Big(h(X_{1:n})-h(X^i_{1:n})\Big)^2\Big].
\end{equation}
\end{prop}
Finally, the proofs of \cref{clt2,clt3} adapt the Stein's method for central limit theorems (see \cite{ross2011stein} for a general introduction).
It is based on the following observation that for any real-valued random variable $W$ and standard normal random variable $Z$ we have:
\begin{equation}
  d_W(W,Z)\;\leq\;\sup_{f\in\mathcal{F}}\abs[\big]{\EE\bigl[W f(W)-f'(W)\bigr]},
\end{equation}
where $\mathcal{F}$ is the following function class:
\begin{equation}
  \mathcal{F} \eqdef \Bigl\{f \in \mathcal{C}^2(\mathbb{R})\,\bigm\vert\,\norm{f}_{\infty} \leq 1, \, \norm{f'}_{\infty} \leq \sqrt{2/\pi}, \, \norm{f''}_{\infty}\leq 2\Bigr\}.
\end{equation}

\subsection{Proof of theorem~\ref{clt1}}
The proof for the train-test split risk is very similar to the proof the cross validated risk;
and we therefore only include the proof for the latter.
For ease of notation we drop the superindex $n$ from the interpolating processes $X^{n,i}$ and $X^{n,i,j}$; and instead write $X^i \eqdef X^{n,i}$ and $X^{i,j} \eqdef X^{n,i,j}$.
Let $(F^n_i)$ denote filtration given by $F^n_i \eqdef \sigma(X^n_1,\dots,X^n_i)$, and 
let $K^i_n, S^i_n$ be defined as:
\begin{align}
  K^i_n(X^n) &\eqdef \calL(X^n_i,\hat{f}_{b(i)}(X^n))-R_n(\hat{f}_{b_n(i)}(X^n)), \\
  S^i_n(X^n) &\eqdef \frac{1}{\sqrt{n}}\sum_{l\le n}\EE\bigl[K_n^l(X^n) \mid \mathbb{F}_i \bigr], \\
  Y^i_n(X^n) &\eqdef S^i_n-S^{i-1}_n.
\end{align}
Note that $(S^i_n,\mathbb{F}^n_i)$ forms a triangular array of martingales.
To prove the desired result we will use \cref{proof:martingale}.
In the rest of the proof, to simplify notations we write $K_n^i \eqdef K_n^i(X^n)$.
Before we start, it will be useful to note that:
\begin{equation}\label{logan}
       Y^i_n = \frac{1}{\sqrt{n}}\EE\Bigl[K_n^i+\sum_{\mathclap{l\not \in B_{b_n(i)}^n}} \nabla_i K_n^l \mid \mathbb{F}^n_i\Bigr].
\end{equation}

\begin{proof}[Proof of \cref{clt1}]
We start by noting that for all sets $A\in \mathbb{F}_i^n$ the following holds:
\begin{equation}\begin{split}
  n\EE \Big[\mathbb{I}(A)(Y_i^n)^2\Big]
  &\oversetclap{(a)}{\le}2\EE\bigl[\mathbb{I}(A) (K^i_n)^2\bigr]
    +2\EE\Bigl[\bigl(\sum_{l\ne b(i)}\sum_{j\in B_{l}(i)} \nabla_i K_n^l \bigr)^2\Bigr]
   \\&\oversetclap{(b)}{\le} 4 \max_{l_n\in N_n}\EE\Bigl[\mathbb{I}(A)\mathcal{L}(\tilde X_1, f_{l,n}(X_{1:l}))^2\Bigr]
   \\&\quad+  4\PP(A)\max_{l_n\in N_n}\EE\Big[\calL(\tilde X_1, f_{l,n}(X_{1:l}))^2\Big]
   \\&\quad+ 2\sum_{l,j\not\in B_{b_n(i)}^n} \EE\bigl[(\nabla_i K_n^l) (\nabla_i K_n^j)\bigr].
\end{split}
\end{equation}
 where (a) and (b) are consequences of Jensen inequality. 
Using $H_1$ we can see that if we prove that
$\sum_{l,j\not\in B_{b_n(i)}^n} \EE\bigl[\nabla_i K^l_n \: \nabla_i K^j_n\bigr] \rightarrow 0$
then we have that for all $\epsilon>0$ the following holds:
$\sum_{i\le n}\EE\bigl(Y_{i,n}^2\mathbb{I}(|Y_{i,n}|\ge \epsilon) \mid \mathbb{F}_{i-1}\bigr)\rightarrow 0$.
If $j,l$ belongs to the same block, i.e. $b_n(j)=b_n(l)$,  but are distinct, i.e. $j\ne l$, then by conditional independence  we obtain that:
\begin{equation}\begin{split}
& \EE\bigl[\nabla_i K^l_n \: \nabla_i K^j_n \bigr]
\\&=\EE\Bigl[\EE\bigl[K^l_n - K^l_n(X^i) \mid X_{\dbracket{n} \setminus \{j,l\}}, X^{i}\bigr]
    \EE\bigl[ K^j_n(X) - K^j_n(X^i) \mid X_{\dbracket{n} \setminus \{j,l\}},X^{i}\bigr]\Bigr]
\\&=0.
\end{split}
\end{equation}
If $j,l$ belong to different blocks, i.e. $b_n(j)\ne b_n(l)$ then we have that:
\begin{equation}\begin{split}
\abs[\big]{\EE\bigl[\nabla_i K_n^l \: \nabla_i K_n^j \bigr]}
&\oversetclap{(a)}{=}\abs[\big]{\EE\bigl[ \nabla_{i,j} K_n^l \: \nabla_{i,j} K_n^j\bigr]}
\\&\leq \max_{l\in N_n}\big\| \beta_2^{l, n} \big\|_{L_2}^2
\end{split}
\end{equation}
where (a) is a consequence of the following identity:
\begin{equation}\begin{split}
&\EE\bigl[ \nabla_i K^l_n \: \bigl(K^j_n(X^{l}) - K^j_n(X^{i,l})\bigr)\bigr]
\\&=\EE\big[\EE\big[ \nabla_i K^l_n \mid X^l\big] \times \bigl(K^j_n(X^{l}) - K^j_n(X^{i,l})\bigr)\bigr]
\\&=0
\end{split}
\end{equation}
Therefore using \cref{logan} we obtain that:
\begin{equation}\begin{split}&
   \sum_{l,j\not\in B_{b_n(i)}^n} \EE\bigl[\nabla_i K^l_n \: \nabla_i K_n^j\bigr]
   \leq n^2\max_{l_n\in N_n} \norm[\big]{\beta^{l_n,n}_{2}}_{L_2}^2
    + n \max_{l_n\in N_n}\norm[\big]{\beta_1^{l_n,n}}_{L_2}^2
   \rightarrow 0.
\end{split}
\end{equation}
    Therefore the first point is proven.

Remains to show that $\sum_{i\le n} \EE\bigl[ (Y^i_n)^2 \mid \mathbb{F}_{i-1}^n\bigr]\xrightarrow{P} \sigma^2_1$.
Using similar arguments to previously we  can note that:
\begin{equation}
\frac{1}{n}\sum_{i\le n}\EE\bigl[\bigl(\sum_{l\not\in B_{b_n(i)}^n} \nabla_i K^l_n\bigr)^2\bigr]
\rightarrow 0,
\end{equation}
and we therefore have that:
\begin{equation}
\sum_{i\le n}\EE\bigl[(Y^i_n)^2 \mid \mathbb{F}_{i-1}^n\bigr]
=\frac{1}{n}\sum_{j\le K_n} \abs{B_j^n} \Var\bigl[\calL_n(\tilde X_1,\hat{f}_j) \mid X \bigr] + o_p(1).
\end{equation}
Using $H_4$ we know that
$
\frac{1}{n} \sum_{j\le K_n}\abs{B_j^n} ~ \EE\bigl[\Var\bigl[ \calL_n(\tilde X_1,\hat{f}_j) \mid X\bigr]\bigr]
=\sigma_1^2+o(1).
$ 
Hence suffices to show that $\frac{1}{n}\sum_{j\le K_n}|B_j^n| \Var\bigl[\calL_n(\tilde X_1,\hat{f}_j(X)) \mid X\bigr]$ concentrates around its expectation.
We wish to use Efron-Stein inequality.
However, to handle the case where $\Var \bigl[\calL_n(\tilde X_1,\hat{f}_j) \mid X\bigr]$ does not have a second moment, we first need to truncate it.
Let us introduce the following clipping function:
\begin{equation}
g(x, M) = \begin{cases}
x & \text{if } \abs{x} \le M, \\
-M &\text{if } x \le -M,\\
M & \text{otherwise.}
\end{cases}
\end{equation}
Remark that $g(\cdot,M)$ is 1-Lipchitz for all $M\in \RR^+$.
Let $(\gamma_n)$ be an increasing sequence such that (i) $\gamma_n\rightarrow \infty$, (ii) $ \gamma_n\max_{l\in N_n}\big\|\beta_1^{l,n}(X)\big\|_{L_2}\rightarrow 0$. 
Then we have
\begin{equation}\begin{split}
&\frac{1}{n}\sum_{j\le K_n} \abs{B_j^n} \Var\bigl[\calL_n(\tilde X_1,\hat{f}_j) \mid X \bigr]
\\&=\frac{1}{n}\sum_{j\le K_n} \abs{B_j^n} \Bigl\{\Var\bigl[\calL_n(\tilde X_1,\hat{f}_j) \mid X\big]
- \Var \bigl[g\bigl(\calL_n(\tilde X_1,\hat{f}_j),\gamma_n\bigr)\mid X\big]\Bigr\}
\\&\quad+  \frac{1}{n}\sum_{j\le K_n} \abs{B_j^n} ~ \Var\bigl[g\bigl(\calL_n(\tilde X_1,\hat{f}_j(X)) ,\gamma_n\bigr)\mid X\big]
\\&\le (a_1)+(a_2).
\end{split} \end{equation}

We will bound each term successively. First, using hypothesis $H_1$ we have that:
\begin{equation}\begin{split}
(a_1) &=
\frac{1}{n}\sum_{j\le K_n} \abs{B_j^n}~ \Bigl\{\Var\bigl[\calL_n(\tilde X_1,\hat{f}_j) \mid X \bigr] 
-\Var\bigl[g\bigl(\calL_n(\tilde X_1,\hat{f}_j), \gamma_n\bigr) \mid X \bigr]\Bigr\}
\\&= o_p(1).
\end{split} \end{equation}
Therefore it is enough to prove that $(a_2)$ concentrates.
For this note that by the Efron-Stein inequality we have:
\begin{equation}\begin{split}
& \Var\Bigl[\frac{1}{n}\sum_{j\le K_n} \abs{B_j^n} \Var\bigl[g\bigl(\calL_n(\tilde X_1,\hat{f}_j) ,\gamma_n\bigr)\mid X \big] \Bigr]
\\&\le \frac{1}{n^2}\sum_{j\le n}\norm[\Big]{
    \sum_{l\le K_n} \abs{B_l^n} \Bigl\{ \Var\bigl[g\bigl(\calL_n(\tilde X_1,\hat{f}_k) , \gamma_n\bigr) \mid X \bigr]
    - \Var\bigl[g\bigl(\calL_n(\tilde X_1, \hat{f}_k(X^j)), \gamma_n\bigr) \mid X^j \bigr]\Bigr\} }_{L_2}^2
    \\&\oversetclap{(a)}{\le} \frac{4}{n^2}\sum_{j\le n}\norm[\Big]{
    \sum_{l\le K_n}|B_l^n|~ \Bigl\{ g\bigl(\calL_n(\tilde X_1,\hat{f}_k(X)) , \gamma_n\bigr)^2
    - g\bigl(\calL_n(\tilde X_1, \hat{f}_k(X^j)), \gamma_n\bigr)^2\Bigr\} }_{L_2}^2
     \\&\le \frac{4}{n^2}\sum_{j\le n} \sum_{l,k\le K_n}|B_l^n||B_k^n|\norm[\Big]{
   \Bigl\{ g\bigl(\calL_n(\tilde X_1,\hat{f}_k(X)) , \gamma_n\bigr)^2
    - g\bigl(\calL_n(\tilde X_1, \hat{f}_k(X^j)), \gamma_n\bigr)^2\Bigr\}
    \\&\qquad \qquad \qquad\qquad \qquad \qquad \times    \Bigl\{ g\bigl(\calL_n(\tilde X_2,\hat{f}_l(X)) , \gamma_n\bigr)^2
    - g\bigl(\calL_n(\tilde X_2, \hat{f}_l(X^j)), \gamma_n\bigr)^2\Bigr\}}_{L_2}^2
\\&\oversetclap{(b)}{\le} 16\gamma_n^2\frac{1}{n^2}\sum_{j\le n}\Bigl(\sum_{l\le K_n}\abs{B_l^n}~
    \norm[\big]{\calL_n\big(\tilde X_1,\hat{f}_l(X^{n})\big) - \calL_n\big(\tilde X_1,\hat{f}_l(X^j)\big)}_{L_2}\Bigr)^2
\\&\le 16 \gamma_n^2n \max_{l\in N_n}\big\|\beta_1^{l,n}(X^n)\big\|_{L_2}^2\rightarrow 0
\end{split}\end{equation}
where (a) is a consequence of the Cauchy-Swartz inequality; and (b) comes from the fact that firstly $\big|g\bigl(\cdot, \gamma_n\bigr)\big|$ is upperbounded by $\gamma_n$ and secondly that $g(\cdot,\gamma_n)$ is 1-Lipchitz.
This implies that  $\sum_{i\le n}\EE\big[( Y_i^n)^2 \mid \mathbb{F}_{i-1}^n\big]\xrightarrow{P} \sigma^2$.
We may conclude by \cref{proof:martingale}.

\end{proof}

\subsection{Proof of theorem~\ref{clt2}}
\begin{proof}
  The key idea will be to use Stein method to prove the desired result. We
  only present the proof for the cross validated risk, as the other case is very
  similar and can be easily deduced from the one presented here.
{ For simplicity of notations we will drop the superindex $n$ when writing the processes: for example we write $X$ instead of $X^n$ and $X^i$ instead of $X^i$.}

Let $(F_i)$ denote filtration given by: $F_i \eqdef \sigma(X_1,\dots,X_i)$, and 
let $K_n^i, W_n, W^i_n$ be defined as:
\begin{align}
  K_n^i(X) &\eqdef \calL_n(X_i,\hat{f}_{b(i)}(X))-R_n(\hat{f}_{b_n(i)}(X)), \\
  W_n(X) &\eqdef \frac{1}{\sqrt{n}}\sum_{j\leq n} K_n^j(X), \\
  W^i_n(X) &\eqdef \EE\bigl[W_n(X) \mid X_{\dbracket{n}\setminus i} \bigr].
\end{align}
By abuse of notation, we write $K_n^i = K_n^i(X)$, $W_n \eqdef W_n(X)$ and $W^i_n \eqdef W^i_n(X)$.
Before diving into the proof we note that:
\begin{equation}\label{faches}
  \EE\bigl(W_n \mid X \bigr)-\EE\bigl(W_n(X^i) \mid X\bigr) = W_n-W^i_n.
\end{equation}

We are now ready to consider the main body of the proof.
We recall Stein's principle for central limit theorems. Let $Z$ denote
a standard random normal variable and $\sigma > 0$, then:
\begin{equation}
  d_W(W_n,\sigma Z)\leq \sup_{H\in \mathcal{F}}\abs*{\EE\Bigl(W_n H(W_n)-\sigma^2H'(W_n)\Bigr)},
\end{equation}
where $d_W$ denotes the Wasserstein distance and $\mathcal{F}$ denotes the set of twice-differentiable functions
with first and second derivatives uniformly bounded by 1.

For any function $H\in \mathcal{F}$ this can be re-expressed as
\begin{equation}
  \begin{split}&\label{eq:triste-theorem-2}
    \EE\Big(W_n H(W_n)-\sigma^2H'(W_n)\Big) \\
    &=\sum_{i \leq n} \Bigl\{ \EE\Big(K_n^i H(W^i_n)\Big)+ \EE\Big(K_{n}^i\big[H(W_n)-H(W^i_n)-[W_n-W^i_n]H'(W_n)\big]\Big)\\
    &\qquad \quad + \EE\Big(K_n^i [W_n-W^i_n]H'(W_n)-\sigma^2H'(W_n)\Big) \Bigr\}
\\&=(a)+(b)+(c).
\end{split}\end{equation}

We bound each respective term and show that they converge to zero.
To begin, note that for all integers $i\leq n$ we have:
\begin{equation}\begin{aligned}
  \text{(a)}&=\EE\Big(K_n^i H(W^i_n)\Big) \\
  &=\EE\Big(\EE\bigl(K_n^i H(W^i_n) \mid X_{\dbracket{n} \setminus i}\bigr)\Big) \\
  &\oversetclap{(a_1)}{=}\EE\Big(H(W^i_n)\EE \bigl(K_n^i \mid X_{\dbracket{n} \setminus i}\bigr)\Big) 
  \oversetclap{(a_2)}{=}0
\end{aligned}\end{equation}
where to get $(a_1)$ we used the fact that $W_{n,i}$ is
$\sigma\Big(X_{\dbracket{n} \setminus i}\Big)$ measurable and to get $(a_2)$ that
$\EE\bigl(K_{n}^i \mid X_{\dbracket{n} \setminus i}\bigr)=0$.

Secondly, we want to upper bound (b). Let us introduce  $D^{i,j}_n \eqdef \nabla_i K_n^j$. We remark that it can also be written in the following fashion:
\begin{equation}
\begin{split}
D^{i, j}_{n}(X) = \calL_n(X_j,\hat{f}_{b_n(l)}(X))-\calL_n(X_j,\hat{f}_{b_n(j)}(X^i)) \\
-\Bigl[R_n\bigl(\hat{f}_{b_n(l)}(X)\bigr)-R_n\big(\hat{f}_{b_n(j)}(X^i)\big)\Bigr].
\end{split}
\end{equation}
As for $K_n^i$, we abuse notations and write $D^{i, j}_n = D^{i,j}_n(X)$. Note that
\begin{align}
W^i_n&=\EE(W_n(X^i) \mid X)
\\&=\EE\bigl[\calL_n(X_i,\hat{f}_{b_n(i)}) + \sum_{\mathclap{j\not\in B_{b_n(i)}^n}}\calL_n(X_j,\hat{f}_{b_n(j)}) \mid X^i\bigr].
\end{align}
This implies that the following equality holds
\begin{align}\label{comfort}
W_n-W^i_n
&=\calL_n(X_i,\hat{f}_{b_n(i)}(X))-R_n\bigl(\hat{f}_{b(i)}(X)\bigr)+\smashoperator[r]{\sum_{j\not\in B_{b_n(i)}^n}} D^{i,j}_n \nonumber \\
&=K_n^i + \smashoperator[r]{\sum_{l \not\in B_{b_n(i)}^n}} D^{i, j}_n.
\end{align}
Therefore for all integers $i\leq n$, by Taylor expansion there is $\tilde W_n \in [W_n,W_n^i]$ such that the following holds:
\begin{equation}\begin{split}
  \label{eq:mal_aux_jambes-theorem-2}
\abs{\text{(b)}}&\leq \abs[\big]{\EE \bigl(K_n^i\big[H(W)-H(W^i_n)-[W_n-W^i_n]H'(W_n)\big]\bigr)}
\\&\leq \frac{1}{2}\abs[\big]{\EE \bigl(K_n^i[W_n-W^i_n]^2 H''(\tilde W_n)\bigr)}
\\&\oversetclap{(a)}{\leq} \EE\Bigl(\abs{K_n^i}(W_n-W^i_n)^2\Bigr) \\
  &\leq \frac{2}{n^{\frac{3}{2}}}\Bigl\{
    \norm{K_n^i}^3_{L_3}
  + \smashoperator[r]{\sum_{j, k \not\in B_{b_n(i)}^n}} \abs[\big]{\EE\bigl(\abs{K_n^i}D^{i,j}_n D^{i,k}_n\bigr)} \Big\}.
\end{split}\end{equation}
where to get (a) we exploited the fact that $H\in \mathcal{F}$ is in the Stein function class and therefore that $\sup_x |H''(x)|\le 2$.
The following upper bound holds:  $\|K_n^i\|_{L_3}^3\le \calS^3$. Therefore if we  upper bound $\sum_{j, k \not\in B_{b_n(i)}^n} \abs[\Big]{\EE\bigl(\abs{K_n^i}D^{i,j}_n D^{i,k}_n\bigr)}$ we will have successfully upperbounded \cref{eq:mal_aux_jambes-theorem-2}.
To do so we proceed differently depending: (i) if $j \neq k$ or if (ii) $j = k$.
We will first work on all indexes $j,k\in \mathbb{N}, j \neq  k$.
We have:
\begin{equation}\label{eq:rory_1-theorem-2}
\begin{split}
\EE\bigl(\abs{K_n^i}D^{i,j}_n D^{i,k}_n\bigr)
&
\oversetclap{(a)}{=}\EE\bigl(D^{i,k}_n~\bigl\{\abs{K_n^i}D^{i,j}_n -\abs{K^i_n(X^k)} D^{i,j}_n(X^k)\bigr\}\bigr)
\\&=\EE\Bigl(D^{i,k}_n \nabla_k\bigl\{\abs{K_n^i} D^{i,j}_n\bigr\} \Bigr),
\end{split}
\end{equation}
where (a) is a consequence of the following equality
\begin{equation}\begin{split}
&\EE\Bigl(D^{i,k}_n\abs{K^i_n(X^k)} D_n^{i,j}(X^k)\Bigr)
\\&=\EE\Bigl[\EE\Bigl(D^{i,k}_n \mid X_{\dbracket{n} \setminus \{k\}}\Bigr)\times\EE\Bigl(\abs{K^i_n(X^k)}D^{i,j}_n (X^k) \mid X^k\Bigr)\Bigr]
\\&=0.
\end{split}\end{equation}
As we have: $\nabla_k\bigl\{\abs{K_n^i}D^{i,j}_n\}=\abs{K_n^i}\nabla_k\bigl\{D^{i,j}_n\}+D^{i,j}_n(X^k)\nabla_k\bigl\{\abs{K_n^i}\}$, we can write:
\begin{equation}\label{luke_1}
\EE\Bigl(D^{i,k}_n \nabla_k \bigl\{\abs{K_n^i}D^{i,j}_n\} \Bigr)
=\EE\Bigl(\abs{K_n^i} D^{i,k}_n \nabla_k \bigl\{D^{i, j}_n\bigr\}\Bigr)
+\EE\Bigl(D^{i,j}(X^k) D^{i,k}_n \nabla_k\bigl\{\abs{K^i(X)}\bigr\}  \Bigr).
\end{equation}
Using a similar argument than in \cref{eq:rory_1-theorem-2}, we can bound the first term on the right hand side of \cref{luke_1} in the following way:
\begin{meqn}
&\abs[\Big]{\EE\Bigl(\abs{K_n^i} D^{i,k}_n ~ \nabla_k \bigl\{D^{i,j}_n\bigr\}\Bigr)}
\\&=\abs[\Big]{ \EE\Bigl(\nabla_j\bigl\{\abs{K_n^i} D^{i,k}_n\bigr\} ~ \nabla_k \{D^{i,j}_n \} \Bigr)}
\\&= \abs[\Big]{\EE\Bigl(\abs{K_n^i} ~ \nabla_j \{ D^{i,k}_n\} ~ \nabla_k \{ D^{i,j}_n \} \Bigr)
  +\EE\Bigl(D^{i,k}(X^j) ~ \nabla_j \{ \abs{K_n^i} \} ~ \nabla_k \{ D^{i,j}_n \} \Bigr)}
\\&\leq\max_{l\in N_n}\norm[\big]{\calL_n(\tilde{X},f_{l,n}(X_{1:l}))}_{L_3}
    \norm[\big]{\beta_2^{l,n}(X^n)}^2_{L_3}
\\&\quad+ \max_{l\in N_n}\norm[\big]{\beta_{1}^{l,n}\big(X^n\big)}^2_{L_3}
    \norm[\big]{\beta^{l,n}_{1,2}\big(X^n\big)}_{L_3}
\\&\le \frac{1}{n^2}\Bigl\{\calS\epsilon_n(\nabla)^2+\epsilon_n(\Delta)^2\epsilon_n(\nabla)\Bigr\}.
\end{meqn}
Moreover we have:
\begin{equation}
\EE\Bigl(D^{i,j}_n(X^k) D^{i,k}_n \: \nabla_k K^i_n  \Bigr)\le \frac{\epsilon_n(\Delta)^3}{n^{\frac{3}{2}}}.
\end{equation}
This implies that
\begin{equation}\label{eq:luke-theorem-2}
\Big|\EE\bigl( D^{i,k}_n \nabla_k \{ \abs{K_n^i} D^{i,j}_n \} \bigr)\Big|
\leq \frac{\epsilon_n(\Delta)^2\epsilon_n(\nabla)+\calS\epsilon_n(\nabla)^2}{n^2}+\frac{\epsilon_n(\Delta)^3}{n^{\frac{3}{2}}}.
\end{equation}
This establishes a bound for $\Big|\EE\bigl( D^{i,k}_n  \nabla_k \{ \abs{K_n^i} D^{i,j}_n \}  \bigr)\Big|$ if $j\ne k$. 
On the other hand, if $j=k$ then we have:
\begin{equation}
\EE\Bigl(\abs{K_n^i} (D^{i,j}_n)^2\Bigr) \leq \frac{\epsilon_n(\Delta)^2\calS}{n}.
\end{equation}
We therefore obtain that:
\begin{equation}
\abs{\text{(b)}}\leq\frac{2}{\sqrt{n}}\Big[\calS^3+\epsilon_n(\Delta)^2\calS+\epsilon_n(\Delta)^2\epsilon_n(\nabla)+\calS\epsilon_n(\nabla)^2+\sqrt{n}\epsilon_n(\Delta)^3\Big].
\end{equation}  

Now define:
\begin{equation}
  \delta_n \eqdef \frac{2}{\sqrt{n}}\Big[\calS^3+\epsilon_n(\Delta)^2\calS+\epsilon_n(\Delta)^2\epsilon_n(\nabla)+\calS\epsilon_n(\nabla)^2+\sqrt{n}\epsilon_n(\Delta)^3\Big].
\end{equation}
then we have that:
\begin{equation}
  \abs[\Big]{d_W(W_n,\sigma Z)-\sup_{H\in \mathcal{F}} \abs[\big]{\EE\bigr(H'(W_n)K_i^n\big[W_n-W_n^i\big]-\sigma^2H'(W_n)\bigl)} }\leq \delta_n.
\end{equation}
Therefore it suffices to upper bound:
\begin{equation}\label{avancer_1}
\sup_{H\in \mathcal{F}} \abs[\Big]{\EE\Big(H'(W_n)K_i^n\big[W_n-W_n^i\big]-\sigma^2H'(W_n)\Big)}.
\end{equation}
Let us bound the previous quantity:
\begin{align*}
&\abs[\Big]{\EE\Big(H'(W_n)\Bigl[\sum_{i\le n}K_i^n[W_n-W^i_n]-\sigma^2_1\Bigr]\Big)}
\\&\oversetclap{(a)}{\le}\norm[\Big]{\sum_{i\le n}K_i^n[W_n-W^i_n]-\sigma^2_1}_{L_1}
\\&\oversetclap{(b)}{\le}\frac{1}{n}\sum_{i\le n}\norm[\big]{K_n^i \sum_{k\not\in B_{b_n(i)}^n} D^{i,k}_n }_{L_1}
+\Big[\Var\Big(\sum_{i\le n} K_n^i \Big)\Big]^{\frac{1}{2}}
\\&\quad+\frac{1}{n}\sum_{i\le n} \abs[\Big]{\EE\Bigl[\Var \bigl(\calL(X_i,\hat{f}_{b_n(i)})\mid \hat{f}_{b_n(i)}\bigr)\Bigr]-\sigma^2_1}
\\&\leq (d)+(e)+\max_{l \in N_n}\abs[\big]{\sigma_{l,n}^2-\sigma^2_1},
\end{align*}
where (a) is due to the fact that $\abs{H'(u)} \leq 1$ for all $u \in \RR$ and $H \in \calF$,
and where (b) is a consequence of \cref{comfort} and the triangle inequality.

We  start by bounding the term $(d)$. By the Cauchy-Schwarz inequality we have:
\begin{equation}
  \frac{1}{n}\sum_{i\le n}\norm[\big]{K_n^i \smashoperator[r]{\sum_{l\not\in B_{b_n(i)}^n}} D^{i,l}_n}_{L_1}
 \leq \max_{i\le n}\norm{K_n^i}_{L_2}\norm[\big]{\smashoperator[r]{\sum_{l\not\in B_{b_n(i)}^n}} D^{i,l}_n}_{L_2}
\end{equation}
By a similar argument as in \eqref{eq:rory_1-theorem-2}, we can establish that
\begin{equation}
\begin{split}
\norm[\Big]{\smashoperator[r]{\sum_{l\not\in B_{b_n(i)}^n}} D^{i,l}_n}_{L_2}^2
&= \smashoperator[r]{\sum_{j,l\not\in B_{b_n(i)}^n}} \EE\Big(D^{i,l}_n D^{i,j}_n \Big)
\\&= \sum_{j\ne l}\EE\Bigl( (\nabla_j D^{i,l}_n) D^{i,j}_n \Bigr) + \sum_{j \leq n}\EE\bigl[(D^{i,j}_n)^2\bigr]
\\&= \sum_{j\ne l}\EE\Bigl( \nabla_j D^{i,l}_n \: \nabla_l D^{i,j}_n \Bigr) + \sum_{j \leq n} \EE\bigl[(D^{i,j}_n)^2\bigr]
\\&\leq\epsilon_n(\nabla)^2+ \epsilon_n(\Delta)^2.
\end{split}
\end{equation}
Therefore we have that $\abs{(d)}\leq \calS\Big[\epsilon_n(\Delta)+\epsilon_n(\nabla)\Big]$.
Finally, we finish by bounding $(e)$.
By Efron-Stein we have:
\begin{equation}\begin{split}
  (e)^2
  &\leq\frac{1}{n^2}\Var\Big(\sum_{i \leq n} K_n^i\Big)
\\&\leq \frac{1}{n^2}\sum_{j\leq n}\EE \Big[\bigl(\sum_{i\le n} \nabla_j K_n^i \bigr)^2\Big]
\\&= \frac{1}{n^2}\sum_{j\leq n}\sum_{i,k\leq n}\EE\Bigl[ \nabla_j \{K_n^i\} ~ \nabla_j \{K_n^k\} \Bigr]
\\&\leq \frac{1}{n^2}\sum_{j\leq n}\Big[\sum_{i\leq n}\norm[\big]{\nabla_j K_n^i}_{L_2}\Big]^2
\\&\leq \calS^2\epsilon_n(\Delta)^2
\end{split}\end{equation}

Therefore we established that 
\begin{equation}\begin{split}
  d_W(W_n,\sigma Z)&\le \frac{2}{\sqrt{n}}\Big[\calS^3+\epsilon_n(\Delta)^2\calS+\epsilon_n(\Delta)^2\epsilon_n(\nabla)+\calS\epsilon_n(\nabla)^2+\sqrt{n}\epsilon_n(\Delta)^3\Big]
  \\&\quad+\max_{l \in N_n}\abs*{\sigma_{l,n}^2-\sigma^2_1}+ \calS \big[\epsilon_n(\nabla)+2\epsilon(\Delta)\big]
.\end{split}\end{equation}
\end{proof}

\subsection{Proof of theorem~\ref{clt3}}
\begin{proof}
    Our proof is based on Stein's method for central limit theorems.
    We present the main proof in the case of the cross-validated risk.
    As the proof for the split risk is similar, we present it inline marked by a margin delineator on the left.
    For ease of notation we drop the superindex $n$ from the processes: For exemple we write $X$ instead of $X^n$ and $X^i$ instead of $X^{n,i}$.

Let $(F_i)$ denote filtration given by: $F_i \eqdef \sigma(X_1,\dots,X_i)$, and 
let $(W_n), (W^i_n)$ be defined as:
\begin{align}
  W_n(X) &\eqdef \frac{1}{\sqrt{n}}\sum_{j\leq n}\calL(X_j,\hat{f}_{b(j)}(X))-R_{\abs{B_{b_n(j)}}}, \\
  W^i_n(X) &\eqdef \EE(W_n \mid X_{\dbracket{n}\setminus i}).
\end{align}
Similarly as in theorem 2 we abuse notations and  write $W_n \eqdef W_n(X)$ and similarly $W^i_n \eqdef W^i_n(X)$,
and define $\Delta_{i,n} \eqdef  \EE(W_n \mid F_i)-\EE(W_n \mid F_{i-1})$.
Before diving into the proof we note that
\begin{equation}\label{pluie}
  \EE\bigl(W_n \mid F_i\bigr)-\EE\bigl(W_n(X^i) \mid F_i\bigr)
  = \EE\bigl(W_n \mid F_i\bigr)-\EE\bigl(W_n \mid F_{i-1} \bigr)=\Delta_{i,n},
\end{equation}
and
\begin{equation}\label{wind}
  \EE\bigl[W_n \mid X \bigr] - \EE\bigl[W_n(X^i) \mid X\bigr]=W_n-W^i_n.
\end{equation}

We are now ready to dive into the main body of the proof. We recall the
principle lemma of Stein's method. Let $Z \sim \calN(0, 1)$, then we have:
\begin{equation}
  d_W(W_n, \sigma_{\rm{cv}} Z) \leq \sup_{H\in \mathcal{F}} \abs[\big]{\EE\bigl(W_n H(W_n)-\sigma^2_{\rm{cv}}H'(W_n)\bigr)}.
\end{equation}
 For any function $H\in \mathcal{F}$ this can be re-expressed as:
\begin{equation}\begin{split}
  &\label{eq:triste-theorem3}\EE\Big(W_nH(W_n)-\sigma^2_{\rm{cv}}H'(W_n)\Big) \\
  &=\sum_{i \leq n}\Big\{\EE\Big(\Delta_{i,n}H(W^i_n)\Big)+ \EE\Big(\Delta_{i,n}\big[H(W_n)-H(W^i_n)-[W_n-W^i_n]H'(W_n)\big]\Big)\\
  &\qquad +\mathbb{E}\Big(\Delta_{i,n}[W_n-W^i_n]H'(W_n)-\sigma^2_{\rm{cv}}H'(W_n)\Big)\Big\}
\end{split}\end{equation}

We will bound each term and show that it converges to zero. 
Similarly as in the proof of \cref{clt2}, for all integer $i\le n$ we have:
\begin{align*}
\EE\Bigl[\Delta_{i,n}H(W^i_n)\Bigr]
&=\EE\Bigl[\EE\bigl(\Delta_{i,n}H(W^i_n) \mid X_{\dbracket{n}\setminus \{i\}}\bigr)\Bigr]
\\&\oversetclap{(a)}{=}\EE\Bigl[H(W^i_n)\EE\bigl(\Delta_{i,n} \mid X_{\dbracket{n}\setminus \{i\}}\bigr)\Bigr]
\\&\oversetclap{(b)}{=}0,
\end{align*}
where to get (a) we used the fact that $W_{n}^i$ is measurable with respect to $\sigma\bigl(X_{\dbracket{n} \setminus \{i\}}\bigr)$
and to get (b)  that $\EE\bigl[\Delta_{i,n}\mid X_{\dbracket{n}\setminus \{i\}}\bigr]=0$.
Having bounded the first term of \cref{eq:triste-theorem3} we now wish to upper-bound
$\abs[\big]{\EE\bigl(\Delta_{i,n}\big[H(W)-H(W^i_n)-[W_n-W^i_n]H'(W_n)\big]\bigr)}$.
Let us write:
\begin{equation}\begin{split}
D^{i,l}_n(X) &\coloneqq \calL_n(X_l,\hat{f}_{b_n(l)}(X))-\calL_n(X_l,\hat{f}_{b_n(l)}(X^i))
  -\Big[R_n\big(\hat{f}_l(X^{n})\big)-R_n\big(\hat{f}_l(X^i)\big)\Big],\\
  K_n^i(X) &\eqdef \EE\Bigl[\calL_n(X_i,\hat{f}_{b_n(i)}(X)) \mid F_i\Bigr]
    - \EE\Bigl[\calL_n(X_i,\hat{f}_{b_n(i)}(X)) \mid F_{i-1}\Bigr], \\
  \calR^n_i(X) &\eqdef \sum_{k \neq b_n(i)} \abs{B_k} \EE\bigl[ R_n(\hat{f}_k(X)) \mid F_i \bigr]
   - \sum_{k \neq b_n(i)} \abs{B_k} \EE\bigl[ R_n(\hat{f}_k(X)) \mid F_{i - 1} \bigr].
\end{split}\end{equation}
We remark that:
\begin{align}
  \sqrt{n}W_n = \sum_{i\le n}\Bigl(\calL_n(X_i,\hat{f}_{b_n(i)})-R(\hat{f}_{b_n(i)})
    +R(\hat{f}_{b_n(i)})-R_{n-|B^n_{b_n(i)}|}\Bigr).
 \end{align}
Using the previously established fact that $\Delta_{i,n}=\EE\big[W_n-W_{n}(X^i) \mid F_i\big]$ we obtain that:
\begin{equation}
 \begin{split}\label{007}
  \sqrt{n}\Delta_{i,n}
  &=K_n^i + \calR^n_i+ \smashoperator[r]{\sum_{l\not\in B_{b_n(i)}^n}} \EE\bigl[D^{i,l}_n \mid F_i\bigr],
\end{split}
\end{equation}
from which we deduce that:
\begin{equation}
 \begin{split}\label{uk}
  \sqrt{n}\nabla_i W_n
  &=\calL_n(X_i,\hat{f}_{b_n(i)}) -\calL_n(X'_i,\hat{f}_{b_n(i)})
    + \smashoperator[r]{\sum_{k\ne b_n(i)}} \abs{B_k}\nabla_i R_n(\hat{f}_k) + \smashoperator[r]{\sum_{l\not\in B_{b_n(i)}^n}}D^{i,l}_n .
\end{split}
\end{equation}

\begin{proofaside}
In the case of the split risk, we obtain different formulation for $\Delta_{i,n}$ depending on which block the $i$th observation belongs to.
Indeed we have that:
\begin{equation}
    \sqrt{n} \Delta_{i, n} = \begin{dcases}
         K^i_n & \text{if } i \in B_1, \\
         \calR^i_n + \sum_{\mathclap{l \not\in B^n_{b_n(i)}}} \EE\bigl[D^{i,l}_n \mid F_i\bigr] & \text{otherwise,}
    \end{dcases}
\end{equation}
and similarly:
\begin{equation}
    \sqrt{n}\bigl(W_n - W_n(X^i)\bigr) = \begin{dcases}
        \calL(X_i, \hat{f}_{b_n(i)}) - \calL(X_i', \hat{f}_{b_n(i)}) & \text{if } i \in B_1, \\
        \smashoperator[r]{\sum_{k \neq b_n(i)}} \abs{B_k}\bigl(\nabla_i R(\hat{f}_k)\bigr) + \sum_{\mathclap{l \not\in B^n_{b_n(i)}}} D^{i,l}_n & \text{otherwise.}
    \end{dcases}
\end{equation}
\end{proofaside}

Our goal is to bound $\abs[\Big]{\EE\Big(\Delta_{i,n}\big[H(W)-H(W^i_n)-[W_n-W^i_n]H'(W^i_n)\big]\Big)}$.
To do so we use two version of Taylor's inequality.
Let $f\in C^2$ defined on the reals, then for all $x\in\RR$ the following holds:
\begin{align}
  \abs[\big]{f(x+h)-f(x)-hf'(x)} &\leq \frac{h^2}{2}\sup_t \abs[\big]{f''(t)}, \\
  \abs[\big]{f(x+h)-f(x)-hf'(x)} &\leq 2 h\sup_t \abs[\big]{f'(t)}.
 \end{align}
Using those inequalities and \cref{007} we obtain that:
\begin{meqn}\label{eq:mal_aux_jambes-theorem-3}
  &n^{\frac{3}{2}} \abs[\Big]{\EE\Big(\Delta_{i,n}\big[H(W)-H(W^i_n)-(W_n-W^i_n)H'(W^i_n)\big]\Big)}
  \\&\le n \abs[\Big]{\EE\Bigl[ \bigl( K_i^n + \calR^n_i+ \sum_{\mathclap{l\not\in B_{b_n(i)}^n}} \EE\bigl[D^{i,l}_n \mid F_i\bigr]\bigr) \times \bigl(H(W)-H(W^i_n)-(W_n-W^i_n)H'(W^i_n)\bigr)\Bigr]}
\\&\leq  n\EE\Bigl[\abs{K_n^i}\bigl(\nabla_i W_n\bigr)^2\Bigr]
+2n\EE\Bigl[\abs[\big]{\nabla_i W_n} \abs[\big]{\smashoperator[r]{\sum_{l\not\in B_{b_n(i)}^n}} \EE\bigl[D^{i,l}_n \mid F_i\bigr]} \Bigr]
    +n\EE\Bigl[\abs{\calR_i^n} \bigl(\nabla_i W_n\bigr)^2\Bigr]
\\&\le (a_i)+(b_i)+(c_i).
\end{meqn}

\begin{proofaside}
In the case of the train-test split risk this inequality simplifies in the following way:
\begin{meqn}
  &\bigl(\frac{n}{K_n}\bigr)^{\frac{3}{2}} \abs[\Big]{\EE\Big(\Delta_{i,n}\big[H(W)-H(W^i_n)-(W_n-W^i_n)H'(W^i_n)\big]\Big)}
\\&\leq
\mathbb{I}(i\not\in B_1)\frac{n}{K_n}\Bigl\{2\EE\Bigl( \abs[\big]{\nabla_i W_n} \abs[\big]{\sum_{\mathclap{l\not\in B_{b_n(i)}^n}} \EE \bigl[D^{i,l}_n \mid F_i\bigr]}  \Bigr)
+\EE\Bigl(\abs{\calR_i^n} \bigl(\nabla_i W_n\bigr)^2\Bigr)\Bigr\}
\\&\quad +\mathbb{I}\Bigl(i\in B_1) \frac{n}{K_n}\EE\Big(\abs{K_n^i}\bigl(\nabla_i W_n\bigr)^2\Bigr)
\\&\le (\tilde b_i)+(\tilde c_i)+(\tilde a_i).
\end{meqn}
\end{proofaside}

We bound each term in succession.
First by using the inequality $(x_1+x_2+x_3)^2\leq 3 (x_1^2+x_2^2+x_3^2)$ and \cref{uk} we obtain that:
\begin{meqn}
  (a_i)
  &\le3 \norm{K_n^i}_{L_3}^3
  +3\sum_{j,k} \EE\Bigl[\abs{K_n^i} D^{i,j}_n D^{i,k}_n \Bigr]
\\&\quad
+3\sum_{j,k} \abs{B_j} \abs{B_k} \EE \Bigl[\abs{K_n^i} \nabla_i R_n(\hat{f}_j) \:  \nabla_i R_n(\hat{f}_k) \Bigr].
\label{worthless}
\end{meqn}
\begin{proofaside}
In the case of the split risk we have:
\begin{equation}
(\tilde a_i)\le \mathbb{I}(i\in B_1)\calS^3.
\end{equation}
\end{proofaside}
The first term of \cref{worthless} is easily bounded by noting that $\norm{K_n^i}_{L_3}^3\le \calS^3$.
We want to bound the second term.
To do so we bound each $\EE\bigl(|K_n^i|D_n^{i,j}D_n^{i,k}\bigr)$ term.
We proceed differently when $k=j$ than when $k\ne j$.
We will first work on all indexes $j,k\in \mathbb{N}, j \neq  k$:
\begin{meqn}
\label{eq:rory_1-theorem-3}
\EE\bigl(\abs{K_n^i}D^{i,j}_n D^{i,k}_n\bigr)
&\oversetclap{(a)}{=}\EE\bigl(D^{i,k}_n~\big[\abs{K_n^i}D^{i,j}_n ~-\abs{K^i_n(X^k)} D^{i,j}_n(X^k)\big]\bigr)
\\&=\EE\Bigl(D^{i,k}_n \nabla_k\bigl\{\abs{K^i_n} D^{i,j}_n\bigr\} \Bigr),
\end{meqn}
where (a) is a consequence of the following equality
\begin{meqn}
&\EE\Bigl(D^{i,k}_n\abs{K_n^i(X^k)} D^{i,j}_n(X^k)\Bigr)
\\&=\EE\Bigl[\EE\Bigl(D^{i,k}_n \mid X_{\dbracket{n} \setminus \{k\}}\Bigr)\times\EE\Bigl(\abs{K^i_n(X^k)}D^{i,j}_n (X^k) \mid X^k\Bigr)\Bigr]
\\&=0.
\end{meqn}
Now, from the equality:
\begin{equation}
\nabla_k\bigl\{\abs{K^i_n}D^{i,j}_n\}=\abs{K_n^i}\nabla_k\bigl\{D^{i,j}_n\}+D^{i,j}_n(X^k)\nabla_k\bigl\{\abs{K^i_n}\},
\end{equation}
we may deduce that:
\begin{equation}\label{eq:luke_1-theorem-3}
\EE\Bigl(D^{i,k}_n \nabla_k \bigl\{\abs{K_n^i}D^{i,j}_n\} \Bigr)
=\EE\Bigl(\abs{K^i_n} D^{i,k}_n \nabla_k \bigl\{D^{i, j}_n\bigr\}\Bigr)
+\EE\Bigl(D^{i,j}(X^k) D^{i,k}_n \nabla_k\bigl\{\abs{K^i_n(X)}\bigr\}  \Bigr).
\end{equation}
Using a similar argument than in \cref{eq:rory_1-theorem-3}, we can bound the first term on the right hand side of \cref{eq:luke_1-theorem-3} in the following way:
\begin{meqn}
&\abs[\Big]{\EE\bigl(\abs{K_n^i} D^{i,k}_n ~ \nabla_k \bigl\{D^{i,j}_n\bigr\}\bigr)}
\\&=\abs[\Big]{ \EE\Bigl(\nabla_j\bigl\{\abs{K_n^i} D^{i,k}_n\bigr\} ~ \nabla_k \{D^{i,j}_n \} \Bigr)}
\\&= \abs[\Big]{\EE\Bigl(\abs{K_n^i} ~ \nabla_j \{ D^{i,k}_n\} ~ \nabla_k \{ D^{i,j}_n \} \Bigr)
  +\EE\Bigl(D^{i,k}(X^j) ~ \nabla_j \{ \abs{K_n^i} \} ~ \nabla_k \{ D^{i,j}_n \} \Bigr)}
\\&\leq\max_{l\in N_n}\norm[\big]{\calL_n(\tilde{X},f_{l,n}(X_{1:l}))}_{L_3}
    \norm[\big]{\beta_2^{l,n}\big(X^n\big)}^2_{L_3}
+ \max_{l\in N_n}\norm[\big]{\beta_1^{l,n}(X^n)}^2_{L_3}
    \norm[\big]{\beta_2^{l,n}}_{L_3}
\\&\le \frac{\calS\epsilon_n(\nabla)^2+\epsilon_n(\Delta)^2\epsilon_n(\nabla)}{n^2}.
\end{meqn}
Moreover, we have that $\EE\Bigl(D^{i,j}(X^k) D^{i,k}_n \nabla_k\bigl\{\abs{K^i(X)}\bigr\}  \Bigr) \leq n^{-3/2} \epsilon_n(\Delta)^3$.
This establishes the following bound for $j \neq k$:
\begin{equation}\label{eq:luke-theorem-3}
\abs[\Big]{\EE\bigl( D^{i,k}_n \nabla_k \{ \abs{K_n^i} D^{i,j}_n \} \bigr)}
\leq
\frac{\epsilon_n(\Delta)^2\epsilon_n(\nabla)+\calS \epsilon_n(\nabla)^2}{n^2}+\frac{\epsilon_n(\Delta)^3}{n^{\frac{3}{2}}}.
\end{equation}
On the other hand, if $j=k$ then we have $\EE\Bigl(\abs{K_n^i} (D^{i,j}_n)^2\Bigr) \leq \frac{\epsilon_n(\Delta)^2\calS}{n}$.
Moreover, the last term of \cref{worthless} can be bounded using Cauchy-Schwarz inequality:
\begin{equation}\begin{split}
  &\sum_{j,k}\abs{B_k}\abs{B_j} \EE \Bigl[\abs{K^i_n} \: \nabla_i R_n(\hat{f}_j) \: \nabla_i R_n(\hat{f}_k)\Bigr]
\le \calS\times\calS(R)^2.
\end{split}\end{equation}
Putting everything together, we obtain the following bound for $(a_i)$:
\begin{equation}\begin{split}&
(a_i)\le \calS\times\calS(R)^2+ \calS^3+  \epsilon_n(\Delta)^2\epsilon_n(\nabla)+\calS\epsilon_n(\nabla)^2+\sqrt{n}\epsilon_n(\Delta)^3+\epsilon_n(\Delta)^2\calS.
\end{split}\end{equation}
We now focus on the term $(b_i)$. By the Cauchy-Schwarz inequality we have that:
\begin{meqn}
  (b_i)
  &\leq n~\EE\Bigl[\abs[\big]{ \smashoperator[r]{\sum_{l\not\in B_{b_n(i)}^n}} D^{i,l}_n } \times \abs[\big]{W_n-W^i_n} \Bigr]
\\&\leq n~\norm[\Big]{\smashoperator[r]{\sum_{l\not\in B_{b_n(i)}^n}} D^{i,l}_n }_{L_2} \norm[\Big]{W_n-W^i_n}_{L_2}
\\&\oversetclap{(a)}{\le} \sqrt{n} \norm[\big]{\smashoperator[r]{\sum_{\quad l\not\in B_{b_n(i)}^n} }D^{i,l}_n}_{L_2}
\Bigl\{ \norm[\big]{ \smashoperator[r]{\sum_{l\not\in B_{b_n(i)}^n}} D^{i,l}_n }_{L_2}
  + \norm[\big]{K^i_n}_{L_2}
  +\norm[\big]{\sum_{k} \abs{B_k} \nabla_i R_n(\hat{f}_k)}_{L_2}\Bigr\},
  \label{os}
\end{meqn}
where to get (a) we used \cref{uk}.
Moreover we have that:
\begin{meqn}
  \norm[\Big]{\smashoperator[r]{\sum_{l\not\in B_{b_n(i)}^n}} D^{i,l}_n}_{L_2}^2
&\le \smashoperator[r]{\sum_{j,l\not\in B_{b_n(i)}^n}} \EE\Bigl[D^{i,l}_n D^{i,j}_n\Bigr]
\\&\le \sum_{j\ne l} \EE\Bigl[\bigl(\nabla_j D^{i,l}_n\bigr) \: D^{i,j}_n\Bigr] + \sum_{j}\EE\Bigl[\bigl(D^{i,j}_n\bigr)^2\Bigr]
\\&\le \sum_{j\ne l}\EE\Bigl[\nabla_j D^{i,l}_n \: \nabla_l D^{i,j}_n \Bigr] + \epsilon_n(\Delta)^2
\\&\le\epsilon_n(\nabla)^2+ \epsilon_n(\Delta)^2.
\end{meqn}
Therefore using \cref{os} we obtain:
\begin{equation}
  (b_i) \le \sqrt{n}\Bigl(\epsilon_n(\nabla)+ \epsilon_n(\Delta)\Bigr) \Bigl(\epsilon_n(\nabla)+ \epsilon_n(\Delta)+\calS+\calS(R)\Bigr).
\end{equation}
\begin{proofaside}
In the case of the split risk we obtain:
\begin{equation}
(\tilde{b}_i) \le \sqrt{\frac{n}{K_n^3}}\Bigl(\epsilon_n(\Delta)+\frac{\epsilon_n(\nabla)}{\sqrt{K_n}}\Bigr)
\Bigl(\frac{\epsilon_n(\nabla)+\calS(R)}{\sqrt{K_n}}+\epsilon_n(\nabla)\Bigr) + o(\frac{1}{n}),
\end{equation}
where the term $o(n^{-1})$ is a correction term coming from:
$\abs[\big]{\frac{|B_1^n|}{n}-\frac{1}{K_n}} \le \frac{1}{n}$. 
\end{proofaside}

Finally, the term $(c_i)$ is bounded using \cref{uk} and
the Cauchy-Schwarz inequality in the following fashion:
\begin{meqn}
  (c_i)
  &\le  \calS(R)\norm[\big]{W_n-W_n(X^{i})}_{L_2}^2
\\&\le  \calS(R)\times \Big[\epsilon_n(\nabla)+ \epsilon_n(\Delta) + \calS + \calS(R)\Big]^2.
\end{meqn}

\begin{proofaside}
In the case of the split risk we obtain:
\begin{equation}
(\tilde{c_i})\le  \frac{ \calS(R)}{K_n^{\frac{3}{2}}}\Bigl[\frac{\epsilon_n(\nabla)+\calS(R)}{\sqrt{K_n}}+\epsilon_n(\Delta)\Big]+o(\frac{1}{n}),
\end{equation}
where the $o(n^{-1})$ is a correction term coming from: $\abs[\big]{\frac{|B_1^n|}{n}-\frac{1}{K_n}}\le \frac{1}{n}$.
\end{proofaside}
This implies that:
\begin{equation}\begin{split}\label{mal_aux_jambes_sol}
  &n^{-\frac{3}{2}} \sum_{i\le n}\Big|\mathbb{E}\Big(\Delta_{i,n}\big[H(W)-H(W^i_n)-[W_n-W^i_n]H'(W^i_n)\big]\Big)\Big|
\\&\le \epsilon_n(\Delta)^3+\Big[\epsilon_n(\nabla)+ \epsilon_n(\Delta)\Big]\times \Big[\epsilon_n(\nabla)+ \epsilon_n(\Delta)+\calS+\calS(R)\Big]
\\&\quad+\frac{1}{\sqrt{n}}\Big\{  \epsilon_n(\Delta)^2\epsilon_n(\nabla)+ \big[\calS(R)+\calS\big]\times \Big[\epsilon_n(\nabla)+ \epsilon_n(\Delta)+\calS+\calS(R)\Big]^2\Big\} 
\end{split}\end{equation}
\begin{proofaside}
In the case of the split risk we obtain
\begin{equation}\begin{split}
  &\sum_{i\le n}\Big|\EE\Big(\Delta_{i,n}\big[H(W)-H(W^i_n)-[W_n-W^i_n]H'(W^i_n)\big]\Big)\Big|
\\&\le \frac{K_n^{\frac{1}{2}}}{n^{\frac{1}{2}}}\calS^3+ \Big[\epsilon_n(\Delta)+\frac{\epsilon_n(\nabla)}{\sqrt{K_n}}+\frac{\calS(R)}{\sqrt{n}}\Big]\Big[\epsilon_n(\Delta)+\frac{\epsilon_n(\nabla)+\calS(R)}{\sqrt{K_n}}\Big]+o(\frac{1}{n}).
\end{split}\end{equation}
\end{proofaside}
Therefore if we denote \begin{equation}\begin{split}
&\delta_n \eqdef \epsilon_n(\Delta)^3+\Big[\epsilon_n(\nabla)+ \epsilon_n(\Delta)\Big]\times \Big[\epsilon_n(\nabla)+ \epsilon_n(\Delta)+\calS+\calS(R)\Big]\\&\qquad+\frac{1}{\sqrt{n}}\Big\{  \epsilon_n(\Delta)^2\epsilon_n(\nabla)+ \big[\calS(R)+\calS\big]\times \Big[\epsilon_n(\nabla)+ \epsilon_n(\Delta)+\calS+\calS(R)\Big]^2\Big\},
\end{split}\end{equation}
then  we know that:
\begin{equation}\begin{split}&
\abs[\Big]{d_W(W_n,\sigma Z)-\sup_{H\in \mathcal{F}} \abs[\big]{\EE\bigl(H'(W_n)\Delta_i^n\big[W_n-W_n^i\big]-\sigma^2_{\rm{cv}}H'(W_n)\bigr)}}\le \delta_n.
\end{split}\end{equation}
\begin{proofaside}
In the case of the split risk we write instead:
\begin{equation}\begin{split}&
\delta_n \eqdef \frac{K_n^{\frac{1}{2}}}{n^{\frac{1}{2}}}\calS^3+ \Big[\epsilon_n(\Delta)+\frac{\epsilon_n(\nabla)+\calS(R)}{\sqrt{K_n}}\Big]^{2}+o(\frac{1}{n}),
\end{split}\end{equation}
then we know that
\begin{equation}\begin{split}&
\abs[\Big]{d_W(W_n,\sigma Z) - \sup_{H\in \calF} \abs[\big]{\EE\bigl[H'(W_n)\Delta_i^n\big[W_n-W_n^i\big]-\sigma^2_{\rm{split}}H'(W_n)\bigr]}}\le \delta_n.
\end{split}\end{equation}
\end{proofaside}
Therefore, it remains for us to upper bound:
\begin{equation}
    \sup_{H\in \mathcal{F}}\abs[\Big]{\EE\Big(H'(W_n)\Delta_{i,n}\big[W_n-W_n^i\big]-\sigma^2_{\rm{cv}}H'(W_n)\Big)}
\end{equation}
Towards this goal, we denote:
\begin{equation}
    g_{n,i}(X) \eqdef \calL(X_i,\hat{f}_{b_n(i)}) 
    +\sum_{j\ne{b_n(i)}}|B_j^n|R_n(\hat{f}_{j}).
\end{equation} 
To upper bound the desired quantity it suffices to upper-bound:
\begin{equation}\begin{split}&
\abs[\big]{\EE \Big(H'(W)\Big[\sum_{i\le n} \Delta_{i,n} [W_n-W^i_n] - \sigma^2_{\rm{cv}}\Big]\Big)}
\\&
\oversetclap{(a)}{\le} \norm[\Big]{\sum_{i\le n} \Delta_{i,n} [W_n-W^i_n] - \sigma^2_{\rm{cv}}}_{L_1}
\\&
\le \norm[\Big]{\sum_{i\le n}\Delta_{i,n}[W_n-W^i_n]-\frac{1}{n}\EE \Big[\Cov \Bigl(g_{n,i}, \, \EE \big(g_{n,i} \mid F_i\big) \Bigm| X_{ \dbracket{n} \setminus \{i\}}\Bigr)\Big]}_{L_1}
\\&\quad + \frac{1}{n}\norm[\Big]{\sum_{i \leq n}\EE \bigl[\Cov \big(g_{n,i}, ~\EE\big(g_{n,i} \mid F_i\big) \bigm| X_{ \dbracket{n} \setminus \{i\}}\big)\bigr]-\sigma^2_{\rm{cv}}}_{L_1}
\\&\le (d)+(e),
\end{split}\end{equation}
where to get (a) we used the fact that $H$ belongs to the Stein function class, which implies that $\sup_{x\in \mathbb{R}}\big|H'(x)\big|\le 1$.

\begin{proofaside}
Respectively for the split risk we have:
\begin{meqn}
&\abs[\big]{\EE \Big(H'(W)\Big[\sum_{i\le n} \Delta_{i,n} [W_n-W^i_n] - \sigma^2_{\rm{split}}\Big]\Big)}
\\&
\oversetclap{(a)}{\le} \norm[\Big]{\sum_{i\le n} \Delta_{i,n} [W_n-W^i_n] - \sigma^2_{\rm{split}}}_{L_1}
\\&
\le \norm[\Big]{\sum_{i\le n}\Delta_{i,n}[W_n-W^i_n]-\frac{1}{n}\EE \Big[\Cov \Bigl(g_{n,i}, \, \EE \big(g_{n,i} \mid F_i\big) \Bigm| X_{ \dbracket{n} \setminus \{i\}}\Bigr)\Big]}_{L_1}
\\&\quad + \frac{K_n}{n}\norm[\Big]{\sum_{i \leq n}\EE \bigl[\Cov \big(g_{n,i}, \, \EE\big(g_{n,i} \mid F_i\big) \bigm| X_{ \dbracket{n} \setminus \{i\}}\big)\bigr]-\sigma^2_{\rm{split}}}_{L_1}
\\&\le (d)+(e).
\end{meqn}
\end{proofaside}

We will bound each term successively.
In this goal, we notice that:
\begin{equation}
    \nabla_i W_n \eqdef \frac{1}{\sqrt{n}}\Bigl(\nabla_i g_{n,i} + \sum_{\mathclap{j\not \in B_{b_n(i)}^n}} D^{i,j}_n\Bigr),
\end{equation}
and therefore by using \cref{pluie} and \cref{wind} we obtain that:
\begin{equation}\label{cambridge}
 \Delta_{i,n}\bigl(W_n-W^{n,i}\bigr)
 = \frac{1}{n}\EE\bigl[\nabla_i g_{n,i} + \sum_{\mathclap{j \not\in B_{b_n(i)}^n}} D^{i,j}_n \mid F_i\bigr]
 \times \EE\bigl[\nabla_i g_{n,i} + \sum_{\mathclap{j \not\in B_{b_n(i)}^n}} D^{i,j}_n \mid X\bigr].
\end{equation}

We first start by bounding the term $(d)$ by using \cref{cambridge}:
\begin{equation}\begin{split}
&
\Big\|\sum_{i\le n} \Delta_{i,n}[W_n-W^i_n]
-\frac{1}{n}\EE\Bigl[ \Cov \big(g_{n,i}, \, \EE\big(g_{n,i} \mid F_i\big)
\mid X_{ \dbracket{n} \setminus \{i\}} \bigl)\Bigr]\Big\|_{L_1}
\\
&\le\frac{1}{n}\Big\|\sum_{i\le n}\EE\bigl[\nabla_i g_{n,i} \mid F_i\bigr]
\times \EE\bigl[\nabla_i g_{n,i} \mid X \bigr]
\\&
\qquad\quad-\EE \Big[\Cov \big(g_{n,i}, \, \EE\big(g_{n,i} \mid F_i\big)
    \Bigm| X_{ \dbracket{n} \setminus \{i\}}\big)\Big]\Big\|_{L_1}
\\
&\quad+2 \max_i\norm[\Big]{\nabla_i g_{n,i}}_{L_2}\Big\|\smashoperator[r]{\sum_{j\not\in B_{b_n(i)}^n}} D^{i,j}_n \Big\|_{L_2}
+\max_i\Big\|\smashoperator[r]{\sum_{j\not\in B_{b_n(i)}^n}} D^{i,j}_n \Big\|^2_{L_2}
\\&\le (d_1)+(d_2)+(d_3).
\end{split}\end{equation}

We bound each term successively. 
We have established that:
\begin{align}
\max_i\norm[\Big]{\smashoperator[r]{\sum_{j\not\in B_{b_n(i)}^n}} D^{i,j}_n}^2_{L_2} &\le \epsilon_n(\Delta)^2+\epsilon_n(\nabla)^2, \\
\max_i\Big\| \nabla_i g_{n,i} \Big\|_{L_2} &\le \calS+ \calS(R).
\end{align}
Therefore we obtain that:
\begin{equation}
(d_2)+(d_3)\le \bigl(\epsilon_n(\Delta)+\epsilon_n(\nabla)\bigr) \times \bigl(\epsilon_n(\Delta)+\epsilon_n(\nabla)+2\calS+ 2\calS(R)\bigr).
\end{equation}
\begin{proofaside}
In the case of train-test split risk this becomes instead:
\begin{equation}
(d_2)+(d_3) \le \Bigl(\epsilon_n(\Delta)+\frac{\epsilon_n(\nabla)}{\sqrt{K_n}}\Bigr)\Bigl(\epsilon_n(\nabla)+\frac{\epsilon_n(\nabla)+2\calS(R)}{\sqrt{K_n}}\Bigr)+o(\frac{1}{n}).
\end{equation}
\end{proofaside}
Remains to control the term $(d_1)$. To this goal, we remark that:
\begin{meqn}
&\EE\Bigl[\EE\bigl[\nabla_i(g_{n,i}) \mid F_i\bigr] \times \EE\bigl[\nabla_i(g_{n,i}) \mid X\bigr]\Bigr]
\\&\oversetclap{(a)}{=}\EE\Big[\EE\Big(\EE\big[\nabla_i g_{n,i} \mid F_i\big]\times \EE\big[\nabla_i g_{n,i} \mid X\big]\Bigm| X_{ \dbracket{n} \setminus \{i\}}\Big)\Big]
\\&\oversetclap{(b)}{=}\EE\Bigl[\EE\Bigl(\EE\big[g_{n,i} \mid F_i\big]\times g_{n,i} \bigm| X_{ \dbracket{n} \setminus \{i\}}\Bigr)\Bigr]
\\&\qquad - \EE\Bigl[\EE\Bigl(\EE\big[g_{n,i} \mid X_{\dbracket{n} \setminus \{i\}}\big]\times \EE\bigl[g_{n,i} \mid F_{i-1}\bigr] \bigm| X_{ \dbracket{n} \setminus \{i\}}\Bigr)\Bigr]
\\&=\EE\Bigl[\Cov\Bigl(g_{n,i}, \, \EE\bigl(g_{n,i} \mid F_i\bigr) \bigm| X_{ \dbracket{n} \setminus \{i\}}\Bigr)\Bigr],
\end{meqn}
where (a) is a result from the tower property and (b) is a result of the bilinearity of the covariance. 
Therefore if we write
\begin{equation}
    C_{n,i}(X) \eqdef \EE\bigl[ \nabla_i g_{n,i}(X) \mid F_i\bigr]\times \EE\bigl[\nabla_i g_{n,i}(X) \mid X\bigr]
\end{equation}
we obtain that
$(d_1)\le \sqrt{\Var\big(\frac{1}{n}\sum_{i\le n}C_{n,i}(X)\big)}$.
By the Efron-Stein and the triangular inequality we have:
\begin{meqn}
(d_1)^2 &\le\frac{1}{n^2}\Var\Big(\sum_{i\le n}C_{n,i}(X)\Big)
\\&\le\frac{1}{n^2}\sum_{j\le n}\EE\Bigl[\Bigl(\sum_{i\le n}C_{n,i}(X)-C_{n,i}(X^j)\Bigr)^2\Bigr]
\\&\le\frac{1}{n^2}\sum_{j\le n}\sum_{i,k\le n}\EE\Bigl[\Bigl(\sum_{i\le n} C_{n,i}(X) - C_{n,i}(X^j) \Bigr)\Bigl(\sum_{k\le n}C_{n,k}(X)-C_{n,k}(X^j)\Bigr)\Bigr]
\\&\le\frac{1}{n^2}\sum_{j\le n}\Bigl(\sum_{i\le n} \norm[\big]{C_{n,i}(X)-C_{n,i}(X^j)}_{L_2}\Bigr)^2
\\&\le\bigl(\calS+\calS(R)\bigr)^2\times \bigl(\epsilon_n(\Delta)+\epsilon_n(R)\bigr)^2 + \frac{1}{n}\bigl(\calS+\calS(R)\bigr)^4.
\label{grace}
\end{meqn}
\begin{proofaside}
In the case of train-test split risk this is slightly more delicate.
To start we can remark that:
\begin{meqn}
    (d_1) &\le \frac{K_n}{n}\sqrt{\Var\Bigl[\smash{\sum_{i\in B^n_1} C_{n,i}} \Bigr]}
        + \frac{K_n}{n}\sqrt{\Var\Bigl[\smash{\sum_{i\not\in B^n_1} C_{n,i}}\Bigr]}
        \phantom{\sum_{i \in B_1^n}} \\ 
    &\le\frac{K_n}{n}\sqrt{\EE\Big[\Var\Bigl[\smash{\sum_{i\in B^n_1} C_{n,i} \mid \hat{f}_1}\Bigr]\Big]}
        +\frac{K_n}{n}\sqrt{\Var\Bigl[\smash{\sum_{i\in B^n_1} \EE\bigl[C_{n,i} \mid \hat{f}_1\bigr]}\Bigr]}
        \phantom{\sum_{i \in B_1^n}} \\
    &\qquad+\frac{K_n}{n}\sqrt{\Var \Bigl[\smash{\sum_{i\not\in B^n_1} C_{n,i}}\Bigr]}.
\end{meqn}
We bound each term separately.
By conditional independence we have that:
\begin{equation}
\frac{K_n}{n}\sqrt{\EE\Big[\Var\Bigl[\smash{\sum_{i\in B^n_1} C_{n,i}} \mid \hat{f}_1\Bigr]\Bigr]}
\le \frac{K_n}{n}\sqrt{|B_n^1|}\calS^2
\le \Bigl(\sqrt{\frac{K_n}{n}}+\frac{K_n}{n}\Bigr)\calS^2.
\end{equation}
Moreover using the Efron-Stein inequality, similarly as in \cref{grace} we obtain that:
\begin{meqn}
\frac{K_n}{n}\sqrt{\Var\Bigl[\smash{\sum_{i\in B^n_1} \EE\bigl[C_{n,i} \mid \hat{f}_1\bigr]}\Big]}
&\le\frac{K_n}{n}\sqrt{|B_1^n|\Var \Bigl[ R_n(\hat{f}_1)\Bigr]}
\\
&\le\frac{K_n}{n}\sqrt{|B_1^n|}\calS\times\calS(R)
\\
&\le \Big[\sqrt{\frac{K_n}{n}}+\frac{K_n}{n}\Bigl]\calS\times\calS(R).
\end{meqn}
Furthermore we obtain that:
\begin{meqn}
\frac{K_n}{n}\sqrt{\Var\Bigl[\smash{\sum_{i\not\in B^n_1} C_{n,i}} \Bigr]}
&\le \frac{K_n}{n}\sqrt{\Var\Bigl[\smash{\sum_{i\not\in B^n_1} C_{n,i}}\Bigr]} \phantom{\sum_{i \in B_1^n}}
\\&\le \frac{1}{K_n}\epsilon_n(R)\calS(R)
\end{meqn}
Therefore we obtain:
\begin{meqn}
(d_1)\le& \Big[\sqrt{\frac{K_n}{n}}+\frac{K_n}{n}\Bigl]\calS\big[\calS(R)+\calS\big]+ \frac{4}{K_n}\epsilon_n(R)\calS(R).
\end{meqn}
\end{proofaside}
Putting everything together, we get:
\begin{meqn}
(d)&\le \Big[\calS + \calS(R)\Big]\times \Big[\epsilon_n(\Delta) + \epsilon_n(R)\Big] +\frac{1}{\sqrt{n}}\Big[\calS+ \calS(R)\Big]^2 \\
&\quad+ \Big[\epsilon_n(\Delta)+\epsilon_n(\nabla)\Big]\times \Big[\epsilon_n(\Delta)+\epsilon_n(\nabla)+\calS+ \calS(R)\Big].
\end{meqn}
\begin{proofaside}
In the case of split risk we have instead:
\begin{meqn}
(d)&\le \Big[\sqrt{\frac{K_n}{n}}+\frac{K_n}{n}\Bigl]\calS\times\calS(R)+ \frac{1}{K_n^{}}\epsilon_n(R)\calS(R)\\
&\qquad+\Bigl[\epsilon_n(\Delta)+\frac{\epsilon_n(\nabla)}{\sqrt{K_n}}]\Big[\epsilon_n(\nabla)+\frac{\epsilon_n(\nabla)+2\calS(R)}{\sqrt{K_n}}\Big]+o(\frac{1}{n}).
\end{meqn}
\end{proofaside}
We will now work on bounding (e).
First of all we can remark that:
\begin{meqn}
  &\frac{1}{n}\abs[\Big]{\sum_{i\le n}\EE\bigl[\Cov \big(g_{n,i}, \EE\big(g_{n,i} \mid F_i\big) \Bigm| X_{\dbracket{n} \setminus \{i\}}\big)\bigr]-\sigma^2_{\rm{cv}}}
\\&\leq \frac{1}{n}\Big|\sum_{i\le n}\EE\bigl[\Cov \big(g_{n,i}, \EE\big(g_{n,i} \mid F_i\big) \Bigm| X_{\dbracket{n} \setminus \{i\}}\big)\bigr]
\\&\qquad- \EE\bigl[\Cov \big(\EE\big(g_{n,i} \mid X_i), ~\EE\big(g_{n,i} \mid X_i\big) \Bigm| X_{\dbracket{n} \setminus \{i\}}\big)\bigr] \Big|
\\&\quad +\abs[\Big]{\EE\bigl[\Cov \big(\EE \big(g_{n,i} \mid X_i), ~\EE\big(g_{n,i} \mid X_i\big) \Bigm| X_{\dbracket{n} \setminus \{i\}}\big)\bigr]-\sigma^2_{\rm{cv}}}.
\end{meqn}
Let us write:
\begin{equation}
\label{flower}
  \EE\bigl(g_{n,i} \mid F_i\bigr)-\EE\bigl(g_{n,i} \mid X_i\bigr)
    \le \sum_{l \leq n}\EE \big(g_{n,i} \mid X_{ \dbracket{l} \cup \{i\}}\big)-\EE \big(g_{n,i} \mid X_{\dbracket{l-1} \cup \{i\}}\big),
\end{equation}
then by a telescopic sum argument we get that:
\begin{meqn}
&\frac{1}{n}\abs[\Big]{\sum_{i\le n}\EE \bigl[\Cov\bigl(g_{n,i}, \, \EE(g_{n,i} \mid F_i) \bigm| X_{\dbracket{n} \setminus \{i\}}\bigr) - \Cov\bigl(g_{n,i}, \, \EE(g_{n,i} \mid X_i) \bigm| X_{ \dbracket{n} \setminus \{i\}}\bigr)\bigr]}
\\&\oversetclap{(a)}{\le }\max_{i\le n} \sum_{l\ne i}\abs[\Big]{\EE\bigl[\Cov\bigl(g_{n,i}, \, \EE(g_{n,i} \mid X_{\dbracket{l} \cup \{i\}} ) \bigm| X_{\dbracket{n} \setminus \{i\}}\big)
    \\ &\hphantom{\max_{i \leq n} \sum_{l \neq i}}\qquad -\Cov\bigl(g_{n,i}, ~\EE(g_{n,i} \mid X_{\dbracket{l - 1} \cup\{ i\}}) \bigm| X_{ \dbracket{n} \setminus \{i\}}\bigr)\bigr]}
\\&\oversetclap{(b)}{\le}\max_{i\le n}\sum_{l\ne i} \abs[\Big]{\EE\bigl[\Cov\bigl(g_{n,i}, \, \EE(\nabla_l g_{n,i} \mid X_{\dbracket{l} \cup \{ i\}}) \bigm| X_{\dbracket{n} \setminus \{i\}}\bigr)\bigr]}
\\&\oversetclap{(c)}{\le}\max_{i\le n}\sum_{l\ne i} \abs[\Big]{\EE\bigl[\Cov\bigl(\nabla_l g_{n,i}, \, \EE(\nabla_l g_{n,i} \mid X_{\dbracket{l} \cup \{ i\}}) \bigm| X_{\dbracket{n} \setminus \{i\}}\bigr)\bigr]}
\\&\oversetclap{(d)}{\le}\Bigl(\epsilon_n(\Delta)+\epsilon_n(R)\Bigr)^2,
\label{hiboux}
\end{meqn}
where to get (a) we used \cref{flower} and the inequality (b) comes from the bilinearity properties of the covariance. 
To get (c) we exploited the following equality:
\begin{equation}
\EE\Bigl[\Cov \bigl(g_{n,i}(X^{n,l}), \, \EE\big(\nabla_l g_{n,i} \mid X_{\dbracket{l} \cup \{ i\}}\big) \bigm| X_{\dbracket{n} \setminus \{i\}}\bigr)\Bigr]
= 0,
\end{equation}
and to get (d) we used the following upper bound:
\begin{equation}
\norm[\big]{\nabla_l(g_{n,i})-\nabla_l(g_{n,i}(X^i))}_{L_2}
\le \frac{\epsilon_n(\Delta)+\epsilon_n(R)}{\sqrt{n}}.
\end{equation}
\begin{proofaside}
In the case of the split risk by following the same line of reasoning we have we have:
\begin{meqn}
&\frac{K_n}{n}\abs[\Big]{\sum_{i\le n}\EE \bigl[\Cov\bigl(g_{n,i}, \, \EE(g_{n,i} \mid F_i) \bigm| X_{\dbracket{n} \setminus \{i\}}\bigr) - \Cov\bigl(g_{n,i}, \, \EE(g_{n,i} \mid X_i) \bigm| X_{ \dbracket{n} \setminus \{i\}}\bigr)\bigr]}
\\&\le \epsilon_n^2(\Delta)+\frac{\epsilon_n(R)^2}{K_n} + o(\frac{1}{n}).
\end{meqn}
\end{proofaside}
Moreover remark that:
\begin{meqn}
&\frac{1}{n}\sum_{i\le n}\EE\Bigl[\Cov\bigl(g_{n,i}, ~\EE\big(g_{n,i} \mid X_i\big) \bigm| X_{\dbracket{n} \setminus \{i\}}\bigr)\Bigr]
\\&=\frac{1}{n}\sum_{i\le n}\EE\Bigl[\Bigl(g_{n,i} - \EE(g_{n,i}(X^i) \mid X)\Bigr) \times \Bigl(\EE\big(g_{n,i} \mid X_i\big)-\EE\big(g_{n,i}\big)\Bigr)\Bigr]
\\&\oversetclap{(a)}{=}\frac{1}{n}\sum_{i\le n}\EE \Bigl[ \EE\Bigl(g_{n,i}-\EE(g_{n,i}(X^i) \mid X) \mid X_i\Bigr)\times\Bigl(\EE\big(g_{n,i} \mid X_i\big)-\EE\big(g_{n,i}\big)\Bigr)\Bigr]
\\&=\frac{1}{n}\sum_{i\le n}\EE\Bigl[\Cov \big(\EE\big(g_{n,i} \mid X_i\big), \, \EE\big(g_{n,i} \mid X_i\big) \mid X_{ \dbracket{n} \setminus \{i\}}\big)\Bigr]
\\&=\frac{1}{n}\sum_{i\le n}\Var\Bigr(\EE\bigr(g_{n,i} \mid X_i\bigl)\Bigl),
\end{meqn}
where to get (a) we used the tower property of the conditional expectation.
Finally, as we have that:
\begin{equation}
\EE\big(g_{n,i} \mid X_i\big)=\bar\calL_{n - \abs{B_{b_n(i)}^n},n}(X_i)+\sum_{j = 1}^{K_n} \abs{B^n_j} \EE\bigl[R_n(\hat{f}_j) \mid X_i\bigr]
\end{equation}
we can see that:
\begin{meqn}
&
\frac{1}{n}\sum_{i\le n}\Var \Bigl(\EE\bigl(g_{n,i} \mid X_i\bigr)\Bigr)
\\&=\frac{1}{n}\sum_{i\le n}\Var\Bigl(\bar\calL_{ \abs{B_{b_n(i)}^n} ,n}(X_i)\Bigr)
+2\Cov\Bigl[\bar\calL_{n - \abs{B_{b_n(i)}^n},n}(X_i),\sum_{j \leq K_n} |B_j^n|\EE(R_n(\hat{f}_j) \mid X_i)\Bigr]
\\&\qquad +\Var\Bigl[\sum_j |B_j^n|\EE\bigl(R_n(\hat{f}_j \bigr) \mid X_i)\Bigr]
\\&= \frac{1}{n}\sum_{i\le n}\sigma_{n-|B^n_{b_n(i)}|,n}^2
+2\sum_{j\ne b_n(i)}|B_j^n|\Cov\Bigl[\bar\calL_{n - \abs{B_{b_n(i)}^n},n}(X_1),\EE(\beta_1^{n-|B_j^n|,n} \mid X_1)\Bigr]
\\&\qquad +\Var\Bigl[\sum_{j\ne b_n(i)} |B_j^n|~\EE(\beta_1^{n-|B_j^n|,n} \mid X_1)\Bigr].
\end{meqn}
Note that if $K_n$ divides $n$ then all the blocks are of same size and  we have
\begin{equation}
\sum_{j\ne b_n(i)} |B_j^n|~\beta_1^{n-\abs{B_j^n},n} =(n-\frac{n}{K_n})\beta_1^{n{(K_n-1)}/{K_n},n}.
\end{equation}
This is not  true in general but something very similar holds.
Indeed one can remark that the following holds for all $i\in \mathbb{N}$ and all integers $k\neq b_n(i)$:
\begin{meqn}
&
\Big|\Var\Bigl[\sum_{j\ne b_n(i)} |B_j^n|~\EE(\beta_1^{n-|B_j^n|,n} \mid X_1)\Bigr]
    - \bigl(n-|B^n_{b_n(i)}|\bigr)^2\Var\Bigl[\EE\bigl(\beta_1^{n-|B_k^n|,n} \mid X_1)\Bigr]\Big|
\\&\le\hspace{-2mm} \sum_{j,l\ne b_n(i)}|B_j^n||B_k^n| \norm[\Big]{\EE\bigl(\beta_1^{n-|B_j^n|,n} \mid X_1)\EE\bigl(\beta_1^{n-|B_l^n|,n} \mid X_1)  -\EE\bigl(\beta_1^{n-|B_k^n|,n} \mid X_1\bigr)^2}_{L_1}
\\&\le 2 \calS(R)\epsilon_n(d).
\end{meqn}
By following a similar reasoning to \cref{hiboux} we can show that, for any $k \leq K_n$:
\begin{equation}
\abs[\Big]{\bigl(n-\abs{B^n_{b_n(i)}}\bigr)^2\Var\Bigl[\EE\bigl(\beta_1^{n-|B_k^n|,n} \mid X_1)\Bigr]} - \bigl(n-|B^n_{b_n(i)}|\bigr)\Var\Bigl[R_n(\hat{f}_k) \Bigr]
\Big|
\le \epsilon_n(R)^2.
\end{equation}
In conclusion, we have:
\begin{meqn}
&
\abs[\Big]{\frac{1}{n}\sum_{i\le n}\Var \Bigl(\EE \bigl(g_{n,i} \mid X_i\bigr)\Bigr)-\sigma^2_{\rm{cv}}}
\\&\le \max_{l_n\in N_n} \abs[\Big]{\sigma_{l_n,n}^2 - \sigma_{1}^2}+2\calS(R)\epsilon_n(d)+ \epsilon_n(R)^2
\\&\quad + 2  \max_{l_1,l_2\in N_n} \abs[\Big]{l_2\times\Cov\Bigl(
  \bar\calL_{l_1,n}(X_1),
  \, \EE\big(\beta_1^{l_2,n} \mid X_1\big)\Bigr) - \rho}
\\&\quad +\max_{l \in N_n} \Big|l \times\Var \big(R_n(f_{l_n, n}(X_{1:l_n}))\big)-\sigma_2^2\Big|.
\end{meqn}

\begin{proofaside}
In the case of the split risk we have instead:
\begin{meqn}
&\abs[\Big]{\frac{1}{n}\sum_{i\le n}\Var \Bigl(\EE \bigl[g_{n,i} \mid X_i\bigr]\Bigr)-\sigma^2_{\rm{split}}}
\\&\quad\le \max_{l_n\in N_n} \abs[\Big]{\sigma_{l_n,n}^2 - \sigma_{1}^2}+\frac{ \epsilon_n(R)^2}{K_n}
+\max_{l \in N_n} \Big|l\times\Var \big(R_n(f_{l_n, n}(X_{1:l_n}))\big)-\sigma_2^2\Big|+o(\frac{1}{n}).
\end{meqn}
\end{proofaside}
Combining all of this together we obtain that:
\begin{meqn}
&\Big|\EE\Big(H(W)W-\sigma^2_{\mathrm{cv}}H'(W)\Big)\Big|
\\&\le \epsilon_n(\Delta)^3 + 2\Bigl(\epsilon_n(\nabla) + \epsilon_n(\Delta)\Bigr)\times \Bigl(\epsilon_n(\nabla)+ \epsilon_n(\Delta)+\calS+\calS(R)\Bigr)
\\&\quad + \frac{1}{\sqrt{n}}\Big\{  \epsilon_n(\Delta)^2\epsilon_n(\nabla)+ \bigl(\calS(R)+\calS\bigr)\times \bigl(\epsilon_n(\nabla)+ \epsilon_n(\Delta)+\calS+\calS(R)\bigr)^2\Big\} 
\\&\quad + \bigl(\calS+ \calS(R)\bigr) \times \bigl(\epsilon_n(\Delta)+ \epsilon_n(R)\bigr)
+\frac{1}{\sqrt{n}} \bigl( \calS + \calS(R)\bigr)^2
\\&\quad + \bigl(\epsilon_n(\Delta)+\epsilon_n(R)\bigr)^2+ 2\calS(R)\epsilon_n(d)+\epsilon_n(R)^2+ 4\epsilon_n(\sigma)
\\&\rightarrow 0.
\end{meqn}
In particular, there exists $C \in \RR$ such that:
\begin{meqn}
&\Big|\EE\Big(WH(W)-\sigma^2_{\rm{cv}}H'(W)\Big)\Big|
\\&\le C\Bigl\{\bigl(\epsilon_n(\nabla)+ \epsilon_n(\Delta)+{\epsilon_n(\hat{R}})\bigr) \times \bigl(\epsilon_n(\nabla)+ \epsilon_n(\Delta)+\epsilon_{\frac{n}{2}}(\hat{R})+\calS+\calS(R)\bigr)
\\&\qquad +\frac{1}{\sqrt{n}} \bigl(\calS(R)+\calS\bigr)^3
+\calS(R)\epsilon_n(d)+ \epsilon_n(\sigma)\Bigr\}
\end{meqn}

\begin{proofaside}
In the case of split risk we have:
\begin{meqn}&
\Big|\EE\Big(H(W)W-\sigma^2_{\rm{split}}H'(W)\Big)\Big|
\\&\le\frac{K_n^{\frac{1}{2}}}{n^{\frac{1}{2}}}\Bigl[\calS^3+2\calS\times\calS(R) + 2 \frac{\calS(R)}{\sqrt{n}}\Bigr]\Bigl[\epsilon_n(\Delta)+\frac{\epsilon_n(\nabla)+2\calS(R)}{\sqrt{K_n}}\Bigr]
\\&\qquad \qquad 
+ \frac{1}{K_n^{}}\epsilon_n(R)\calS(R)
+2 \epsilon_n(\sigma)+\frac{ \epsilon_n(R)^2}{K_n}+o(\frac{1}{n}).
\end{meqn}
Therefore there exists a real $C \in \mathbb{R}$ such that:
\begin{meqn}
&\abs[\Big]{\EE\bigl[W H(W)-\sigma^2_{\mathrm{split}}H'(W)\bigr]}
\le C\Bigl\{\sqrt{\frac{K_n}{n}}\calS\bigl(\calS+\calS(R)\bigr)^2
\\ &\quad+ \Bigl(\epsilon_n(\Delta)+ \frac{1}{\sqrt{K_n}}\bigl(\epsilon_n(\nabla)+\epsilon_n(R)\bigr)\Bigr)\frac{\calS(R)}{\sqrt{K_n}}+\epsilon_n(\Delta)^2
+ \epsilon_n(\sigma)\Bigr\}.
\end{meqn}
\end{proofaside}
\end{proof}

\subsection{Proof of proposition~\ref{nice}}
\label{sec:proof_estimator_sigma1}
\begin{proof}
The proof consists in first upper bounding the bias of $\Sigma_{\rm{cross}}^2$ and then studying its variance.
For simplicity of notation we omit the superscript $n$: i.e we write $X$ for $X^{n}$ and $X^i$ for $X^{n,i}$.

We have that, for all $j\le K_n$ and $l\in B^n_j$:
\begin{meqn}
&\EE\Bigl[\Bigl(\calL_n(X_l,\hat{f}_j) - \frac{1}{\abs{B_j^n}}\smashoperator[r]{\sum_{k \in B_j^n}} \calL_n(X_k, \hat{f}_j)\Bigr)^2\Bigr]
\\&=\EE\Bigl[\Bigl(\calL_n(X_l, \hat{f}_j) - R_n(\hat{f}_j)\Bigr)^2\Bigr]
+ \frac{1}{\abs{B_j^n}^2}\EE\Bigl[\Bigl(\smashoperator[r]{\sum_{k \in B_j^n}}\calL_n(X_l, \hat{f}_j) - R_n(\hat{f}_j)\Bigr)^2\Bigr]
\\&\quad -\frac{2}{\abs{B^n_j}} \EE\Bigl[\Bigl(\calL_n(X_l,\hat{f}_j) - R_n\big(\hat{f}_j\big)\Bigr)\Bigl(\smashoperator[r]{\sum_{k\in B_j^n}}\calL_n(X_k,\hat{f}_j) -R_n (\hat{f}_j)\Bigr) \Bigr].
\end{meqn}
We study each term separately. First, we have by definition that:
\begin{equation*}
\EE\Bigl[\Bigl(\calL_n(X_l,\hat{f}_j) - R_n(\hat{f}_j)\Bigr)^2\Bigr] = \sigma_{n - \abs{B^n_j},n}^2.
\end{equation*}
Therefore we only need to study the other two terms.
We may expand the terms to obtain that:
\begin{meqn}
&\frac{1}{\abs{B_j^n}^2} \EE\Bigl[\Bigl(\smashoperator[r]{\sum_{i\in B_j^n}} \calL_n(X_i,\hat{f}_j) - R_n (\hat{f}_j)\Bigr)^2\Bigr]
\\&=\frac{1}{\abs{B_j^n}^2} \smashoperator[r]{\sum_{i\in B_j^n}} \EE\Bigl[\Bigl(\calL_n(X_i, \hat{f}_j) - R_n(\hat{f}_j)\Bigr)^2\Bigr]
\\&\quad+ \frac{1}{\abs{B_j^n}^2}\sum_{k, l \in B_j^n, k \neq l}
    \EE\Bigl[\Bigl(\calL_n(X_l,\hat{f}_j)-R_n(\hat{f}_j)\Bigr)
        \Bigl(\calL_n(X_k,\hat{f}_j) - R_n(\hat{f}_j)\Bigr)\Bigr].
\end{meqn}
Note that when conditioning on $\hat{f}_j$, the loss on $X_k$ and the loss on $X_l$ are independent.
Therefore for all  distinct $k\ne l$ in $B_j^n$,  we have that:
\begin{equation}
    \EE\Bigl[\Bigl(\calL_n(X_l,\hat{f}_j)-R_n(\hat{f}_j)\Bigr)
        \Bigl(\calL_n(X_k,\hat{f}_j) - R_n(\hat{f}_j)\Bigr)\Bigr] = 0.
\end{equation}
This implies that:
\begin{equation}
 \frac{1}{\abs{B_j^n}^2} \EE\Bigl[\Bigl(\sum_{i\in B_j^n}\calL_n(X_i, \hat{f}_j) - R_n (\hat{f}_j)\Bigr)^2\Bigr]
=\frac{\sigma_{n - \abs{B_j^n},n}^2}{\abs{B_j^n}}.
\end{equation}
Similarly one can show that:
\begin{equation}
    \frac{1}{\abs{B^n_j}} \EE\Bigl[\Bigl(\calL_n(X_l,\hat{f}_j) - R_n\big(\hat{f}_j\big)\Bigr)\Bigl(\sum_{k\in B_j^n}\calL_n(X_k,\hat{f}_j) -R_n (\hat{f}_j)\Bigr) \Bigr] = 
    \frac{\sigma_{n - |B_j^n|,n}^2}{|B_j^n|}.
\end{equation}
Therefore we have established that:
\begin{equation}\begin{split}&
 \EE\Bigl[\Bigl(\calL_n(X_l,\hat{f}_j) - \frac{1}{\abs{B_j^n}}\sum_{l\in B_j^n}\calL_n(X_l,\hat{f}_j)\Bigr)^2\Bigr]
 = \frac{\abs{B_j^n}-1}{\abs{B_j^n}} \sigma_{n - \abs{B^n_j},n}^2,
\end{split}\end{equation}
from which we may bound the bias of our estimator:
\begin{equation}
    \abs[\Big]{ \EE\bigl[\Sigma_{\rm{cross}}^2\bigr] - \sigma_1^2 } \leq \epsilon_n(\sigma).
\end{equation}

We now look to control the variance of our estimator.
Let us write:
\begin{equation}
    \hat{L}_l(X) \eqdef \calL_n\bigl(X_l,\hat{f}_{b_n(l)}(X)\bigr) - \frac{1}{|B_{b_n(l)}^n|}\smashoperator[r]{\sum_{k\in B_{b_n(l)}^n}} \calL_n\bigl(X_k,\hat{f}_j(X)\bigr).
\end{equation}
Using Efron-Stein inequality we obtain that:
\begin{meqn}
\Var \bigl[\Sigma_{\rm{cross}}^2\bigr]
&\le\sum_{i\le n}\EE\Bigl[\Sigma_{\rm{cross}}^2(X)-\Sigma_{\rm{cross}}^2(X^i)\Bigr]
\\&\le\sum_{i\le n} \EE\Bigl[\Bigl(\frac{1}{K_n}\smashoperator[r]{\sum_{j\le K_n}} \Sigma^2_j(X)-\Sigma^2_j(X^i)\Bigr)^2\Bigr]
\\&\le \frac{1}{K_n^2}\sum_{i\le n}\sum_{j,k \le K_n}
    \frac{1}{\bigl(\abs{B^n_j}-1\bigr)\bigl(\abs{B^n_k}-1\bigr)}
        \sum_{l \in B^n_j} \smashoperator[r]{\sum_{m \in B^n_k}} \EE\Bigl[\nabla_i\hat{L}_l^2 \: \nabla_i \hat{L}_m^2\Bigr]
\\&\oversetclap{(a)}{\le} \frac{1}{K_n^2}\sum_{i\le n}\Bigl(\sum_{j\le K_n}\frac{1}{\abs{B^n_j}-1}
        \sum_{l\in B^n_j}\norm[\big]{\nabla_i\hat{L}_l^2}_{L_2}\Bigr)^2
\\&\oversetclap{(b)}{\le} \frac{4}{K_n^2} \sum_{i\le n} \norm[\big]{\hat{L}_l}^2_{L_4}\Bigl(\sum_{j\le K_n}\frac{1}{\abs{B^n_j} -1} 
    \sum_{l\in B^n_j}\norm[\big]{\nabla_i\hat{L}_l}_{L_4}\Bigr)^2
 \\&\oversetclap{(c)}{\le} \frac{16\calS^2}{K_n^2} \sum_{i\le n}\Bigl(\sum_{j\le K_n}\frac{1}{\abs{B^n_j}-1}\sum_{l\in B^n_j} \norm[\big]{\nabla_i\hat{L}_l}_{L_4}\Bigr)^2,
\end{meqn}
where (a)  and (b) are consequences of Cauchy-Schwarz inequality, and (c) is immediate from the fact that
$\norm{\hat{L}_l}_{L_4}\le 2 \calS$.
Finally note that the following holds for all $i$ and all $l$ in the same fold:
\begin{equation}\begin{split}
\norm[\big]{\nabla_i\hat{L}_l}_{L_4} \le \mathbb{I}(l=i)\frac{|B^n_{b_n(i)}|-1}{|B^n_{b_n(i)}|}\calS+ \frac{\mathbb{I}(l\ne i)}{|B^n_{b_n(i)}|}\calS.
\end{split}\end{equation}
Instead, if $i$ and $l$ belong to different fold, i.e $b_n(i)\ne b_n(l)$, we have:
\begin{equation}\begin{split}
\big\|\nabla_i\hat{L}_l \big\|_{L_4}\le2\frac{|B^n_{b_n(l)}|-1}{|B^n_{b_n(l)}|}\frac{\epsilon_n(\Delta)}{n}.
\end{split}\end{equation}
Therefore combined with the fact that $\Big|\frac{1}{|B_{b_n(i)}|}-\frac{K_n}{n}\Big|=O(\frac{K_n^2}{n^2})$ we have that:
\begin{equation}\begin{split}
\Var\bigl[\Sigma_{\rm{cross}}^2\bigr]
 {\le}\frac{16\calS^2}{n}\Bigl(2\calS + 2\epsilon_n(\Delta)\Bigr)^2+o(\frac{1}{n}).
\end{split}\end{equation}
This implies that:
\begin{equation}\begin{split}
\norm[\Big]{\Sigma_{\rm{cross}}^2-\sigma_1^2}_{L_2}
 {\le}\frac{8\calS}{\sqrt{n}}\bigl(\calS + \epsilon_n(\Delta)\bigr) + \epsilon_n(\sigma)+o(\frac{1}{\sqrt{n}}).
\end{split}\end{equation}
\end{proof}

\subsection{Proof of proposition~\ref{cannes}}
\begin{proof} 
The proof proceeds in two stages.
In a first stage, we bound the difference between $\EE(\hat{S}^2_{\text{cv}})$ and $\sigma_{\rm{cv}}^2$.
In a second stage, we upper-bound the variance of $\hat{S}^2_{\text{cv}}$ to ensure consistency.
Note that we can assume without loss of generality that $n$ is even.
For ease of notation we drop the superindex $n$ from the processes: for example we write $X$ instead of $X^n$ and $X^i$ instead of $X^{n,i}$.
Note that as $(X_l)$ is an process with independent entries, we can assume without loss of generality that $X^i=\tilde{X}^i$ for all $i\le n$.
We introduce the following notations:
\begin{align*}
K_n^{i}(X) &\eqdef \calL_n(X_i,\hat{f}_{b_{\frac{n}{2}}(i)}(X))-R(\hat{f}_{b_{\frac{n}{2}}(i)}(X)), \\
D_n^{i,j}(X) &\eqdef \nabla_jK_n^i(X), \\
F_i &\eqdef \sigma(X_1,\dots,X_i),
\end{align*}
and by abuse of notation will simply write $K_n^{i}$ and $D_n^{i,j}$ for respectively $K_n^{i}(X)$ and $D_n^{i,j}(X)$.
We remark that by symmetry we have:
\begin{equation}
\EE(\hat{S}_{\text{cv}}^2) =\frac{n^2}{4}\EE\Big[\big(\hat{R}^{n/2}_{\rm{cv}}-\hat{R}^{n/2}_{\rm{cv}}(X^1)\big)^2\Big].
\end{equation}
Therefore we will study $\EE\Bigl[\bigl(\hat{R}^{n/2}_{\rm{cv}}-\hat{R}^{n/2}_{\rm{cv}}(X^1)\bigr)^2\Bigr]$ and relate it to $\sigma^2_{\rm{cv}}.$
By standard computations we have that:
 \begin{meqn}
 &\label{eq:jane-estimate-cv}
\frac{n^2}{4}\EE\Big[\big(\hat{R}^{n/2}_{\rm{cv}}-\hat{R}^{n/2}_{\rm{cv}}(X^1)\big)^2\Big]
\\&=\EE\Bigl[\Bigl(\nabla_1K^1_n + \sum_{l\ne 1} \nabla_1 \calL_n(X_l, \hat{f}_{b_{\frac{n}{2}}(l)})\Bigr)^2\Bigr]
\\&=\EE\Bigl[(\nabla_1 K^1_n)^2\Big]
    + 2\sum_{l\ne 1}\EE\Big[\nabla_1 K^1_n \times \nabla_1 \calL_n(X_l, \hat{f}_{b_{\frac{n}{2}}(l)})\Bigr]
\\& \quad+\smashoperator[r]{\sum_{l,k\ne 1}} \EE\Bigl[\nabla_1\calL_n(X_k,\hat{f}_{b_{\frac{n}{2}}(k)}) \times \nabla_1\calL_n(X_l,\hat{f}_{b_{\frac{n}{2}}(l)})\Bigr].
\end{meqn}
We study each term successively. To start, note that:
 \begin{equation}
 \abs[\Big]{\EE\Big[(\nabla_1 K^1_n)^2\Big] - \sigma^2_1} \le \abs[\Big]{\sigma^2_{|B^n_1|,n}-\sigma^2_1}
 \le \epsilon_{\frac{n}{2}}(\sigma).
\end{equation}
We now focus on studying the two other terms of \cref{eq:jane-estimate-cv}.
For  all $l\ne 1$ we have by the definition of $K^1_n$ that:
\begin{meqn}&
\EE\bigl[\nabla_1 K^1_n\times \nabla_1\calL_n(X_l,\hat{f}_{b_{\frac{n}{2}}(l)})\bigr]
\\&= \EE\bigl[\nabla_1 K^1_n\times \nabla_1K^l_n\bigr] + \EE\bigl[\nabla_1 K^1_n\times \nabla_1R(\hat{f}_{b_{\frac{n}{2}}(l)})\bigr].
\end{meqn}
The first term can be controlled immediately as:
\begin{meqn}
\abs[\Big]{\EE\bigl[\nabla_1 K^1_n\times \nabla_1 K^l_n\bigr]}
&\oversetclap{(a)}{=}\abs[\Big]{\EE\bigl[\nabla_{1,l} K^1_n\times \nabla_1 K^l_n\bigr]}
\\&\le 2 \sqrt{2}\frac{\epsilon_{\frac{n}{2}}(\nabla)\epsilon_{\frac{n}{2}}(\Delta)}{n^{\frac{3}{2}}}.
\end{meqn}
where the equality (a) is due to the fact that:
\begin{equation}
\abs[\Big]{\EE\bigl[\nabla_1 K^1_n(X^l) \times \nabla_1 K^l_n\bigr]}=0.
\end{equation}
For the second term, we note that:
\begin{equation}
\EE\bigl[\nabla_1 K^1_n\times \nabla_1R(\hat{f}_{b_{\frac{n}{2}}(l)})\bigr] = \EE\Bigl[\Cov\big(K_n^1, R_n(\hat{f}_{b_{\frac{n}{2}}(l)}) \mid X^1\big)\Big],
\end{equation}
 and we will thus compare it to $\rho$.
 If $l$ belongs to a different fold than $1$, then by a telescopic sum argument we have:
\begin{meqn}
 &\abs[\Big]{\EE\bigl[\Cov\bigl(K_n^1, R_n(\hat{f}_{b_{\frac{n}{2}}(l)}) \mid X^1\bigr)\bigr] - \frac{\rho}{\frac{n}{2}-|B^{n}_{b_{\frac{n}{2}}(l)}|}}
 \\&\le \abs[\Big]{\EE\bigl[\Cov\bigl(K_n^{1},R_n(\hat{f}_{b_{\frac{n}{2}}(l)}) \mid X^l\bigr)
    - \Cov\bigl(\EE[K_n^{1} \mid X_l], \EE[R_n(\hat{f}_{b_{\frac{n}{2}}(l)}) \mid X_l] \mid X^l\bigr)\bigr]}
    \\&\quad +\frac{\epsilon_{\frac{n}{2}}(\sigma)}{\frac{n}{2}-|B^{n}_{b_{\frac{n}{2}}(l)}|}
 \\&\le \sum_{j\ne l}\abs[\Big]{\EE\Bigl[\Cov\bigl(\EE[\nabla_j K_n^{1} \mid X_{\dbracket{j}\cup\{l\}}], \EE[R_n(\hat{f}_{b_{\frac{n}{2}}(l)}) \mid X_{\dbracket{j}\cup\{l\}}] \mid X^l\bigr)\Bigr]}
    \\&\quad + \sum_{j\ne l}\Big|\EE\Bigl[\Cov\bigl(\EE[K_n^{1} \mid X_{\dbracket{j}\cup\{l\}}], \EE[\nabla_jR_n(\hat{f}_{b_{\frac{n}{2}}(l)}) \mid X_{\dbracket{j}\cup\{l\}}] \mid X^l\bigl)\Bigr]\Big|
        +\frac{\epsilon_{\frac{n}{2}}(\sigma)}{\frac{n}{2}-|B^{n}_{b_{\frac{n}{2}}(l)}|}
 \\&\oversetclap{(a)}{\le} 2\sum_{j\ne l}\Big|\EE\Bigl[\Cov\Bigl(\EE[\nabla_jK_n^1 \mid X_{\dbracket{j}\bigcup\{l\}}], \EE[\nabla_jR_n(\hat{f}_{b_{\frac{n}{2}}(l)}) \mid X_{\dbracket{j}\bigcup\{l\}}] \mid X^l\Bigr)\Bigr]\Big| +\frac{\epsilon_{\frac{n}{2}}(\sigma)}{\frac{n}{2}-|B^{n}_{b_{\frac{n}{2}}(l)}|}
 \\&\le2\sqrt{2}\frac{\epsilon_{\frac{n}{2}}(\Delta)\epsilon_{\frac{n}{2}}(\hat{R})}{{n}}
    + \frac{\epsilon_{\frac{n}{2}}(\sigma)}{\frac{n}{2}-|B^{n}_{b_{\frac{n}{2}}(l)}|},
\label{petra}
\end{meqn}
where to get (a) we used the fact that:
\begin{equation*}
\EE\Bigl[\Cov\bigl(\EE[\nabla_jK_n^{1} \mid X_{\dbracket{j}\cup\{l\}}], \EE[R_n(\hat{f}_{b_{\frac{n}{2}}(l)}(X^j)) \mid X_{\dbracket{j}\cup\{l\}}] \mid X^l \bigr)\Bigr] =0.
\end{equation*}
This implies that:
\begin{meqn}&
\abs[\Big]{\sum_{l\ne 1}\EE\bigl[\nabla_1 K^1_n \times \nabla_1\calL_n(X_l,\hat{f}_{b_{\frac{n}{2}}(l)})\bigr]-\rho}
\\&\quad \le \sqrt{2}\frac{\epsilon_{\frac{n}{2}}(\nabla)\epsilon_{\frac{n}{2}}(\Delta)}{\sqrt{n}}+ \sqrt{2}{\epsilon_{\frac{n}{2}}(\Delta)\epsilon_{\frac{n}{2}}(\hat{R})}+{\epsilon_{\frac{n}{2}}(\sigma)}+O(\frac{1}{n}
).
\end{meqn}
where the $O(\frac{1}{n})$ term is a consequence of the (potentially) uneven size of the folds.
Therefore the only term left to study in \cref{eq:jane-estimate-cv} is:
\begin{equation}\label{jane}
\smashoperator[r]{\sum_{l,k\ne 1}} \EE\Bigl[\nabla_1\calL_n(X_k, \hat{f}_{b_{\frac{n}{2}}(k)})\times \nabla_1 \calL_n(X_l,\hat{f}_{b_{\frac{n}{2}}(l)})\Bigr].
\end{equation}
If $k=l$ then we know that:
\begin{equation}
\EE\Bigl[\bigl(\nabla_1 \calL_n(X_k,\hat{f}_{b_{\frac{n}{2}}(k)})\bigr)^2\Bigr] \le \frac{2\epsilon_{\frac{n}{2}}(\Delta)^2}{n}.
\end{equation}
If $k$ is different from $l$ then we consider two different cases: (i) if they belong to the same fold, and (ii) if they belong to different folds.
If $b_{\frac{n}{2}}(k)=b_{\frac{n}{2}}(l)$, then by conditional independence  we have:
 \begin{meqn}
 &\EE\Bigl[\bigl(\nabla_1\calL_n(X_l,\hat{f}_{b_{\frac{n}{2}}(l)})\bigr)\bigl(\nabla_1\calL_n(X_k,\hat{f}_{b_{\frac{n}{2}}(k)})\bigr)\Bigr]
 \\&\oversetclap{(a)}{=}\EE\Bigl[\nabla_1\EE\Big[\calL_n(X_l,\hat{f}_{b_{\frac{n}{2}}(k)}) \mid \hat{f}_{b_{\frac{n}{2}}(k)}\Big] \times
 \nabla_1 \EE \Big[\calL_n(X_k,\hat{f}_{b_{\frac{n}{2}}(k)}) \mid \hat{f}_{b_{\frac{n}{2}}(k)}\Big]\Bigr]
 \\&\oversetclap{(b)}{=}\EE\Big[\big(\nabla_1R_n(\hat{f}_{b_{\frac{n}{2}}(k)})\bigr)^2\Big]
 \\&=\EE\Bigl[\Var\Big[R_n(\hat{f}_{b_{\frac{n}{2}}(k)}) \mid X^1\Big]\Bigr],
\end{meqn}
where (a) is a consequence of the tower property, (b) comes from the conditional independence of $\calL_n(X_l,\hat{f}_{b_{\frac{n}{2}}(l)})$ and $\calL_n(X_k,\hat{f}_{b_{\frac{n}{2}}(l)})$.
One can note that as $(F_i)$ forms a filtration the following holds:
\begin{meqn}
 &\Var\Big[R_n(\hat{f}_{b_{\frac{n}{2}}(k)})\Big]
 \\&= \smashoperator[r]{\sum_{i\not \in B^{\sfrac{n}{2}}_{b_{\sfrac{n}{2}}(k)} }}
        \EE\Bigl[\Bigl(\EE\big(R_n(\hat{f}_{b_{\frac{n}{2}}(k)}) \mid F_i\big)
        -\EE\big(R_n(\hat{f}_{b_{\frac{n}{2}}(k)}) \mid F_{i-1}\big)\Bigr)^2\Bigr]
  \\&= \smashoperator[r]{\sum_{i\not \in B^{\sfrac{n}{2}}_{b_{\sfrac{n}{2}}(k)} }} \EE\Bigl[\EE\big(\nabla_i R_n(\hat{f}_{b_{\frac{n}{2}}(k)}) \mid F_i\big)^2\Big].
\end{meqn}
Using a telescopic sum argument we deduce that:
\begin{meqn}
&\abs[\Big]{\EE\Bigl[\EE\bigl(\nabla_iR_n(\hat{f}_{b_{\frac{n}{2}}(k)}) \mid F_i\bigr)^2\Bigr]
    -\EE\Bigl[\EE\bigl(\nabla_iR_n(\hat{f}_{b_{\frac{n}{2}}(k)}) \mid X_i\bigr)^2\Bigr]}
\\&\le\sum_{j\le i}\abs[\Big]{\EE\Bigl[\EE\bigl(\nabla_iR_n(\hat{f}_{b_{\frac{n}{2}}(k)}) \mid X_{\dbracket{j}\cup\{i\}}\bigr)^2
    -\EE\bigl(\nabla_iR_n(\hat{f}_{b_{\frac{n}{2}}(k)}) \mid X_{\dbracket{j-1}\cup\{i\}}\bigr)^2\Bigr]}
\\&\le 2\sum_{j<i}\abs[\Big]{\EE\Bigl[\EE\bigl(\nabla_{i,j}R_n(\hat{f}_{b_{\frac{n}{2}}(k)}) \mid X_{\dbracket{j}\cup\{i\}}\bigr)^2\Bigr]}
\\&\le \frac{8}{n^2}\epsilon_{\frac{n}{2}}(\hat{R})^2.
\label{nyc1}
\end{meqn}
This implies that:
\begin{meqn}
 &\Big|\EE\Bigl[\bigl(\nabla_1\calL_n(X_l,\hat{f}_{b_{\frac{n}{2}}(k)})\bigr) \bigl(\nabla_1\calL_n(X_k,\hat{f}_{b_{\frac{n}{2}}(k)})\bigr)\Bigr]
    -\frac{\sigma^2_2}{\bigl(\frac{n}{2}-|B^{\sfrac{n}{2}}_{b_{\frac{n}{2}}(k)}|\bigr)^2}\Big|
 \\&\le \frac{\epsilon_{\frac{n}{2}}(\sigma)}{\bigl(\frac{n}{2}-|B^{\sfrac{n}{2}}_{b_{\frac{n}{2}}(k)}|\bigr)^2}
    + 8\frac{\epsilon_{\frac{n}{2}}(\hat{R})^2}{n^2}.
\label{nyc2}
\end{meqn}

We studied the case where $k$ and $l$ belong to the same fold.
On the other hand, if $k,l\not\in B^n_1$ and belong to different folds, i.e. $b_{\frac{n}{2}}(k) \neq b_{\frac{n}{2}}(l)$ then by expanding we have:
 \begin{meqn}
 &\EE\Bigl[\bigl(\nabla_1\calL_n(X_l,\hat{f}_{b_{\frac{n}{2}}(l)}) \bigr)
 \bigl(\nabla_1\calL_n(X_k,\hat{f}_{b_{\frac{n}{2}}(k)}) \bigr) \Bigr]
\\&=\EE\Bigl[ \bigl(\nabla_1R_n(\hat{f}_{b_{\frac{n}{2}}(k)}) \bigr)
\bigl( \nabla_1R_n(\hat{f}_{b_{\frac{n}{2}}(l)}) \bigr)\Bigr]
+\EE\Bigl[ \bigl(\nabla_1 K_n^{k}\bigr) \bigl(\nabla_1 K_n^{l} \bigr)\Bigr]
\\&\quad +\EE\Bigl[\bigl(\nabla_1 K_n^{k}\bigr) \bigl(\nabla_1R_n(\hat{f}_{b_{\frac{n}{2}}(l)}) \bigr)\Bigr]
+\EE\Bigl[ \bigl(\nabla_1 K_n^{l} \bigr) \bigl( \nabla_1R_n(\hat{f}_{b_{\frac{n}{2}}(k)}) \bigr)\Bigr].
\label{tete}
\end{meqn}
We analyze each term separately.
First, note that we have:
\begin{equation}\begin{split}\label{eq_1}
 \Big|\EE\bigl[\nabla_1 K_n^l \: \nabla_1 R_n(\hat{f}_{b_{\frac{n}{2}}(k)})\bigr]\Big|
 &=\Big|\EE\Bigl[\nabla_1 K_n^{l} \: \nabla_{1,l} R_n(\hat{f}_{b_{\frac{n}{2}}(k)})\Big]\Big|
            \\
            &\le \frac{4}{n^2}\epsilon_{\frac{n}{2}}(\Delta)\epsilon_{\frac{n}{2}}(\hat{R}).
\end{split}\end{equation}
Moreover, by conditional independence, we have:
\begin{equation}\label{eq_2}
    \abs[\Big]{\EE\Bigl[\nabla_1 K_n^{l} \: \nabla_1 K_n^{k}\Bigr]}
            = \abs[\Big]{\EE\Bigl[\nabla_{1,l}K_n^{l} \: \nabla_{1,k}K_n^{k}\Bigr]}
\le \frac{4}{n^2} \epsilon_{\frac{n}{2}}(\nabla)^2.
\end{equation}
Finally we study the last term of \cref{tete}. Towards this goal notice that by using the definition of $\epsilon_{\frac{n}{2}}(\sigma)$ we obtain:
 \begin{meqn}
 &
\Big|\EE\Bigl[\nabla_1 R_n(\hat{f}_{b_{\frac{n}{2}}(k)})\nabla_1R_n(\hat{f}_{b_{\frac{n}{2}}(l)})\Bigr]-\frac{\sigma_2^2}{\big[\frac{n}{2}-|B^{{n}/{2}}_{b_{\frac{n}{2}}(k)}|\big]^2}\Big|
\\&\le \Big|\EE\Big(\Cov\big[R_n(\hat{f}_{b_{\frac{n}{2}}(k)}), R_n(\hat{f}_{b_{\frac{n}{2}}(l)}) \mid X^1\bigr]\Big) - \EE\Bigl(\Var\bigl[R_n(\hat{f}_{b_{\frac{n}{2}}(k)})\mid X^1\bigr]\Bigr)\Big|
\\&\qquad+\frac{\epsilon_{\frac{n}{2}}(\sigma)}{\Bigl(\frac{n}{2}-|B^{{n}/{2}}_{b_{\frac{n}{2}}(k)}|\Bigr)^2}
    +\frac{8}{n^2}\epsilon_{\frac{n}{2}}(\hat{R})^2,
\end{meqn}
where the term $\frac{8}{n^2}\epsilon_{\frac{n}{2}}(\hat{R})^2$ comes from \cref{nyc1} and \cref{nyc2}.
 We would like to compare $\Cov\Big[\nabla_1R_n(\hat{f}_{b_{\frac{n}{2}}(k)}),\nabla_1 R_n(\hat{f}_{b_{\frac{n}{2}}(l)})\Bigr]$ to $\sigma^2_2$. 
Towards this goal, we remark that the following holds:
 \begin{meqn}
&\Big|\EE\Bigl(\EE\big[\nabla_1R_n(\hat{f}_{b_{\frac{n}{2}}(k)}) \mid X^1\big] \EE\big[\nabla_1 R_n(\hat{f}_{b_{\frac{n}{2}}(l)}) \mid X^1\big]
\\&\qquad\quad - \EE\big[\nabla_1R_n(\hat{f}_{b_{\frac{n}{2}}(k)}) \mid X_1\big] \EE\big[\nabla_1R_n(\hat{f}_{b_{\frac{n}{2}}(l)}) \mid X_1\big]\Bigr)\Big|
\\&\leq \sum_{j\ne 1}\Big|\EE\Bigl(\EE\big[\nabla_1R_n(\hat{f}_{b_{\frac{n}{2}}(k)}) \mid X_{\dbracket{j}\cup\{1\}}\big] \EE\big[\nabla_1R_n(\hat{f}_{b_{\frac{n}{2}}(l)}) \mid X_{\dbracket{j}\cup\{1\}}\big]\Bigr)
\\&\qquad\quad  -\EE\Bigl(\EE\big[\nabla_1R_n(\hat{f}_{b_{\frac{n}{2}}(k)}) \mid X_{\dbracket{j-1}\cup\{1\}}\big] \EE\big[\nabla_1R_n(\hat{f}_{b_{\frac{n}{2}}(l)}) \mid X_{\dbracket{j-1}\cup\{1\}}\big]\Bigr)\Big|
\\&\le2 \sum_{j\ne 1}\Big|\EE\Bigl[\EE\big[\nabla_{1,j}R_n(\hat{f}_{b_{\frac{n}{2}}(k)}) \mid X_{\dbracket{j}\cup\{1\}}\big] \EE\big[\nabla_{1,j}R_n(\hat{f}_{b_{\frac{n}{2}}(l)}) \mid X_{\dbracket{j}\cup\{1\}}\big]\Bigr]\Big|
\\&\le \frac{8}{n^2}\epsilon_{n/2}(\hat{R})^2.
\label{vomir_1}
\end{meqn}
By the definition of $\epsilon_{\frac{n}{2}}(d)$ we know that:
\begin{meqn}
\label{vomir_3}
&\Big|\EE\Bigl[\EE\bigl[\nabla_1R_n(\hat{f}_{b_{\frac{n}{2}}(k)}) \mid X_1\bigr] \EE\bigl[\nabla_1R_n(\hat{f}_{b_{\frac{n}{2}}(l)}) \mid X_1\bigr]\Bigr]
    -\Var\bigl[\EE\bigl[R_n(\hat{f}_{b_{\frac{n}{2}}(k)}) \mid X_1 \bigr]\bigr]\Big|
\\&\le 8\frac{\calS(R)\epsilon_{\frac{n}{2}}(d)}{n^2}.
\end{meqn} 
Therefore \cref{vomir_1,vomir_3} imply that:
 \begin{meqn}
 &\label{eq_3}
 \Big|\Cov\Big[\nabla_1R_n(\hat{f}_{b_{\frac{n}{2}}(k)}),\nabla_1 R_n(\hat{f}_{b_{\frac{n}{2}}(l)})\Bigr]-\frac{\Var\Big[R_n(\hat{f}_{b_{\frac{n}{2}}(k)})\Big]}{\frac{n}{2}-|B^{{n}/{2}}_{b_{\frac{n}{2}}(k)}|} \Big|
 \\&\le \frac{16}{n^2}\epsilon_{\frac{n}{2}}(\hat{R})^2 + \frac{4}{n^2}\calS(R)\epsilon_{\frac{n}{2}}(d).
\end{meqn}
Hence using \cref{eq_1}, \cref{eq_2} and \cref{eq_3} we know that 
 \begin{equation}\begin{split}&
\Big|\EE\Bigl[\nabla_1\calL_n(X_l,\hat{f}_{b_{\frac{n}{2}}(l)})\nabla_1\calL_n(X_k,\hat{f}_{b_{\frac{n}{2}}(k)})\Bigr]-\frac{\sigma_2^2}{\Big[\frac{n}{2}-|B^{{n}/{2}}_{b_{\frac{n}{2}}(k)}|\Big]^2}\Big|
\\&\le \frac{8\epsilon_{\frac{n}{2}}(\Delta)\epsilon_{\frac{n}{2}}(\hat{R})+16\epsilon_{\frac{n}{2}}(\hat{R})^2+ 4\epsilon_{\frac{n}{2}}(\nabla)^2+4\calS(R)\epsilon_{\frac{n}{2}}(d)}{n^2}+\frac{\epsilon_{\frac{n}{2}}(\sigma)}{\Big[\frac{n}{2}-|B^{{n}/{2}}_{b_{\frac{n}{2}}(k)}|\Big]^2}.
\end{split}\end{equation}
Putting everything together, we thus deduce that:
\begin{equation}\begin{split}&
\Big|\sum_{l \ne k}\EE\Bigl[\nabla_1\calL_n(X_l,\hat{f}_{b_{\frac{n}{2}}(l)})\nabla_1\calL_n(X_k,\hat{f}_{b_{\frac{n}{2}}(k)})\Bigr]-\sigma_2^2\Big|
\\&\le {2\epsilon_{\frac{n}{2}}(\Delta)\epsilon_{\frac{n}{2}}(\hat{R})+4\epsilon_{\frac{n}{2}}(\hat{R})^2+ \epsilon_{\frac{n}{2}}(\nabla)^2+\calS(R)\epsilon_{\frac{n}{2}}(d)}+2\epsilon_{\frac{n}{2}}(\sigma)+O(\frac{1}{n}),
\end{split}\end{equation}
where the $O(\frac{1}{n})$ term is a consequence of the folds being (potentially) uneven.
We have thus successfully established the following bound on the bias of $\hat{S}_{\text{cv}}$:
\begin{meqn}
\abs[\Big]{\EE\bigl[\hat{S}^2_{\text{cv}}\bigr] - (\sigma^2_1+\sigma_2^2+2\rho)}
&\leq 4\epsilon_{\frac{n}{2}}(\Delta)\epsilon_{\frac{n}{2}}(\hat{R})+4\epsilon_{\frac{n}{2}}(\hat{R})^2+ \epsilon_{\frac{n}{2}}(\nabla)^2
+\calS(R)\epsilon_{\frac{n}{2}}(d)
\\&\quad+5\epsilon_{\frac{n}{2}}(\sigma)+ \sqrt{\frac{2}{n}}\epsilon_{\frac{n}{2}}(\nabla)\epsilon_{\frac{n}{2}}(\Delta)
+O(\frac{1}{n}).
\label{asterix}
\end{meqn}
We know work on the variance. First of all we note that:
 \begin{meqn}
 &
\Big\|\hat{S}^2_{\rm{cv}} - \frac{2}{n}\sum_{i\le \frac{n}{2}}\bigl(\nabla_i K^i_n + \sum_{l\ne i} \nabla_iR(\hat{f}_{b_{\frac{n}{2}}(l)})\bigr)^2\Big\|_{L_1}
\\&\le \frac{2}{n}\sum_{i\le \frac{n}{2}}\norm[\Big]{\sum_{l\ne i} \nabla_i K^l_n}_{L_2}^2
\\&\le \frac{2}{n}\sum_{i\le \frac{n}{2}}\sum_{j, l\ne i}\EE\Bigl[\nabla_i K^l_n \: \nabla_i K^j_n\Bigr]
\\&\le \frac{2}{n}\sum_{i\le \frac{n}{2}}\Big[\sum_{j\ne i}\EE\bigl[ (\nabla_i K^j_n)^2\bigr] + \sum_{j \ne l}\EE\bigl[ \nabla_i K^l_n \: \nabla_i K^j_n \bigr]\Big]
\\&\le \frac{2}{n}\sum_{i\le \frac{n}{2}}\Bigl[\epsilon_{\frac{n}{2}}(\Delta)^2 + \sum_{j \ne l}\EE\bigl[ \nabla_{i,j}K^l_n \: \nabla_{i,l} K^j_n\bigr]\Bigr]
\\&\le \epsilon_{\frac{n}{2}}(\Delta)^2+ \epsilon_{\frac{n}{2}}(\nabla)^2.
\end{meqn}
Therefore we only need to study the variance of:
\begin{equation}
\frac{2}{n}\sum_{i\le \frac{n}{2}}\big[\nabla_i K^i_n+ \sum_{l\ne i} \nabla_i R(\hat{f}_{b_{\frac{n}{2}}(l)})\big]^2.
\end{equation}
By the Efron-Stein inequality one can notice that the following holds:
 \begin{meqn}
 &\Var\Bigl(\frac{2}{n}\sum_{i\le \frac{n}{2}}\Bigl(\nabla_iK^i_n+ \sum_{l\ne i} \nabla_i R(\hat{f}_{b_{\frac{n}{2}}(l)})\Bigr)^2\Bigr)
\\&\quad\oversetclap{(a)}{\le}
\frac{4}{n^2} \sum_j \EE\Bigl[\Bigl(\sum_{i\le \frac{n}{2}} \nabla_j \bigl\{ \bigl(\nabla_i K^i_n + \sum_{l\ne i} \nabla_i R(\hat{f}_{b_{\frac{n}{2}}(l)})\bigr)^2 \bigr\} \Bigr)^2\Bigr]
\\&\quad\oversetclap{(b)}{\le}\frac{16}{n^2} \sup_i\Big\|\nabla_i K^i_n + \sum_{l\ne i} \nabla_i R(\hat{f}_{b_{\frac{n}{2}}(l)})\Big\|^2_{L_4} \sum_j\Bigl(\sum_{i\le \frac{n}{2}}\Big\|\nabla_{i,j} K^i_n + \sum_{l\ne i} \nabla_{i,j}R(\hat{f}_{b_{\frac{n}{2}}(l)})\Big\|_{L_4}\Bigr)^2,
\end{meqn}
where (a) is an application of the Efron-Stein inequality and (b) is a consequence of the Cauchy-Schwarz inequality.
We bound each term of separately.
To start, note that:
\begin{equation}
\Big\|\nabla_iK^i_n+ \sum_{l\ne i} \nabla_iR(\hat{f}_{b_{\frac{n}{2}}(l)})\Big\|_{L_4}
\le \calS+\calS(R).
\end{equation}
Secondly if $j\ne i$ then we obtain that:
\begin{equation}
\Big\|\nabla_{i,j} K^i_n+ \sum_{l\ne i} \nabla_{i,j}R(\hat{f}_{b_{\frac{n}{2}}(l)})\Big\|_{L_4}
\le\sqrt{\frac{2}{n}}\bigl(\epsilon_{\frac{n}{2}}(\Delta) + \epsilon_{\frac{n}{2}}(\hat{R})\bigr).
\end{equation}
This implies that:
 \begin{meqn}
 &\Var\Bigl[\frac{2}{n}\sum_{i\le \frac{n}{2}}\bigl(\nabla_i K^i_n + \sum_{l\ne i} \nabla_iR(\hat{f}_{b_{\frac{n}{2}}(l)})\bigr)^2\Bigr]
\\&\quad \le \frac{32}{n}\bigl(\calS+\calS(R)\bigr)^4 + 32\Big[\calS+\calS(R)\Big]^2\Big[ \epsilon_{\frac{n}{2}}(\Delta)+\epsilon_{\frac{n}{2}}(\hat{R})\Big]^2,
\end{meqn}
from which we may deduce that:
 \begin{meqn}
 \norm[\Big]{\hat{S}_{\text{cv}}^2 - \sigma^2_{\rm{cv}}}_{L_1}
 &\le \frac{4\sqrt{2}}{\sqrt{n}}\bigl(\calS+\calS(R)\bigr)^2 + 4\sqrt{2}\Big[\calS+\calS(R)\Big]\Big[ \epsilon_{\frac{n}{2}}(\Delta)+\epsilon_{\frac{n}{2}}(\hat{R})\Big] + \epsilon_{\frac{n}{2}}(\Delta)^2
 \\&\quad+{6\epsilon_{\frac{n}{2}}(\Delta)\epsilon_{\frac{n}{2}}(\hat{R}) + 2\epsilon_{\frac{n}{2}}(\hat{R})^2+4 \epsilon_{\frac{n}{2}}(\nabla)^2+\calS(R)\epsilon_{\frac{n}{2}}(d)}+5\epsilon_{\frac{n}{2}}(\sigma)
 \\&\quad + \sqrt{\frac{2}{n}}\epsilon_{\frac{n}{2}}(\nabla)\epsilon_{\frac{n}{2}}(\Delta)
 +O(\frac{1}{n}).
\end{meqn}

\end{proof}

\subsection{Confidence Interval for Ridge Regression}
\subsubsection{Fast computation of replace-one estimate}
\label{sec:proof_confidence_ridge}

In general, the estimator proposed in \cref{cannes} is computationally intractable for large values of $n$, as a naive computation requires $nk$ different fits of the estimator, where $n$ denotes the number
of samples and $k$ the number of folds.
For the special case of the ridge estimator, we derive a simplification based on its specific form as a linear smoother, and the existence of closed-form LOOCV estimates for such estimators.

More precisely, let $\lambda > 0$ be a tuning parameter, and consider the ridge estimator:
\begin{equation}
    \hat{\theta}(X) = \argmin_{\theta \in \RR^p} \frac{1}{n} \sum_{i = 1}^n (y_i - x_i^\top\theta)^2 + \lambda \norm{\theta}_2^2.
\end{equation}
We may consider the estimator proposed in \cref{cannes} for the square loss:
\begin{equation}
    \hat{R}^{n/2}_{\text{cv}} = \frac{2}{n} \sum_{j = 1}^k \norm{y_{\underline{j}} - X_{\underline{j}}\hat{\theta}^{j}}_2^2,
\end{equation}
where $(y_{\underline{j}}, X_{\underline{j}})$ denotes the observation data in the $j$th fold, and $\hat\theta^{j}$ denotes
the cross-validated estimate computed on all except the $j$th fold.
Now, we have that:
\begin{meqn}
    \hat{S}_{\text{cv}}^2
    &= \frac{n}{2}\sum_{i = 1}^{n / 2} \bigl(\hat{R}_{\text{cv}}(X) - \hat{R}_{\text{cv}}(X^i)\bigr)^2 \\
    &= \frac{2}{n}\sum_{i = 1}^{n / 2} \Bigl\{ (y_i - x_i^\top \hat{\theta}^{b(i)})^2 - (y'_i - x_i'^{\top}\hat{\theta}^{b(i)})^2 + \sum_{j \neq b(i)}^k \norm{y_{\underline{j}} - X_{\underline{j}} \hat{\theta}^j}_2^2 - \norm{y_{\underline{j}} - X_{\underline{j}} \hat{\theta}^{j, i}}_2^2 \Bigr\}^2,
\end{meqn}
where we have written $\hat{\theta}^{j,i} = \hat\theta^j(X^i)$ the estimator trained on all folds of $X^i$ except the $j$th fold.
In order to efficiently compute $\hat{S}_{\text{cv}}^2$, we are interested in computing $\hat{\theta}^{j,i}$
efficiently for all $i$.
In the case of ridge, this is possible due to the following fact:
\begin{align*}
    \hat{\theta}(X^i)
    &= \Bigl( \sum_{j} x_j x_j^\top + \lambda I - x_i x_i^\top + x'_i x_i'^{\top} \Bigr)^{-1} \Bigl(\sum_{j} y_j x_j - y_i x_i + y_i' x_i'\Bigr) \\
    &= \bigl( S + U D U^\top \bigr)^{-1} \Bigl(\sum_{j} y_j x_j - y_i x_i + y_i' x_i'\Bigr),
\end{align*}
where we have written:
\begin{equation}
    S_{p \times p} = \sum_{j} x_j x_j^\top + \lambda I, \quad
    U_{p \times 2} = \begin{pmatrix} x_i^\top \\ x_i'^{\top}\end{pmatrix}, \quad
    D_{2 \times 2} = \begin{pmatrix} -1 & 0 \\ 0 & 1 \end{pmatrix}.
\end{equation}
By the well-known Woodbury identity, we have that:
\begin{equation}
    (S + U D U^\top)^{-1} = S^{-1} - S^{-1} U (D^{-1} + U^\top S^{-1} U)^{-1} U^\top S^{-1},
\end{equation}
from which we may see that computing $S^{-1}$ (which does not depend on $i$) is sufficient to efficiently compute $\hat{\theta}(X^i)$ for any $i$.
More explicitly, writing:
\begin{equation}
h_{rs} = x_r^\top S^{-1} x_s, \quad h'_{rs} = x_r^\top S^{-1} x_s', \quad h''_{rs} = x_r'^{\top} S^{-1} x_s',
\end{equation}
we may compute explicitly:
\begin{gather*}
    U^\top S^{-1} U = \begin{pmatrix}
        h_{ii} & h'_{ii} \\ h'_{ii} & h''_{ii}
    \end{pmatrix}, \\
    (D^{-1} + U^\top S^{-1} U)^{-1} = -\frac{1}{(1 - h_{ii})(1 + h''_{ii}) + h_{ii}^{\prime \, 2}}
    \begin{pmatrix}
        1 + h''_{ii} & - h'_{ii} \\ - h'_{ii} & h_{ii} - 1
    \end{pmatrix}, \\
    (S + U D U^\top)^{-1} = S^{-1} + \frac{1}{(1 - h_{ii})(1 + h''_{ii}) + h_{ii}^{\prime \, 2}}
        \Bigl(
        (1 + h''_{ii}) S^{-1}x_i x_i^\top S^{-1} \\
        \quad - h'_{ii} S^{-1} x_i x_i'^{\top} S^{-1}
        - h'_{ii} S^{-1} x_i' x_i^\top S^{-1} + (h_{ii} - 1) S^{-1} x_i'^{\top} x_i' S^{-1}\Bigr).
\end{gather*}
Putting everything together, we thus deduce that:
\begin{meqn}
    \hat{\theta}(X^i)
    &= \hat{\theta}(X) - y_i S^{-1} x_i + y'_i S^{-1} x_i'  \\
    &\quad + \frac{S^{-1}( (1 + h''_{ii}) x_i x_i^\top - h'_{ii} x_i x_i^{\prime\top} - h'_{ii} x_i' x_i^\top + (h_{ii} - 1) x'_i x_i^{\prime\top})\hat{\theta}}{ (1 - h_{ii})(1 + h''_{ii}) + h_{ii}^{\prime \, 2}} \\
    &\quad - \frac{y_i S^{-1}( (1 + h''_{ii}) h_{ii} x_i - h'_{ii} x_i h'_{ii} - h'_{ii} x_i' h_{ii}  + (h_{ii} - 1) x_i' h'_{ii})}{(1 - h_{ii})(1 + h''_{ii}) + h_{ii}^{\prime \, 2}} \\
    &\quad + \frac{y_i' S^{-1}( (1 + h''_{ii}) h'_{ii} x_i - h'_{ii} h''_{ii} x_i - h'_{ii} h'_{ii} x_i' + (h_{ii} - 1) h''_{ii} x_i')}{ (1 - h_{ii})(1 + h''_{ii}) + h_{ii}^{\prime \, 2}}
\end{meqn}
Simplifying somewhat, we obtain that:
\begin{meqn}
    \hat{\theta}(X^i)
    &= \hat{\theta}(X) - y_i S^{-1} x_i + y'_i S^{-1} x_i'  \\
    &\quad + \frac{S^{-1}( (1 + h''_{ii}) x_i x_i^\top - h'_{ii} x_i x_i^{\prime\top} - h'_{ii} x_i' x_i^\top + (h_{ii} - 1) x'_i x_i^{\prime\top})\hat{\theta}}{ (1 - h_{ii})(1 + h''_{ii}) + h_{ii}^{\prime \, 2}} \\
    &\quad - \frac{y_i S^{-1}( (h_{ii} + h''_{ii} h_{ii} - h_{ii}^{\prime \, 2}) x_i - h'_{ii} x_i')}{(1 - h_{ii})(1 + h''_{ii}) + h_{ii}^{\prime \, 2}} \\
    &\quad + \frac{y_i' S^{-1}( h'_{ii} x_i + (h_{ii} h''_{ii} - h''_{ii} - h_{ii}^{\prime \, 2}) x_i')}{ (1 - h_{ii})(1 + h''_{ii}) + h_{ii}^{\prime \, 2}} \label{eq:ridge-swap-one}
\end{meqn}
The formula in \cref{eq:ridge-swap-one} may be leveraged to compute the estimator proposed in \cref{cannes} in a computationally tractable fashion.

\subsubsection{Details of simulation}
\label{sec:confidence-ridge-simulation}

The simulation was performed for a ridge estimator, where the data was simulated according to the following process:
\begin{gather*}
    x_i \sim \calN(0, S_X), \quad \epsilon_i \sim \calN(0, \sigma^2), \\
    \beta^* = \frac{1}{\sqrt{p}} (1, \dotsc, 1)^\top, \\
    y_i = x_i^\top \beta^* + \epsilon_i.
\end{gather*}
In \cref{table:ridge-confidence}, the simulation was performed with $p = 3$, $S_X$ being Toeplitz with
diagonal elements being $1$, $0.5$ and $0.25$ respectively, and $\sigma^2 = 1$.
For each value of $n$, 5000 replicates were computed, and cross-validated risk and estimates of its variance computed.

\subsection{Proof of proposition~\ref{ien}}
\label{sec:proof_parametric_rho}
\begin{proof}
As the conditions $(H_0)-(H_3)$ of \cref{clt3} hold, we know that we need to study the quantity:
\begin{equation}
    \lim_{n\rightarrow \infty}\max_{l_1,l_2\in N_n}l_1\EE\Bigl[\bar\calL_{l_2,n}(X_1)\Bigl(\EE\bigl[R(\hat\theta_{l,n}) \mid X_1\bigr] -R_{l_1,n}\Bigr)\Bigr].
\end{equation}
Moreover by the mean value theorem we know that there is $\tilde\theta \in \bigl[\hat{\theta}_{l_1,n}(X_{1:l_1}), \hat{\theta}_{l_1,n}(X^1_{1:l_1})\bigr]$ (where here and below, $[a, b]$ denotes the line segment between $a, b \in \RR^d$) such that the following holds: 
\begin{meqn}
&\EE\bigl[R(\hat{\theta}_{l_1,n}) \mid X_1\bigr] - R_{l_1,n}
  \\&= \EE\bigl[\nabla_1 R(\hat{\theta}_{l_1,n}) \mid X_1\bigr]
  \\&= \EE\bigl[\bigl(\partial_\theta R(\tilde \theta)\bigr)^\top \bigl(\nabla_1 \hat{\theta}_{l_1,n}\bigr) \mid X_1\bigr]
  \\&= \EE\bigl[\bigl(\partial_\theta R(\hat{\theta}^*)\bigr)^\top \bigl(\nabla_1 \hat{\theta}_{l_1,n}\bigr) \mid X_1\bigr]+ \EE\bigl[\bigl(\partial_\theta R(\tilde \theta)) - \partial_\theta R(  \theta^*)\bigr)^\top \bigl(\nabla_1 \hat{\theta}_{l_1,n}\bigr) \mid X_1\bigr]
  \\&\oversetclap{(a)}{=} \partial_\theta R(\theta^*)^\top \EE\bigl[\nabla_1 \hat{\theta}_{l_1,n} \mid X_1\bigr]+ o_p(\frac{1}{n}).
\end{meqn}
where to get (a) we used the $L_2$ consistency of $\hat{\theta}_{l_1,n}(X_{1:l_1})$ coupled with the fact that $n\EE\bigl[\nabla_1 \hat{\theta}_{l_1,n} \mid X_1\bigr] = O(1)$.
Moreover by another application of the mean value theorem we know that there is $\tilde\theta \in  \bigl[\hat{\theta}_{l_1,n}(X_{1:l_1}), \hat{\theta}_{l_1,n}(X^1_{1:l_1})\bigr]$ such that the following holds:
 \begin{meqn}&
  \frac{1}{l_1}\sum_{i \le l_1}\partial_{\theta}\Psi(X^1_i, \hat{\theta}_{l_1,n}) 
  \\
  &=\frac{1}{l_1}\sum_{i \le l_1}\partial_{\theta}\Psi(X^1_i, \hat{\theta}_{l_1,n}(X^1_{1:l_1})) 
    + \Bigl(\frac{1}{l_1}\sum_{i\le i} \partial_\theta^2 \Psi(X^1_i,\tilde\theta)\Bigr) \bigl(\hat{\theta}_{l_1,n}(X_{1:l_1})-\hat{\theta}_{l_1,n}(X^1_{1:l_1})\bigr)
  \\
  &\oversetclap{(a)}{=} \Bigl(\frac{1}{l_1}\sum_{i \leq l_1} \partial_\theta^2 \Psi(X^1_i,\tilde\theta)\Bigr) \bigl(\hat{\theta}_{l_1,n}(X_{1:l_1})-\hat{\theta}_{l_1,n}(X^1_{1:l_1})\bigr)
\end{meqn}
where (a) is a consequence of the first order optimality conditions of $\hat{\theta}$ as a minimizer.
Moreover we remark that:
\begin{meqn}
  \frac{1}{l_1}\sum_{i \le l_1}\partial_{\delta}\Psi(X^1_i, \hat{\theta}_{l_1,n}) 
 &= \frac{1}{l_1}\sum_{i \le l_1}
    \Bigl\{\partial_\theta \Psi(X_i, \hat{\theta}_{l_1,n}) +\partial_{\delta}\Psi(X^1_i, \hat{\theta}_{l_1,n}) -\partial_{\delta}\Psi(X_i, \hat{\theta}_{l_1,n})\Bigr\}
    \\&\oversetclap{(a)}{=}\frac{1}{l_1}\sum_{i \leq l_1} \Bigl\{\partial_{\theta}\Psi(X^1_i, \hat{\theta}_{l_1,n}) -\partial_{\theta}\Psi(X_i, \hat{\theta}_{l_1,n}) \Bigr\},
\end{meqn}
where (a) is again  due to the first-order optimality conditions of $\hat\theta$ as a minimizer.
Therefore we have 
 \begin{meqn}
 &\EE\bigl[\nabla_1 \hat{\theta}_{l_1,n}(X_{1:l_1}) \mid X_1\bigr] \\
 &= -\frac{1}{l_1}\EE\Bigl[
    \bigl\{\frac{1}{l_1}\sum_{i\le l_1}H_\theta\Psi(X_i,\tilde\theta)\bigr\} 
    \bigl\{\partial_\theta \Psi(X_1, \hat{\theta}_{l_1,n}(X_{1:l_1})) - \partial_\theta \Psi(X'_1,\hat{\theta}_{l_1,n}(X_{1:l_1}))\bigr\} \mid X_1\Bigr]
  \\
  &\oversetclap{(a)}{=} -\frac{1}{l_1} \bigl(\EE H_{\theta} \Psi(X_1, \theta^*)\bigr)^{-1} \EE\bigl[\partial_\theta \Psi(X_1,\theta^*) -\partial_\theta \Psi(X'_1, \theta^*) \mid X_1 \bigr] + o_p(\frac{1}{n}).
\end{meqn}
where to get (a) we exploited once again the consistency of the estimator $\hat{\theta}_{l_1,n}(X_{1:l_1})$ as well as the law of large number. This implies that:
 \begin{equation}
 \rho =- \Cov\Bigl(\bigl(\partial_\theta R(\theta^*)\bigr)^\top \EE\bigl[H_{\theta}\Psi(X_1, \theta^*)\bigr]^{-1} \bigl(\partial_\theta \Psi(X_1,\theta^*)\bigr), \calL(X_1,\theta^*)\Bigr).
\end{equation}
\end{proof}

\subsection{Proof of proposition~\ref{ien_2}}
\label{sec:proof_conditions_parametric}

\begin{proof}

The key of the proof is to exploit the strict convexity of the loss function. We denote:
\begin{equation}\label{eq:proof_parametric_hat_theta_def}
  \hat{\theta}(X) \eqdef \argmin_{\theta \in \RR^d}\frac{1}{n}\sum_{i \leq n}\Psi(X_i, \theta).
\end{equation}
By abuse of notation, we write $\hat\theta = \hat\theta(X)$ and $\hat\theta^1 = \hat\theta(X^1)$.
Let $\norm{\cdot}_{L_p(v)}$ denote the $L_p$ vector norm on $\RR^d$.
We prove the desired result by checking that the conditions of theorem 3 are satisfied.
Towards this goal, we first bound $\norm[\big]{\hat{\theta} - \hat{\theta}^1}_{L_2(v)}$.
We assume w.l.o.g. that $\hat\theta(X) \overset{a.s.}{\in}\calV$.
By Taylor expansion, we have that there is $\tilde\theta \in \bigl[\hat\theta,\hat\theta^1 \bigr]$ (where here and below, $\bigl[a, b\bigr]$ denotes the line segment between $a, b \in \RR^d$) such that:
\begin{meqn}
    \sum_{i\le n}\partial_\theta \Psi(X_i,\hat\theta^1)
    &=\sum_{i\le n}\partial_\theta \Psi(X_i,\hat\theta) + \sum_{i\le n} \bigl(\partial^2_\theta \Psi(X_i,\tilde\theta)\bigr) \bigl(\hat\theta^1 - \hat\theta \bigr)
    \\&\oversetclap{(a)}{=}\sum_{i\le n} \bigl(\partial^2_\theta \Psi(X_i,\tilde\theta)\bigr)(\hat\theta^1 - \hat\theta),
\end{meqn}
where to get (a) we used the fact that $\hat\theta$ satisfies the first-order optimality condition for \eqref{eq:proof_parametric_hat_theta_def}.
Moreover we have that:
 \begin{equation}
 \sum_{i\le n}\partial_\theta \Psi(X_i,\hat\theta(X^1))
 =\sum_{i\le n}\partial_\theta \Psi(X^1_i,\hat\theta(X^1)) + \partial_\theta \Psi(X_1,\hat\theta(X^1)) - \partial_\theta \Psi(X'_1, \hat\theta(X^1)),
\end{equation}
from which we may deduce:
\begin{equation}
\norm[\big]{\hat\theta - \hat\theta^1}_{L_2(v)}
\leq \frac{\norm{\partial_\theta \Psi(X_1, \hat\theta^1)-\partial_\theta \Psi(X'_1,\hat\theta^1)}_{L_2(v)}}{\sum_{i\le n}\lambda_{\mathrm{min}}\bigl(\partial^2_\theta \Psi(X_i,\tilde\theta)\bigr)}.
\end{equation}
We can use this to verify the conditions of theorem 3.
Note that for any independent sample $\tilde X$ we have:
\begin{meqn}
&\abs[\big]{\calL\bigl(\tilde X, \hat\theta \bigr)-\calL\bigl(\tilde X,\hat\theta^1 \bigr)}
\\&
\le \sup_{\theta \in \calV} \norm[\big]{\partial_\theta \calL\bigl(\tilde X, \theta\bigr)}_{L_2(v)}
    \frac{\norm[\big]{\partial_\theta \Psi(X_1,\hat\theta^1) - \partial_\theta \Psi(X'_1, \hat\theta^1)}_{L_2(v)}}{\sum_{i\le n}\lambda_{\mathrm{min}}\bigl(\partial^2_\theta \Psi(X_i,\tilde\theta)\bigr)}.
\end{meqn}
Therefore we have by the Cauchy-Schwarz inequality:
 \begin{align*}
 \epsilon_n(\Delta)&\le \frac{1}{\delta\sqrt{n}}
    \norm[\Big]{\sup_{\theta \in \calV}\norm[\big]{\partial_\theta \Psi(X, \theta)}_{L_2(v)}}_{L_8}
    \norm[\Big]{\sup_{\theta \in \calV}\norm[\big]{\partial_\theta \calL(X, \theta)}_{L_2(v)}}_{L_8}, \\
 \calS(R)&\le \frac{1}{\delta}
    \norm[\Big]{\sup_{\theta \in \calV}\norm[\big]{\partial_\theta \Psi(X, \theta)}_{L_2(v)}}_{L_8}
    \norm[\Big]{\sup_{\theta \in \calV}\norm[\big]{\partial_\theta \calL(X, \theta)}_{L_2(v)}}_{L_8}.
\end{align*}

Now write $\hat\theta^{1,2} = \hat\theta(X^{1,2})$, and
note that there exists $\tilde\theta,\tilde\theta_2\in \RR^k$ verifying  respectively $\tilde\theta \in [\hat\theta, \hat\theta^1]$ and $\tilde\theta_2 \in [\hat\theta^2, \hat\theta^{1,2}]$ such that :
 \begin{meqn}
&\norm[\big]{\hat\theta - \hat\theta^1 - \bigl(\hat\theta^2 - \hat\theta^{1,2}\bigr)}_{L_2(v)} \\
&=\Bigl\| \Bigl(\sum_{i\le n} \partial_\theta^2 \Psi(X_i, \tilde \theta)\Bigr)^{-1} \bigl(\partial_\theta \Psi(X_1, \hat\theta) - \partial_\theta \Psi(X'_1, \hat\theta)\bigr) \\
&\quad - \Bigl(\sum_{i\le n} \partial_\theta^2 \Psi(X_i, \tilde \theta_2) \Bigr)^{-1} \bigl(\partial_\theta \Psi(X_1, \hat\theta^2) - \partial_\theta \Psi(X'_1,\hat\theta^2) \bigr)\Bigr\|_{L_2(v)}
\\&\le \norm[\Big]{\Bigl\{\Bigl(\sum_{i\le n} \partial^2_\theta \Psi(X_i, \tilde\theta)\Bigr)^{-1} - \Bigl(\sum_{i\le n} \partial^2_\theta \Psi(X_i, \tilde\theta_2)\Bigr)^{-1}\Bigr\}
    \Bigl\{\partial_\theta \Psi(X_1,\hat\theta^2) - \partial_\theta \Psi(X'_1, \hat\theta^2) \Bigr\}}_{L_2(v)}
\\&\quad+ \Big\| \Bigl(\sum_{i\le n} \partial^2_\theta \Psi(X_i, \tilde\theta)\Bigr)^{-1}
\bigl(\partial_\theta \Psi(X_1, \hat\theta^2) - \partial_\theta\Psi(X'_1, \hat\theta^2)
- \partial_\theta \Psi(X_1,\hat\theta) - \partial_\theta \Psi(X'_1, \hat\theta)\bigr)\Big\|_{L_2(v)}
\\&\le A+B.
\end{meqn}
We  bound each term successively.
Firstly we can see that there exists $\tilde f_3\in [\hat{f}(X^1),\hat{f}(X^{1,2})] $ such that the following holds:
\begin{meqn}
\norm{B}_{L_8} &\leq \Bigl\| \Bigl\| \Bigl(\sum_{i\le n} \partial^2_\theta \Psi(X_i, \tilde\theta)\Bigr)^{-1} \bigl(\partial_\theta \Psi(X_1, \hat\theta^2) - \partial_\theta \Psi(X'_1, \hat\theta^2)\\ &\qquad \qquad \qquad \qquad \qquad
- \partial_\theta \Psi(X_1, \hat\theta) - \partial_\theta \Psi(X'_1,\hat\theta) \bigr)\Big\|_{L_2(v)}\Big\|_{L_8}
\\
&\leq \frac{1}{n\delta} \norm[\Big]{ \norm[\big]{\partial_\theta \Psi(X_1, \hat\theta^2) - \partial_\theta \Psi(X'_1, \hat\theta^2) - \partial_\theta \Psi(X_1, \hat\theta) - \partial_\theta \Psi(X'_1, \hat\theta) }_{L_2(v)}}_{L_8}
\\&\le \frac{1}{n\delta} \norm[\Big]{ \lambda_{\mathrm{max}} \bigl( \partial_\theta^2 \Psi(X_1, \tilde\theta_3) - \partial_\theta^2 \Psi(X'_1, \tilde\theta_3)\big) \norm[\big]{\hat\theta^{1,2} - \hat\theta^{1}}_{L_2(v)} }_{L_8} 
\\&\leq \frac{2}{n^2 \delta^2}\norm[\Big]{  \sup_{\theta_1, \theta_2 \in \calV} \norm[\big]{\partial_\theta \Psi(X_1, \theta_1) }_{L_2(v)} \lambda_{\mathrm{max}} \bigl(\partial^2_\theta \Psi(X_1, \theta_2)\bigr)}_{L_8}.
\label{g_1}
\end{meqn}
Moreover the following holds:
\begin{meqn}
\norm{A}_{L_8} &
\leq \norm[\Big]{\norm[\Big]{\Bigl(\Bigl(\sum_{i\le n} \partial^2_\theta \Psi(X_i,\tilde\theta)\Bigr)^{-1} - \Bigl(\sum_{i\le n} \Psi(X_i, \tilde\theta_2)\Bigr)^{-1} \Bigr)
    \bigl(\partial_\theta \Psi(X_1, \hat\theta^2) - \partial_\theta \Psi(X'_1, \hat\theta^2 )\bigr)}_{L_2(v)}}_{L_8}
\\&
\leq \norm[\Big]{\norm[\big]{\bigl(\sum_{i\le n}\partial^2_\theta \Psi(X_i, \tilde\theta)\bigr)^{-1} - \bigl(\sum_{i\le n} \partial^2_\theta \Psi(X_i, \tilde\theta_2)\bigr)^{-1}}_{op} 
\norm[\big]{\partial_\theta \Psi(X_1, \hat\theta) - \partial_\theta \Psi(X'_1, \hat\theta)}_{L_2(v)}}_{L_8}
\\
&\oversetclap{(a)}{\le} \frac{1}{n^2 \delta^2} \norm[\Big]{
    \norm[\big]{\sum_{i\le n} \partial^2_\theta \Psi(X_i, \tilde\theta) - \sum_{i\le n} \partial_\theta \Psi(X_i, \tilde\theta_2)}_{op}
    \norm[\big]{\partial_\theta \Psi(X_1, \hat\theta) - \partial_\theta \Psi(X'_1, \hat\theta)}_{L_2(v)}}_{L_8}
\\
&\le \frac{2}{n\delta^2} \norm[\Big]{\norm[\big]{\partial^2_\theta \Psi(X_1, \hat\theta\big) - \partial^2_\theta \Psi(X_1, \theta^*)}_{op}}_{L_{16}}
    \norm[\Big]{\sup_{\theta \in \calV} \norm[\big]{\partial_\theta \Psi(X_1, \theta)}_{L_2(v)} }_{L_{16}},
\label{g_2}
\end{meqn}
where to get (a) we used the fact that for two matrices $C$ and $D$ the following holds: 
\begin{equation}
\norm[\big]{C^{-1}-D^{-1}}_{\mathrm{op}} \leq \norm[\big]{C^{-1}}_{\rm{op}} \norm[\big]{D^{-1}}_{\rm{op}} \norm[\big]{C-D}_{\rm{op}},
\end{equation}
where $\norm{\cdot}_{\text{op}}$ denotes the operator 2-norm.
Therefore this implies that there exists $\tilde\theta \in [\hat\theta - \hat\theta^1]$ and $\tilde\theta_2 \in [\hat\theta^2 - \hat\theta^{1,2}]$ such that 
\begin{meqn}
&
n \abs[\Big]{\calL(\tilde X, \hat\theta) - \calL(\tilde X,\hat\theta^1) - \bigl( \calL(\tilde X,\hat\theta^2) - \calL(\tilde X,\hat\theta^{1,2}) \bigr)}
\\
&\quad\le n\abs[\Big]{ \bigl(\hat\theta - \hat\theta^1\bigr)^\top \bigl(\partial_\theta \calL(X, \tilde \theta)\bigr)
    - \bigl(\hat\theta^2 - \hat\theta^{1, 2}\bigr)^\top \bigl(\partial_\theta \calL(X, \tilde\theta_2) \bigr)}
\\
&\quad\le  n\Bigl\lvert \norm[\big]{\hat\theta - \hat\theta^1 - [\hat\theta^2 - \hat\theta^{1, 2}]}_{L_2(v)} 
        \norm[\big]{\partial_\theta \calL(X, \tilde\theta_2)}_{L_2(v)}
\\
&\qquad+ \norm[\big]{\hat\theta - \hat\theta^1}_{L_2(v)}
    \norm[\big]{\partial_\theta \calL(X, \tilde\theta) - \partial_\theta \calL(X, \tilde\theta_2)}_{L_2(v)}\Bigr\rvert.
    \label{g_3}
\end{meqn}
Therefore using \cref{g_1,g_2,g_3} we obtain that:
\begin{meqn}
\epsilon_n(\nabla) &\le\frac{2}{\delta^2}\Big[\frac{1}{n}
    \norm[\Big]{\sup_{\theta \in \calV} \lambda_{max}\bigl(\partial^2_\theta \Psi(X_1, \theta)\bigr)}_{L_{16}}
    + \norm[\Big]{\lambda_{max} \bigl( \partial^2_\theta \Psi(X_1, \hat\theta) - \partial^2_\theta \Psi(X_1, \theta^*)\big)}_{L_{16}}\Big]
\\&\qquad \quad \times \norm[\Big]{\sup_{\theta \in \calV} \norm[\big]{\partial_\theta \calL(X, \theta)}_{L_2(v)}}^2_{L_{16}}
\\&\quad+ \frac{4}{n\delta^2}\norm[\Big]{\sup_{\theta \in \calV}\norm[\big]{\partial_\theta \Psi(X_1, \theta)}_{L_2(v)}}_{L_{12}}^2
    \norm[\Big]{\lambda_{max} \{\partial^2_\theta \calL(\tilde X, \theta^*)\} }_{L_{12}}.
\end{meqn}
Similarly one can prove that $\epsilon_n(R)\rightarrow 0$.

Note that by the consistency of $\hat\theta$ we have  $\Var\bigl(\calL(\tilde X, \hat\theta)\bigr)\rightarrow \Var\bigl(\calL(\tilde X, \theta^*)\bigr)$. Similarly using \cref{ien} we can prove that
\begin{equation}
    \rho =- \Cov\Bigl( \bigl(\partial_\theta R(\theta^*)\bigr)^\top \bigl(\EE\partial^2_\theta \Psi(X_1, \theta^*)\bigr)^{-1} \bigl(\partial_\theta \Psi(X_1, \theta^*)\bigr), \calL(X_1, \theta^*)\Bigr).
\end{equation}
Therefore the only quantity left to study is:
\begin{equation}
    \lim_n \max_{l_1\in N_n} l_1 \Var\bigl[R(\hat\theta_{l_1,n})\bigr].
\end{equation}
Let $(F_i)$ be the filtration such that $F_i=\sigma(X_1,\dots,X_i)$.
One can note that the following holds for all $l_1\in N_n$:
\begin{meqn}
 \Var\bigl[R(\hat\theta_{l_1,n})\bigr]
 &=\sum_{i\le l_1} \EE\Bigl[\Bigl(\EE\bigl[R(\hat\theta_{l_1,n}) \mid F_i\bigr] - \EE\bigl[R(\hat\theta_{l_1,n}) \mid F_{i-1}\bigr]\Bigr)^2\Bigr]
  \\&=\sum_{i\le l_1 } \EE\Bigl[\EE\bigl(\nabla_i R(\hat\theta_{l_1,n}) \mid F_i\bigr)^2\Bigr].
\end{meqn}
Using a telescopic sum argument we obtain that
\begin{meqn}
&
\abs[\Big]{ \EE\Bigl[\EE\bigl(\nabla_i R(\hat\theta_{l_1,n}) \mid F_i\bigr)^2\Bigr]
    - \EE\Bigl[\EE\bigl(\nabla_i R(\hat\theta_{l_1,n}) \mid X_i\bigr)^2\Bigr] }
\\&\le\sum_{j\le i} \abs[\Big] {\EE\Bigl[\EE\bigl(\nabla_i R(\hat\theta_{l_1,n}) \mid X_{\dbracket{j}\cup\{i\}}\bigr)^2\Bigr]
    -\EE\Bigl[\EE\bigl(\nabla_i R(\hat\theta_{l_1,n}) \mid X_{\dbracket{j-1}\cup\{i\}}\bigr)^2\Bigr] }
\\&\le 2\sum_{j<i} \abs[\Big]{ \EE\Bigl[\EE\bigl(\nabla_{i,j} R(\hat\theta_{l_1,n}) \mid X_{\dbracket{j}\cup\{i\}}\bigr)^2\Bigr] }
\\&\le \frac{2}{n^2}\epsilon_{\frac{n}{2}}(\hat{R})^2.
\end{meqn}
Therefore we deduce that:
\begin{equation}
   \lim_n \max_{l_1 \in N_n} \abs[\Big]{ l_1\Var\bigl[R(\hat\theta_{l_1,n})\bigr]
    - l_1^2\EE\Bigl[\bigl(\EE\bigl[R(\hat\theta_{l_1,n}) \mid X_1\bigr] - R_{l_1,n}\bigr)^2\Bigr] }\rightarrow 0.
\end{equation}

We now consider:
\begin{equation}
    \lim_{l_1\in N_n}l_1^2\EE\Bigl[\Bigl( \EE\bigl[R(\hat\theta_{l_1,n}) \mid X_1\bigr] - R_{l_1,n}\Bigr)^2\Bigr].
\end{equation}
By the mean value theorem we know that there is $\tilde\theta \in [\hat\theta_{l_1}, \hat\theta_{l_1}^1]$ such that the following holds: 
 \begin{meqn}
 &\EE\bigl[R(\hat\theta_{l_1}) \mid X_1\bigr] - R_{l_1,n}
  \\&= \EE\bigl[R(\hat\theta_{l_1}) - R(\hat\theta_{l_1}^1) \mid X_1\bigr]
  \\&= \EE\bigl[\bigl(\partial_\theta R(\tilde\theta)\bigr)^\top \bigl(\nabla_1 \hat\theta_{l_1,n}(X_{1:l_1})\bigr) \mid X_1\bigr]
  \\&= \EE\bigl[\bigl(\partial_\theta R( \theta^*)\bigr)^\top \bigl(\nabla_1 \hat\theta_{l_1,n}(X_{1:l_1})\bigr) \mid X_1\bigr]
  \\&\quad+ \EE\bigl[\bigl(\partial_\theta R(\tilde\theta) - \partial_\theta R(\theta^*)\bigr)^\top
        \bigl(\nabla_1 \hat\theta_{l_1,n}(X_{1:l_1}) \bigr) \mid X_1\bigr]
  \\&\oversetclap{(a)}{=} \bigl(\partial_\theta R(\theta^*)\bigr)^\top \EE\bigl[\nabla_1(\hat\theta_{l_1,n}(X_{1:l_1})) \mid X_1\bigr]+ o_p(\frac{1}{n}).
\end{meqn}
where to get (a) we used the $L_2$-consistency of $\hat\theta_{l_1,n}(X_{1:l_1})$ coupled with the fact that $n\EE\bigl[\nabla_1(\hat\theta_{l_1,n}) \mid X_1\bigr]=O(1)$.
Moreover by another application of the mean value theorem we know that there is $\tilde\theta \in  [\hat\theta_{l_1,n}, \hat\theta_{l_1,n}^1]$ such that the following holds:
 \begin{equation}\begin{split}
 & \EE\bigl[\nabla_1 \hat\theta_{l_1,n}(X_{1:l_1}) \mid X_1\bigr]
 \\&= \frac{1}{l_1}\EE\bigl[ \bigl(\frac{1}{l_1}\sum_{i\le l_1} \partial^2_\theta \Psi(X_i, \tilde\theta)\bigr)^{-1} 
        \bigl(\partial_\theta \Psi(X_1, \hat\theta_{l_1,n}) - \partial_\theta \Psi(X'_1, \hat\theta'_{l_1,n})\bigr) \mid X_1 \bigr]
  \\&\oversetclap{(a)}{=} \frac{1}{l_1} \bigl(\EE \partial^2_\theta \Psi(X_1, \theta^*)\bigr)^{-1}
        \EE\bigl[\partial_\theta \Psi(X_1, \theta^*) - \partial_\theta \Psi(X'_1, \theta^*) \mid X_1\bigr] + o_p(\frac{1}{n}).
\end{split}\end{equation}
where to obtain (a) we exploited once again the consistency of the estimator $\hat\theta_{l_1,n}(X_{1:l_1})$.
This implies that:
\begin{equation}
    \sigma_2^2 \eqdef G_R ^\top H^{-1} \Sigma H^{-1} G_R,
\end{equation}
where we have written:
\begin{equation}
G_R \eqdef \EE\bigl[\partial_\theta R(\theta^*)\bigr], \quad H \eqdef \EE\bigl[\partial^2_\theta \Psi(X_1, \theta^*)\bigr], \quad \Sigma \eqdef \Cov(\partial_\theta \Psi(X_1, \theta^*)).
\end{equation}
Therefore we get the desired result.
\end{proof}
\subsection{Proof of proposition~\ref{cabris}}
\label{sec:proof-of-cabris}
\begin{proof}
To prove the desired result we check that the conditions of Theorem 3 are respected. We denote
$\hat{\mu}_1(X) \eqdef \frac{2}{n}\sum_{i\le n}\mathbb{I}(Y_i=1)Z_i$ and
$\hat{\mu}_2(X) \eqdef \frac{2}{n}\sum_{i\le n}\mathbb{I}(Y_i=2)Z_i$ and
$\bar{Z} \eqdef \frac{1}{n}\sum_{i\le n}Z_i$.
Let $(\tilde Z,\tilde Y)$ be an independent copy of $(Z_1,Y_1)$.
Note that our classification rule is equivalent to classifying $\tilde X$ as belonging to the first class if:
\begin{equation}
\tilde Z \bigl(\hat{\mu}_1(X)-\hat{\mu}_2(X)\bigr) > \frac{1}{2}\bigl(\hat{\mu}_1(X)^2-\hat{\mu}_2(X)^2\bigr).
\end{equation}
Let us denote this classification region as:
\begin{equation}
\calR(X) \eqdef \Bigl\{z\in \RR: z\bigl(\hat{\mu}_1(X)-\hat{\mu}_2(X)\bigr) > \frac{1}{2}\bigl(\hat{\mu}_1(X)^2-\hat{\mu}_2(X)^2\bigr)\Bigr\}.
\end{equation}
We denote $X=o_p(Y)$ if $X/Y\xrightarrow{P}0$. If $\hat{\mu}_1(X)-\hat{\mu}_2(X)$ and  $\hat{\mu}_1(X^1)-\hat{\mu}_2(X^1)$ are both positive then the following holds:
\begin{equation}\label{eq:classification-stability}
\mathbb{I}(z\in \calR(X)) - \mathbb{I}(z\in \calR(X^1)) = \mathbb{I}(z> \bar{Z})-\mathbb{I}(z> \bar{Z}-\frac{1}{n}[Z_1-Z'_1]).
\end{equation} 
Let $M=\sup_x \Big(g_1(x),g_2(x)\Big)$.
We obtain that 
\begin{meqn}
&\norm[\Big]{\calL(\tilde X, f_n(X))-\calL(\tilde X, f_n(X^1))}_{L_1}
\\&\le  \PP\Bigl(\tilde Z \in \calR(X)\triangle \calR(X^1)\Bigr)
\\&\le  \PP\Bigl(\mathrm{sign}\bigl(\hat{\mu}_1(X)-\hat{\mu}_2(X)\bigr) =
    \mathrm{sign}(\hat{\mu}_1(X^1)-\hat{\mu}_2(X^1)), \tilde Z\in \calR(X) \triangle \calR(X^1)\Bigr)
\\& \quad+ \mathbb{P}\Bigl(\mathrm{sign}\bigl(\hat{\mu}_1(X)-\hat{\mu}_2(X)\bigr)\ne\mathrm{sign}\bigl(\hat{\mu}_1(X^1)-\hat{\mu}_2(X^1)\bigr)\Bigr)
\\&\le \sup_t \EE\Bigl[\int_{t}^{t+\frac{X_1-X'_1}{n}}\frac{g_1(x)+g_2(x)}{2} \, dx\Bigr]
\\&\quad+ \PP\Bigl(\mathrm{sign}\bigl(\hat{\mu}_1(X)-\hat{\mu}_2(X)\bigr)
                   \neq \mathrm{sign}\bigl(\hat{\mu}_1(X^1)-\hat{\mu}_2(X^1)\bigr)\Bigr)
\\&\le M\EE\Big[\frac{\abs{X_1-X'_1}}{n}\Bigr] 
+ \PP\Bigl(\mathrm{sign}\bigl(\hat{\mu}_1(X)-\hat{\mu}_2(X)\bigr) \neq \mathrm{sign}\bigl(\hat{\mu}_1(X^1)-\hat{\mu}_2(X^1)\bigr)\Big)
\\&\le\frac{2\mu M}{n}
+ \PP\bigl(\abs{\hat{\mu}_1(X)-\mu_1} \ge \abs{\mu_1-\mu_2} / 2\bigr)
    + \PP\bigl(\abs{\hat{\mu}_2(X)-\mu_2} \ge \abs{\mu_1-\mu_2} / 2\bigr)
\\&\oversetclap{(b)}{\le}\frac{2\mu M}{n}
+\frac{32\max_{i\in \{1,2\}} \EE_{g_i}\abs{Z}^3}{\abs{\mu_1-\mu_2}^3 n^{2}},
\end{meqn}
where (b) is due to Chebyshev inequality.
Therefore we have that:
\begin{equation}
n\norm[\Big]{\calL(\tilde Z, f_n(X)) - \calL(\tilde Z, f_n(X^1))}_{L_1}=O(1),
\end{equation}
which implies that $\epsilon_n(\Delta)=o(1)$ and $\calS(R)=O(1)$. 

We can note that conditionally on the signs of $\hat{\mu}_1(X)-\hat{\mu}_2(X)$ and  $\hat{\mu}_1(X^1)-\hat{\mu}_2(X^1)$, \cref{eq:classification-stability} does not depend on $X_2$.
Therefore we have that:
\begin{meqn}
&\norm[\Big]{\calL(\tilde Z, f_n(X))-\calL(\tilde Z, f_n(X^1)) -
    \bigl[\calL(\tilde Z, f_n(X^1))-\calL(\tilde Z, f_n(X^{1,2}))\bigr]}_{L_1}
\\&\le\PP\Bigl(\mathrm{sign}\bigl(\hat{\mu}_1(X)-\hat{\mu}_2(X)\bigr)\ne \mathrm{sign}\bigl(\hat{\mu}_1(X^1)-\hat{\mu}_2(X^1)\bigr)\Bigr)\\
&\quad+\PP\Bigl(\mathrm{sign}\bigl(\hat{\mu}_1(X^2)-\hat{\mu}_2(X^2)\bigr)\ne \mathrm{sign}\bigl(\hat{\mu}_1(X^{1,2})-\hat{\mu}_2(X^{1,2})\bigr)\Bigr)
\\&\le 
\frac{64}{\abs{\mu_1-\mu_2}^3 n^{2}} \max_{i\in \{1,2\}} \EE_{g_i}\abs{Z}^3,
\end{meqn}
which implies that $\epsilon_n(\nabla)=o(1)$, $\epsilon_n(d)=o(1)$ and $\epsilon_n(R)=o(1)$.
Finally note that that if $\hat{\mu}_1(Z)- \hat{\mu}_2(Z) $ and $\hat{\mu}_1(Z^1)- \hat{\mu}_2(Z^1) $ both have the same sign as $\mu_1-\mu_2$, then the following holds:
\begin{meqn}
&R(f_n(X))-R(f_n(X^1))
\\&=\frac{1}{2} \Big[\mathbb{P}\big(\tilde Z\in \mathcal{R}(Z) \mid Z,C_n(\tilde Z)=2\big)
    -\mathbb{P}\big(\tilde Z\in \mathcal{R}(Z^1) \mid Z^1, C_n(\tilde Z)=2\big)\Big]
\\&\quad+\frac{1}{2}\Big[ \mathbb{P}\big(\tilde Z\not\in \mathcal{R}(Z) \mid Z,C_n(\tilde Z)=1\big)
    -\mathbb{P}\big(\tilde Z\not\in \mathcal{R}(Z^1) \mid Z^1,C_n(\tilde Z)=1\big)\Big]
\\&=\frac{1}{2}\Big[F_1(\bar{Z})-F_1\big(\bar{Z}+\frac{Z_1-Z'_1}{ n}\big)-\big[F_2(\bar{Z})-F_2\big(\bar{Z}+\frac{Z_1-Z'_1}{ n}\big)\big]\Big]
\\&\oversetclap{(a)}{=}\frac{Z_1-Z'_1}{2 n}\Big[g_1(\bar{Z})-g_2(\bar{Z})\Big]+o_p(\frac{1}{n})
\\&\oversetclap{(b)}{=}\frac{Z_1-Z'_1}{2 n}\Big[g_1(\mu)-g_2(\mu)\Big]+o_p(\frac{1}{n}).
\end{meqn}
where (a) and (b) are consequence of a Taylor expansion.
Therefore for all $l_n\in  N_n$ we have:
\begin{equation}
l_n\EE\bigl[\beta_1^{l_n,n}(X) \mid Z_1\bigr] = \frac{Z_1-\EE(Z_1)}{2 }\Delta+o_p(1),
\end{equation}
where we have written $\Delta \eqdef g_1(\mu) - g_2(\mu)$.
This implies that  for all $l_n\in N_n$ we have:
\begin{meqn}
l^2_n\Var\Bigl[\EE\bigl[\beta_1^{l_n,n}(X) \mid Z_1 \bigr]\Bigr]
&= \frac{\Var(Z_1)}{4}\Delta^2+o(1) \\
&=\frac{\Delta^2}{8}\Bigl(\Var_{g_1}(Z) + \Var_{g_2}(Z) + \frac{1}{2}(\mu_1-\mu_2)^2\Bigr).
\end{meqn}
Moreover we can observe that $\norm[\big]{\bar\calL_{l_1,n}(X_1)-M(Z_1,Y_1)}_{L_2}\rightarrow 0$ where
\begin{equation}
M(z,y)=y\mathbb{I}\big(z\le \mu\big)+(1-y)\mathbb{I}\big(z>\mu\big).
\end{equation}
We remark that $M(Z_1,Y_1)$ is a Bernoulli random variable with parameter $q \eqdef F_1(\mu)+ 1-F_{2}(\mu)$.
This implies that for all $l_n\in N_n$ we have:
\begin{equation}
    \Var\bigl[\bar\calL_{l_2,n}(X_1)\bigr] = q(1-q) + o(1),
\end{equation}
From which we may deduce that, for all $l_1,l_2\in N_n$ we have
\begin{meqn}
&l_1\Cov\Bigl[\EE\bigl[\beta_1^{l_1,n}(X) \mid X_1\bigr], \bar\calL_{l_2,n}(X_1)\Bigr]
\\&=\frac{\Delta}{2}\Cov\Bigl[\EE\bigl[X_1, \bar\calL_{l_2,n}(X_1)\Bigr]+o(1)
\\&=\frac{\Delta}{4}\Big[\EE_{g_2}\Big(X\mathbb{I}(X>\mu)\Big) + \EE_{g_1}\Big( X\mathbb{I}(X\le\mu)\Big)-2\beta\mu\Big]+o(1).
\end{meqn}
\end{proof}

\subsection{Proof of proposition~\ref{luna}}
\label{sec:proof_nearest_neighbour}

\begin{proof}
Before diving into the proof, recall that $c(z, Z_{1:n}) = \argmin_{i \leq n} \abs{Z_i - z}$, and define for all subsets $B\subset \mathbb{N}$
\begin{align*}
    c^*_1(B) &\eqdef c(\frac{1}{2},Z_{B}), \\
    c^*_2(B) &\eqdef \begin{cases}
        \displaystyle\argmin_{i\in B \text{ s.t } Z_i\le \frac{1}{2}} |\frac{1}{2}-Z_i| &\text{if }Z_{c^*_1(B)}\ge \frac{1}{2},\\
        \displaystyle\argmin_{i\in B \text{ s.t } Z_i> \frac{1}{2}} |\frac{1}{2}-Z_i| &\text{if } Z_{c^*_1(B)}\le \frac{1}{2}.
    \end{cases}.
\end{align*}

A new observation $(\tilde Z, \tilde Y)$  will be misclassified if $Y_{c(\tilde Z, Z_{1:n})}$ is different from $\tilde Y$.
Therefore it is mislabeled if it falls in the following set:
\begin{equation}
 \mathcal{E}\Bigl(X_{c_1^*(\dbracket n),c_2^*(\dbracket n)}\Bigr) \eqdef \begin{cases}
    \bigl[\frac{1}{2}, \frac{1}{2}\bigl(Z_{c_1^*(\dbracket n)}+Z_{c_2^*(\dbracket n)}\bigr)\bigr] & \text{if } Z_{c_1^*(\dbracket n)}\le \frac{1}{2},\\
    \bigl[\frac{1}{2}\bigl(Z_{c_1^*(\dbracket n)}+Z_{c_2^*(\dbracket n)}\bigr),\frac{1}{2} \bigr] & \text{otherwise.}
    \end{cases}
\end{equation}

The key point is to note that the 2-fold cross-validated error can be rewritten in the following way:
\begin{meqn}
\sqrt{n} \hat{R}_{\rm{cv}}(X) &=
\abs[\Big]{\Bigl\{i\le \frac{n}{2} : Z_i\in \mathcal{E}\Big(X_{c_1^*({\lfloor\frac{n}{2}\rfloor+1:n}),c_2^*({\lfloor\frac{n}{2}\rfloor+1:n})}\Big) \Bigr\}}\\
&\quad+ \abs[\Big]{\Bigl\{i> \frac{n}{2} : X_i\in\mathcal{E}\Big(X_{{X_{c^*_1(\dbracket{\lfloor\frac{n}{2}\rfloor}) ,c^*_2(\dbracket{\lfloor\frac{n}{2}\rfloor})}}) }\Big) \Big\}}.
\end{meqn}

For a given pair of random variables $X'_1,X'_2$, we write:
\begin{equation}
    \lambda_{X'_{1:2}}\eqdef \frac{n}{2}P\bigl(\tilde X_1\in \mathcal{E}(X'_{1:2}) \mid X'_{1:2}\bigr),
\end{equation}
and shorthand
$B_1 = \{1, \dotsc, \lfloor n / 2 \rfloor \}$ and $B_2 = \{ \lfloor n / 2 \rfloor + 1, \dotsc, n \}$.
By definition, we have that:
\begin{equation}
\Big(\mathbb{I}\Big(Z_i\in\mathcal{E}(X_{c_1^*(B_2),c_2^*(B_2)})\Big)\Big)_{i \in B_1} \mid X_{c^*_1(B_2),c^*_2(B_2)}
\overset{i.i.d}{\sim}
\mathrm{Bernoulli}\Bigl(\lambda_{X_{c^*_1(B_2)},c^*_2(B_2)}\Bigr).
\end{equation}
Therefore using the classical Poisson limit theorem, we have that:
\begin{equation}
\abs[\Big]{\Bigl\{i \in B_1 : Z_i\in\mathcal{E}(X_{c^*(2)} ) \Bigr\}} \Bigm| X_{B_2} \xrightarrow{d} \mathrm{Poisson}\Bigl(\lambda_{X_{c^*_1(B_2)},c^*_2(B_2)}\Bigr),
\end{equation}
which leads to:
\begin{equation}
\sqrt{n}\hat{R}_{\mathrm{split}} \mid X_{B_2}\xrightarrow{d}\mathrm{Poisson}\Big(\lambda_{X_{c^*_1(B_2),c^*_2(B_2)}}\Big).
\end{equation}
Moreover exploiting the fact that $(X_i)\overset{i.i.d}{\sim}\rm{unif}([0,1])$ we can see that: 
\begin{equation*}
    \begin{split}
        \lambda_{X_{c^*_1(B_2),c^*_2(B_2)}}=\frac{n}{4}\Big|1-\Big[Z_{c_1^*(B_2)}+Z_{c_2^*(B_2)}\Big]\Big| \xrightarrow{d} \exp(1/2)
    \end{split}
\end{equation*}
Furthermore, by definition of the indexes $c^*_1(\cdot)$ and $c^*_2(\cdot)$, we remark that:
\begin{align}
Z_{B_1 \setminus\{ c_1^*(B_1),c_2^*(B_1)\}} \mid X_{c_1^*(B_1),c_2^*(B_1)} &\overset{i.i.d}{\sim}
\begin{cases}
    \mathrm{Unif}\bigl(\bigl[0,1\bigr]\setminus\big[Z_{c^*_1(B_1)},Z_{c^*_2(B_1)}\big]\bigr) &\text{if } Z_{c^*_1(B_1)}\le \frac{1}{2} \\
    \mathrm{Unif}\bigl(\bigl[0,1\bigr]\setminus\big[Z_{c^*_2(B_1)},Z_{c^*_1(B_1)}\big]\bigr) &\text{otherwise}.
\end{cases}, \\
Z_{B_2 \setminus\{ c_1^*(B_2),c_2^*(B_2)\}} \mid X_{c_1^*(B_2),c_2^*(B_2)} &\overset{i.i.d}{\sim}
\begin{cases}
    \mathrm{Unif}\bigl(\bigl[0,1\bigr]\setminus\big[Z_{c^*_1(B_2)},Z_{c^*_2(B_2)}\big]\bigr) &\text{if } Z_{c^*_1(B_2)}\le \frac{1}{2} \\
    \mathrm{Unif}\bigl(\bigl[0,1\bigr]\setminus\big[Z_{c^*_2(B_2)},Z_{c^*_1(B_2}\big]\bigr) &\text{otherwise.}
\end{cases}.
\end{align}
Therefore the misclassified observations will be those falling in:
\begin{meqn}
\calM &= \Bigl(\calE(X_{c_1^*(B_1),c_2^*(B_1)}) \bigcap \big[X_{c_1^*(B_2)},X_{c_2^*(B_2)}\big]^c\Bigr) \\
&\quad\bigcup\Bigl(\calE(X_{c_1^*(B_2),c_2^*(B_2)}) \bigcap \big[X_{c_1^*(B_1)},X_{c_2^*(B_1)}\big]^c\Bigr).
\end{meqn}

We denote:
\begin{align*}
    s_1 &\eqdef Y_{c_1^*(B_1)}-(1-Y_{c_1^*(B_1)}), \\
    s_2 &\eqdef Y_{c_2^*(B_2)}-(1-Y_{c_2^*(B_2)}),
\end{align*}
and note that if $s_1=-1$ we have $Z_{c_1^*(B_1)} \le 0.5$.
We now have that:
\begin{meqn}
    \abs{\calM}
    &=\mathbb{I}(s_1 s_2=1)\Bigl[\Bigl(\frac{s_1}{2}Z_{c_2^*(B_1)} + \frac{s_1}{2} Z_{c_1^*(B_1)} - s_1 Z_{c_2^*(B_2)} \Bigr)_+
    \\ &\phantom{\mathbb{I}(s_1 s_2 =1)}
        \qquad+\Bigl(\frac{s_2}{2}Z_{c_2^*(B_2)} + \frac{s_2}{2} Z_{c_1^*(B_2)} - s_2 Z_{c_2^*(B_1)} \Bigr)_+\Bigr]
    \\&\quad+ \mathbb{I}(s_1s_2\ne1)
        \Bigl[\Bigl(\frac{s_1}{2}Z_{c_2^*(B_1)} + \frac{s_1}{2}Z_{c_1^*(B_1)} + s_2 Z_{c_1^*(B_2)} \Bigr)_{+}
    \\&\phantom{\mathbb{I}(s_1 s_2\neq 1)}
    \qquad+\Bigl(\frac{s_2}{2} Z_{c_2^*(B_2)} + \frac{s_2}{2}Z_{c_1^*(B_2)} + s_1 Z_{c_2^*(B_1)} \Bigr)_+\Bigr].
\end{meqn}
We note that $s_1,s_2\overset{i.i.d}{\sim}\mathrm{Unif}(\{-1,1\}),$ and
writing $U_1,U_2\overset{\rm{i.i.d}}{\sim} \mathrm{Exp}(1)$, and $N_1,N_2\overset{\rm{i.i.d}}{\sim} \mathrm{Exp}(\frac{1}{2})$, then we have:
\begin{equation*}\begin{split}&
    \sqrt{n}\hat{R}_{cv}\xrightarrow{d} \mathbb{I}(s_1s_2=-1)\Bigl[\mathbb{I}(r^d_1\ge 0)\bigl[1+\mathrm{Poisson}(r^d_1)\bigr]+\mathbb{I}(r^d_2\ge 0)\bigl[1+\mathrm{Poisson}(r^d_2)\bigr] \Bigr]
    \\&\qquad \qquad+\mathbb{I}(s_1s_2=1)\Bigl[\mathbb{I}(r^s_1\ge 0)\bigl[1+\mathrm{Poisson}(r^s_1)\bigr]+\mathbb{I}(r^s_2\ge 0)\bigl[1+\mathrm{Poisson}(r^s_2)\bigr] \Bigr]
    \\&   \sqrt{n}\hat{R}_{\rm{split}}\xrightarrow{d} \mathrm{Poisson}(N_1)
\end{split}\end{equation*}
 where we have
 \begin{gather*}
    r^d_1 \eqdef s_1 N_1 + s_2 U_2, \quad
    r^d_2 \eqdef s_2 N_2 + s_1 U_1, \\
    r^s_1 \eqdef r^d_1 - 2 s_2 N_2, \quad
    r^s_2 \eqdef r^d_2 - 2 s_1 N_1.
 \end{gather*}
\end{proof}

\section{Acknowledgements}
We would like to deeply thank Peter Orbanz for helpful discussions and remarks; as well as for providing a supportive collaboration environment.

\printbibliography[heading=bibintoc]

\end{document}